\documentclass[10pt]{amsart}
\usepackage{amsfonts}
\usepackage{amssymb}
\usepackage{amsmath}
\usepackage[arrow,matrix,curve]{xy}

\newcommand{\C}{{\mathbb C}}
\newcommand{\R}{{\mathbb R}}
\newcommand{\Z}{{\mathbb Z}}
\newcommand{\N}{{\mathbb N}}
\newcommand{\Q}{{\mathbb Q}}

\newcommand{\A}{{\mathbb A}}
\newcommand{\T}{{\mathbb T}}
\newcommand{\F}{{\mathbb F}}

\newcommand{\G}{{\mathbb G}}

\newcommand{\K}{{\mathbb K}}

\renewcommand{\H}{{\mathbb H}}

\newtheorem{theorem}{Theorem}[section]
\newtheorem{lemma}[theorem]{Lemma}
\newtheorem{remark}[theorem]{Remark}
\newtheorem{example}[theorem]{Example}
\newtheorem{corollary}[theorem]{Corollary}
\newtheorem{proposition}[theorem]{Proposition}
\newtheorem{definition}[theorem]{Definition}

\def\cal{\mathcal}

\newcommand{\calh}[0]{{\cal H}}
\newcommand{\calr}[0]{{\cal R}}
\newcommand{\calf}[0]{{\cal F}}
\newcommand{\call}[0]{{\cal L}}

\newcommand{\cale}[0]{{\cal E}}
\newcommand{\calk}[0]{{\cal K}}

\newcommand{\cals}[0]{{\cal S}}
\newcommand{\calm}[0]{{\cal M}}

\newcommand{\calz}[0]{{\cal Z}}

\newcommand{\functionals}[2]{{\Theta_{#1}({#2})}}
\newcommand{\moduleright}[1]{{{\rm MOD}_#1}}
\newcommand{\moduleleft}[1]{{{\rm MOD}_#1}}

\newcommand{\modulehomomorphisms}[3]{{{\rm Hom}_{#1}(#2,#3)}}
\newcommand{\moduleendomorphisms}[2]{{{\rm End}_{#1}({#2})}}
\newcommand{\Ad}[0]{{\rm Ad}}
\newcommand{\id}[0]{{\rm id}}
\newcommand{\compacts}[2]{{\calk_{#1} (#2)}}
\newcommand{\idd}[0]{{1}}

\newcommand{\triv}[0]{{\rm 1}}   

\newcommand{\ind}[2]{{{\rm Ind }_{#1}^{#2}}}
\newcommand{\res}[2]{{{\rm Res }_{#1}^{#2}}}
\newcommand{\Aut}[0]{{\rm Aut}}
\newcommand{\Hom}[1]{{{\rm Hom }_{#1}}}
\newcommand{\ring}[1]{{{ R }^{#1}}}
\newcommand{\image}[1]{{{\rm im \;} {#1} }}
\newcommand{\diag}[0]{{\rm diag}} 
\renewcommand{\span}[0]{{\rm span}} 
\renewcommand{\Hom}[1]{{{\rm Hom}_{#1}}} 
\newcommand{\AddMaps}[0]{{\rm AddMaps}} 
\newcommand{\character}[0]{{\rm char}} 
\newcommand{\spec}[0]{{{\rm sp}}}

\title{A kind of $KK$-theory for rings} 

\author[B. Burgstaller]{Bernhard Burgstaller}

\email{bernhardburgstaller@yahoo.de}


\begin{document}

\begin{abstract}
A group equivariant $KK$-theory 
for rings 
will be defined and studied 
in analogy to Kasparov's $KK$-theory for 
$C^*$-algebras. 
It is a kind of linearization of the category 
of rings by allowing addition of homomorphisms, imposing also homotopy invariance, invertibility of matrix corner embeddings, 
and allowing morphisms which are the opposite split of split exact sequences. 
We demonstrate the potential of this theory 
by proving for example equivalence induced by  Morita equivalence 
and a Green-Julg isomorphism in this framework. 
  
\end{abstract}

\maketitle

\tableofcontents



\section{Introduction}

Let $G$ be a discrete group. 
In this paper we consider 
a 
variant of $KK$-theory for the class of 
$G$-equivariant, quadratik rings which
is closely related to Kasparov's original $KK$- theory for $C^*$-algebras \cite{kasparov1981,kasparov1988}. 
This theory defined in this paper, called
$GK^G$-theory, is the universal additive, homotopy invariant, stable and split exact  category 
formed from the category of quadratik rings 
and ring homomorphisms. 

It has it roots in Cuntz \cite{cuntzn} and Higson's \cite{higson} findings, that Kasparov's $KK$-theory, when restricted to ungraded separable $C^*$-algebra, is the universal 
additive, homotopy invariant, split-exact theory formed from the category of separable 
$C^*$-algebras and $*$-homomorphisms. 
See 
\cite{bgenerators,baspects} for more on this link. 

A ring has the 
condition to be quadratik 
if every element of it can be written as 
a finite sum of products $ab$ for $a,b$ in the ring. 

Now in this paper we do homotopy as follows.
We at first complexify the two rings 
at the endpoints, and then do continuous-function homotopy in $\C$-algebras as usual. 
This may not be viewed as a proper homotopy in rings, as the complexified space is bigger, 
but it is a convenient equivalence relation suitable for our purposes. 

Now stability means the following.
We at first consider modules $\cale$ over a ring $A$ which are equipped with a functional space 
comprised
of $A$-module homomorphisms $\phi:\cale \rightarrow A$. Such functionals are the analogy of the functionals $\cale \rightarrow A: \xi \mapsto \langle \eta,\xi\rangle$ induced by the inner product on Hilbert modules $\cale$ over a $C^*$-algebra $A$. 
Then we define the ring of compact operators $\calk_A(\cale)$ generated by elementary operators 
$\theta_{\eta,\phi}$ defined by $\theta_{\xi,\phi}(\xi)= \eta \phi(\xi)$ in analogy to such operators in $KK$-theory. 
Then a corner embedding is the canonical 
ring homomorphism 
$$e:A \rightarrow \calk_A(\cale \oplus A)$$
acting by multiplication on the coordinate $A$, 
and we declare $e$ to be an invertible morphism in $GK^G$-theory. 
 This means ``stability".

If we restrict the class of quadratik rings 
to the separable $C^*$-algebras, define the 
functional spaces by the inner product as explained before as usual, restrict $\cale$ to countably generated Hilbert $A$-modules, 
and define $\calk_A(\cale)$ as before but also taking the norm closure, then we exactly get 
Kasparov's theory $KK^G$ for $GK^G$ 
by remark \ref{rem12} and \cite{bgenerators} out. 

This theory is also split-exact, meaning that given a short split-exact sequence of rings, splitting at the quotient ring, the opposite split on the ideal-side is declared to be a valid morphism. It is automatically always an additive group homomorphism, but not a ring homomorphism, and so this theory 
gets the glance of a kind of linearization of the ring category. 

The main concerns of this concept are two-folds. 
Firstly, general rings, even 
general Banach algebras, do not behave so well as $C^*$-algebras. More precisely, they do not have all the good analysis as $C^*$-algebras do have, like the Kasparov stabilization theorem, 
approximate units, positivity, to name a view. 

A sort of 
compensate of this is what we do in 
section \ref{sec10}, loosely called $\nabla$-calculation. 
An arbitrary complex algebraic expression which is zero in $GK^G$-theory is equivalently reformulated to an expression which is zero on 
the level (subcategory) of ring homomorphism in $GK^G$-theory by a simple algorithm. 
The problem that remains is the complicated equivalence relations on this level coming from $GK^G$-theory. 
Nevertheless, with this method we are able to show that a functor on $GK^G$-theory is faithful if and only if it is faithful on the sublevel of ring homomorphism. 
 See theorem \ref{theorem108}.

But the second issue is more subtle and appears to us also more severe. 
If one wants to construct a morphism in $GK^G$-theory involving one synthetical split morphism $\Delta$ axiomatically postulated by the theory, 
then one `automatically' runs into double-split exact sequences. For example the composition
$\pi \Delta$ of $\Delta$ with a ring homomorphism $\pi$ becomes $\pi' \Delta_{\pi''}$ 
with respect to a double split-exact sequence 
by lemmas \ref{lemma101} and \ref{lemma215}. 
But in practice it is extremely difficult to write down two splitting ring homomorphisms 
$\pi',\pi''$ 
of a short exact sequence, such that $\pi' \Delta_{\pi''} \neq 0$. 
This difficulty besets also Kasparov's theory 
  in that it 
might be narrow in applicability. 
Mostly such double splits are constructed 
by Clifford algebras, the $K$-theory Bott
(or dual-Dirac) element, and the Dirac element induced
 by the Dirac operator 
\cite{kasparov1988}.

The next 
goal of $GK^G$-theory should be clarifying applicability by coming up with examples involving synthetical splits. 
Of course, one might consider similar examples 
known from 
 Atiyah-Singer index theory and Kasparov theory 
by replacing continuous function 
spaces $C_0$  by differentiable function spaces $C^\infty$. 
This might be a possible application. 
Connes' consideration of $p$-summable operators 
in noncommutative differential geoemetry \cite{connesn} seems to be such a continuous by 
smooth functions `replacement', very roughly said, and 
actually such replacements appear 
in various contexts 
 in the literature. 

But for the moment we show some standard results, as roughly speaking, Morita 
equivalence induces $GK^G$-equivalence 
in section \ref{sec14},
 the establishment of a descent functor 
in section \ref{sec12},
and an induction functor for 
subgroups $H$
of $G$ in section \ref{sec15}, and even a Green-Julg isomorphism in section \ref{sec16}. 
Also the Baum-Connes map might be verbatim generalized to $\C$-algebras, 
and we check that it is injective if and only 
if it is injective on the `pure' homomorphism 
level, 
by the functor theorem mentioned earlier,	
see section \ref{sec17}. 


There are other generalizations of $KK$-theory 
to bigger or other classes than $C^*$-algebras. 
Notably, this is Lafforgue's Banach $KK$-theory 
for Banach algebras \cite{lafforgue}, and Cuntz' $kk$-theory for 
locally convex algebras and diffeomorphic 
homotopy \cite{cuntz}. 
Cuntz' theory has been considered and extended by various authors, 
see Cuntz \cite{cuntz2}, Cuntz and Thom 
\cite{cuntzthom}, {Corti\~nas} 
and Thom \cite{cortinasthom}, 
Ellis 
\cite{ellis}, 
and 
Garkusha \cite{garkusha1,garkusha2}.
Related is also Grensing \cite{grensing}, 
and 
we should also mention 
Weidner 
\cite{weidner,weidner2}.

The theory by Cuntz is almost half-exact, meaning that short exact sequences of 
algebras  
having linear splits induce, even, long exact and cyclic sequences in $kk$-theory, and so might be even better 
compared with, the half-exact, $E$-theory by Higson \cite{higson2},
and Connes and Higson \cite{conneshigson,
conneshigson2}, which was designed to make, the split-exact, $KK$-theory half-exact.  
Notice, that our theory here is also split-exact by its very definition,  
but extremely likely not half-exact.


We remark that all results presented in this paper hold analogously in $KK^G$-theory for $C^*$-algebras. 

In particular, the $\nabla$-calculation 
might also be  interesting in $KK$-theory, 
as it shows that an arbitrarily long product 
of Kasparov products is zero in $KK^G$-theory 
if and only if an obvious ordinary Kasparov element is zero in $KK^G(A,B)$ without the need of the Kasparov product at all, 
see for instance example \ref{example2}. 
However, this is theoretical, because in practice, this element is so complicated, that 
we need to compute the element classically 
by computing the Kasparov products
to decide the equality with zero by a
homotopy, 
which on the other hand is also theoretical, 
because in general 
one needs the axiom of choice 
for the Kasparov product. 

We note that we have chosen to do everything 
for discrete rings here 
not because 
we think 
this is the best choice or even necessary, 
but rather to have it easier at first and to make a point.
First of all, it works for discrete rings. 
But, the theory is also understood to be adapted to 
various situations, and one has to 
choose appropriate closures in norm, 
or locally convex spaces, or Schwartz spaces 
and so on, and choose appropriate functional spaces, and also allow $G$ to be 
a locally compact group.  
Notably, the 
corner embeddings have to be adapted by choosing the appropriate topological closure,
as, notice, the stability axiom is 
mainly  the 
axiom of $GK^G$-theory sensitive to the differences, 
but 
the homotopy axiom might also be adjusted, 
for example by allowing only smooth functions 
rather than continuous ones. 

We also are confident that everything, 
excepting the Baum-Connes map, may  
easily be 
adapted to the inverse semigroup and 
not necessarily Hausdorff locally compact 
groupoid equivariant setting as 
in \cite{baspects}. 
But to avoid overloading very moderate technicalities, 
we abstained from such a general situation 
the first time.  

The brief overview of this paper is as follows. 
In section \ref{sec2}-\ref{sec3} 
we explain $GK$-theory. 
The way one works with double split exact 
sequences in $GK$-theory is explained in
sections \ref{sec4}-\ref{sec9}. 
In section \ref{sec10} the $\nabla$-calculation 
is shown. 
In the final sections \ref{sec11}-\ref{sec17} 
the above mentioned standard results 
will be 
performed.   

\section{Rings and Functional Modules}
					\label{sec2}

All rings in this paper
are neither necessarily commutative nor
unital, associative discrete rings.  
Throughout, $G$ denotes a discrete group, 
and all rings $A$ are equipped with a $G$-action, that is, a group homomorphism 
$\alpha:G \rightarrow \Aut(A)$ 
into the automorphism group $\Aut(A)$ of $A$. 
All ring homomorphisms are $G$-equivariant. 
We sometimes say ``non-equivariant" ring, 
homomorphism, module etc. if we want to 
ignore any 
possible $G$-structure. 
We often write $\id$ or $\idd$ for the identity map, for example in $T \otimes \idd$.
Likewise we write $\triv$ for the trivial 
$G$-action. The unitization of a ring $A$ 
will be denoted by $A^+$ or $\tilde A$. 

Given an invertible operator $U$ 
or a $G$-action $U$,
we write $\Ad(U)$ for the map $T \mapsto 
U \circ T \circ U^{-1}$ and 
the $G$-action 
$(g,T)\mapsto  U_g \circ T \circ U_{g^{-1}}$, respectively.

All structures 
have $G$-action. 
As soon as we introduce the definition of a $G$-structure on a 
module, operator spaces etc. it is understood 
without saying that these objects carry $G$-actions. 

Functionals on modules are the substitute 
for inner products on Hilbert modules. 
Functional spaces are just 
auxiliary spaces 
to define the space of compact operators. 
Everything runs in parallel to $C^*$-algebras (now rings),
Hilbert modules (now functional modules) and inner products (now functionals).

\begin{definition}[$G$-action on module]
{\rm

Let $(A,\alpha)$ be a ring. 
Let 
$M$ be a right $A$-module. 

A {\em $G$-action} $S$ on $M$ is a  
group homomorphism 
$S:G \rightarrow \AddMaps (M)$ (additive maps 
= abelian group homomorphisms) 
into the invertible additive maps on $M$ such that
$S_1 = \idd$ 
and 
$S(\xi a)= S(\xi) \alpha(a)$ for all $\xi \in \cale, a \in A$. 

}
\end{definition}

\begin{definition}[$G$-action on Module Homomorphisms]
{\rm

Let $(A,\alpha)$ 
be a ring. 
Let $(M,S)$ and $(N,T)$ be right $A$-modules. 
We define $\Hom A (M,N)$ to be the abelian group of all right $A$-module homomorphisms 
$\phi:M \rightarrow N$.

We equip $\Hom A (M,N)$ with the $G$-action 
$\Ad (S,T)$, that is, we set 
$g(\phi):= T_{g} \circ \phi \circ S_{g^{-1}} \in \Hom A (M,N)$ for all $g \in G$ and $\phi \in \Hom A (M,N)$. 

}
\end{definition}

We also write $\Hom A (M):=\Hom A (M,N)$ if $(M,S)=(N,T)$. 
This is a ring under concatenation and its 
$G$-action $\Ad(S)$.

\begin{definition}[Functionals]			\label{def12}    
{\rm

Let $(A,\alpha)$ be a ring. 
Let $M$ be a right $A$-module. 
Turn $\Hom A (M,A)$ to a left $A$-module 
by setting 
$(a \phi)(\xi):= a \phi(\xi)$  for all 
$a \in A, \xi \in M$ and $\phi \in \Hom A (M,A)$. 

Assume 
we are given a 
distinguished {\em functional space} $\Theta_A(M) \subseteq \Hom A (M,A)$ which is a $G$-invariant, left $A$-submodule of $\Hom A (M,A)$. 

Then we call $(M,\Theta_A(M))$ a 
{\em right functional $A$-module. 
}

}
\end{definition}

The functional space $\Theta_A(M)$ will usually not be notated in $M$, as it is called anyway always in the same way. 

\begin{definition}[Compact operators] 
		\label{defcompact}
{\rm

Let $(A,\alpha)$ be a ring. 
To right functional $A$-modules $(M,S)$ and $(N,T)$ is 
associated 
the $G$-invariant, abelian subgroup 
(under addition) of {\em compact operators} $\calk_A (M,N) \subseteq \Hom A (M,N)$
which consists of all 
finite sums  of  all {\em elementary compact operators}  
$\theta_{\eta,\phi} \in \Hom A (M,N)$ defined by
$\theta_{\eta,\phi}(\xi)= \eta \phi(\xi)$ for $\xi \in M, \eta \in N$ and $\phi \in \Theta_A(M)$.
}
\end{definition}

We write $\calk_A(M) \subseteq \Hom A (M)$ for the $G$-invariant subring $\calk_A(M,M)$. 
To observe $G$-invariance, we compute
$g(\theta_{\xi,\phi})(\eta)= T_g ( \xi \phi (S_{g^{-1}}(\eta)) = \theta_{T_g(\xi),g(\phi)}(\eta)$.

\begin{definition}[Multiplier operators] 
{\rm

Let $M$ and $N$ be right functional $A$-modules.

We write $\calm(\calk_A(M,N)) \subseteq \Hom A (M,N)$ 
for the $G$-invariant, abelian subgroup of {\em multipliers of the compact operators}, that is, we set 
\begin{eqnarray*}
&&\calm(\calk_A(M,N)) := \{ V \in \Hom A(M,N) \;|\; 
V \circ X, Y \circ V \in \calk_A (M,N) \;  
\\
&&\qquad \forall X \in \calk_A(M), Y \in \calk_A(N) \}
\end{eqnarray*} 

}
\end{definition}

The $G$-invariant subring $\calm(\call_A(M,M)) \subseteq \Hom A (M)$ is denoted by $\calm(\call_A(M))$. 
Clearly, $\calk_A(M)$ is a two-sided ideal 
in $\calm(\calk_A(M))$. 
We note that the above requirement $V \circ X  
\in \calk_A (M,N)$ is trivially automatic and so superfluously stated. 
The other requirement $Y \circ V \in \calk_A (M,N)$  
for all $Y$ is trivially satisfied when 
$\phi \circ V \in \Theta_A(M)$ for all $\phi \in \Theta_A(N)$, as $\theta_{\xi,\phi} \circ V= 
\theta_{\xi,\phi \circ V}$.  
One might write ``$V^*(\phi)=\phi \circ V$" and this might justify to sloppily call the elements 
of 
$\calm(\calk_A(M,N))$ also `adjointable operators', but we shall not follow this path. 
In practice there will be 
little difference between the formal adjointable-operators definition or the multiplier definition, even one can construct 
counterexamples, see 
 the paragraph after definition \ref{def115}. 
 
The adjointable operators behave slightly better, and we shall exclusively work with them: 

\begin{definition}[Adjointable operators]
{\rm

Let $M$ and $N$ be right functional $A$-modules.

We write $\call_A(M,N) \subseteq \Hom A (M,N)$ 
for the $G$-invariant, abelian subgroup of {\em adjointable operators}, that is, we set 
\begin{eqnarray*}
&&\call_A(M,N) := \{ V \in \Hom A(M,N) \;|\; 
\phi \circ V 
\in \Theta_A(M)  , 
\; 
\forall 
\phi \in \Theta_A(N)  
\}
\end{eqnarray*} 

}
\end{definition}

\begin{definition}[Module homomorphism between functional modules]  
			\label{def116}
{\rm
A {\em functional module homomorphism} between two right functional $A$-modules
$(\cale,S,\Theta_A(\cale))$ and 
$(\calf,T,\Theta_A(\calf))$
is a $G$-equivariant, right 
$A$-module homomorphism $X:\cale \rightarrow \calf$ 
together with a  
$G$-equivariant, left $A$-module homomorphism 
$f:\Theta_A(\cale)  \rightarrow \Theta_A(\calf)$ 
such that for every $\phi \in \Theta_A(\cale)$ 
and $\xi \in \cale$ 
one has 
\begin{equation} 			\label{ref16b}
f(\phi) \big (X (\xi) \big )= \phi(\xi) 
\end{equation} 
 

}
\end{definition}

Just to demonstrate that the last definition works properly with compact operators, we note: 
 
\begin{lemma}				\label{lemma17}  
A module 
homomorphism $(X,f)$ as in the last definition 
with $X$ being surjective  induces a ring homomorphism 
$\sigma: \calk_A(\cale) \rightarrow \calk_A(\calf)$, and $f$ is then automatically injective.  

If $(X,f)$ is an isomorphism (i.e. $X$ and $f$ bijective)
$\sigma$ is an isomorphism 
and extends to an isomorphism 
$\call_A(\cale) \rightarrow \call_A(\calf)$.
\end{lemma}

\begin{proof}
Define $\sigma$ additively and by 
$\sigma(\theta_{\xi,\phi})= \theta_{X(\xi),f(\phi)}$. To see that it
is well-defined, assume that 
$\sum_i \xi_i \phi_i(\eta) = 0$ for all $\eta \in \cale$. 
Then by the last definition 
$$0=  \sum_i X(\xi_i \phi_i(\eta)) 
= \sum_i X(\xi_i) f(\phi_i)(X(\eta))
= \sum_i \theta_{X(\xi_i),f(\phi_i)} (X(\eta)) 
$$
Finally, with $\theta_{\xi,\phi} 
\theta_{\eta,\psi}  = \theta_{\xi \phi(\eta),\psi}$  we observe multiplicativity 
of $\sigma$. 
%
%
Injectivity of $f$ follows directly from  
(\ref{ref16b}). 
\end{proof}

If nothing else is said, the tensor product $\otimes$ means the 
exterior tensor product of 
abelian groups, or in other words of $\Z$-modules, so $\otimes$ means $\otimes_\Z$. 
We 
note that there is a well-known 
ring homomorphism 
\begin{equation}			\label{eqh} 
\pi: \Hom A(\cale 
) \otimes \Hom B(\calf)
\rightarrow 
\Hom {A \otimes B}(\cale \otimes \calf)  
: \pi(S \otimes T)  =  
S \otimes T
\end{equation} 
which notice, maps compact operators 
to compact operators, and adjoint-able 
operators to adjoint-able ones. 
However, none of the above ring homorphimsms 
need either 
to be injective 
or surjective.

\begin{definition}[Direct sum]  
{\rm

Let $(\cale_i,S_i)_{i \in I}$ be a 
a family of right functional $A$-modules. 
The algebraic {\em direct sum} $\cale:=\bigoplus_{i \in I} \cale_i$ 
of right $A$-modules with finite support with respect 
to $I$ is a right functional $A$-module 
with diagonal $G$-action $S:=\bigoplus S_i$ 
defined by $S_g(\oplus_{i \in I} \xi_i):= \sum_{i \in I} S_{i,g}(\xi_i)$. 
The functional space $\Theta_{A}(\cale)$ is defined to be the set of 
functionals $\phi:=\bigoplus_{i \in I} \phi_i$
defined by $\phi(\oplus_{i \in I} \xi_i)=
\sum_{i \in I} \phi_i(\xi_i)$
for $\phi_i \in \Theta_A(\cale_i)$. 
}
\end{definition}

\begin{definition}[External tensor product]  
{\rm

\label{def19}

Let $(\cale,S)$ be a right functional $A$-module and
$(\calf,T)$ a right functional $B$-module. 
The {\em external tensor product} 
$\cale \otimes \calf$ ($:=\cale \otimes_\Z \calf$) is a right functional $(A \otimes B)$-module under the diagonal $G$-action 
$S \otimes T$ defined by $(S_g \otimes T_g)(\xi \otimes \eta) := S_g(\xi) \otimes T_g(\eta)$. 
The functional space $\Theta_{A \otimes B}(\cale \otimes \calf)$ is defined to be the set of all sums  
of all elementary functionals 
$\phi \otimes \psi$ defined by
$(\phi \otimes \psi)(\xi \otimes \eta) = \phi(\xi) \otimes \psi(\eta)$
for $\phi \in \Theta_A(\cale)$ and $\psi \in \Theta_B(\calf)$. 
}
\end{definition}

\begin{definition}[Internal tensor product]  
		\label{def110}
{\rm

Let $\cale$ be a right functional $A$-module and
$\calf$ a right functional $B$-module 
and
$\pi:A \rightarrow \call_B( \calf)$ a ring homomorphism. 
Let 
$$\cale \otimes_\pi \calf := 
(\cale \otimes \calf)/
\Big \{ \sum_i \xi_i  \otimes \eta_i \in \cale \otimes \calf\, \Big |\, 
 \sum_i \pi \big (\phi(\xi_i)\big)\eta_i=0, 
\forall \phi \in \Theta_A(\cale) \Big \}$$
%
be the 
{\em internal tensor product}, which is a right functional $B$-module by the $B$-module structure of $\calf$, 
%
equipped with the diagonal $G$-action 
$S \otimes T$. 

We define $\Theta_B(\cale \otimes_\pi \calf)$ to be set of all sums of elementary functionals 
$\phi \otimes \psi$ defined by
$$(\phi \otimes \psi)(\xi \otimes \eta) = \psi \Big (\pi\big(\phi(\xi)\big) \eta \Big)$$
for $\phi \in \Theta_A(\cale)$ and $\psi \in \Theta_B(\calf)$.

}
\end{definition}

One may observe that $a \phi \otimes \psi = \phi \otimes \psi \circ \pi(a)$ 
for $a \in A$. 

\begin{definition}			\label{def111} 
{\rm 
A ring $A$ is called {\em quadratik} 
if $A= \sum A^2$ (the two-sided ideal of $A$ comprised of all finite sums of all 
products $ab$ for $a,b \in A$). 

A ring 
$A$ is called an {\em essential (right) ideal 
in itself} if the ring homomorphism 
$\pi:A \rightarrow \Hom A (A): \pi(a)(b)= ab$ is
injective.  
}
\end{definition}

\begin{definition}
{\rm 
A right functional $A$-module $\cale$ is called {\em cofull} 
if 
$\cale = \sum \cale \Theta_A(\cale)(\cale)$
(all finite sums of all 
products $\xi \phi(\eta)$ for $\xi,\eta \in 
\cale$ and $\phi \in \Theta_A(\cale)$). 

}
\end{definition}

As
$\call_A(\cale)$ is unital, 
it is obviously quadratik and 
an essential ideal in itself. 


\begin{lemma}		\label{lemma113}

If $\cale$ is a cofull, right functional $A$-module then $\calk_A(\cale)$ is 
quadratik 
and an essential ideal in itself. 

\end{lemma}

\begin{proof}
Quadratik: 
Given $\zeta \in \cale$ we may write it as 
$\zeta = \sum_i \xi_i \phi_i (\eta_i)$ 
and so
$$\theta_{\zeta,\psi} 
=   \theta_{ \sum_i \xi_i \phi_i(\eta_i),\psi}
= \sum_i \xi_i \phi_i(\eta_i) \psi 
= \sum_i \theta_{\xi_i,\phi_i} \theta_{\eta_i,\psi}
$$

Essential  
ideal in itself: 
Assume that $\sum_j \theta_{\eta_j,\psi_j} \neq 0$.  
Then its evaluation in some $\zeta$ (written as above) is nonzero. Thus 
$$\sum_j \theta_{\eta_j,\psi_j}(\zeta) 
=  \sum_i \Big ( \sum_j \eta_j \psi_j(\xi_i) \phi_i(\eta_i) \Big) =: \sum_i x_i
\neq 0$$
($x_i$ obviously defined). 
Hence at least one $x_i \neq 0$. But this means 
$\Big (\sum_j \theta_{\eta_j \psi_j} \Big ) \theta_{\xi_i, \phi_i} \neq 0$.
\end{proof}

\begin{definition}			\label{def115}
{\rm 

We turn any ring $(A,\alpha)$ into a right 
functional $(A,\alpha)$-module $(A,\alpha)$ 
over itself by setting $\Theta_A(A) := A^+$, where these functionals act by left multiplication in $A$ (i.e. $\phi_a(b)=ab$). 
If $A$ is quadratik, we 
prefer to set 
$\Theta_A(A) := A$.
}
\end{definition}

If $A$ is quadratik, then both functional spaces $A$ and $A^+$ yield obviously the same ring of compact operators $\calk_A(A)$, and this is 
what primarily counts for us. 
But notice, `functional-conceptually' there is a difference: 
for example, if $T \in \call_A(A)$ is a {\em multiplier} of $\calk_A(A)$, then $T$ 
may be `adjoint-able' with respect to the functional space $A$ but not with respect to $A^+$, as 
$\idd \circ T=T \notin A^+$ 
in general. 

Also, the functional space of $A\otimes_\pi \calf$ appears to become different under the two choices $A$ and $A^+$ for the functional space of $A$. 
So we choose $A$ if $A$ is quadratik, but sometimes allow $1 \in \Theta_A (A)$ for pure notational purposes, for example we shall write $\theta_{a,1}$, knowing it can always be resolved in $\theta_{a,1}= \sum_i \theta_{a_i,b_i}$ for $a_i, b_i \in A$. 

\begin{lemma}		\label{lemma114}

A ring $A$ is quadratik if and only if 
the right functional $A$-module $A$ over itself is cofull. 

\end{lemma}

\begin{proof}
Cofullness of the module $A$ obviously implies quadratik. For the reverse implication, note that $A = \sum A^n$ for all $n \ge 1$ 
whenever $A$ is quadratik. 
Hence $n=3$ yields.
\end{proof}

\begin{lemma}			\label{lemma116} 
Let $\pi:A \rightarrow B$ is a ring homomorphism between not necessarily quadratik rings $A$ and $B$. 

{\rm (i)} 
If $A$ is quadratik, then the range $\pi(A)$ is quadratik. 

{\rm (ii)} If $A$ is  quadratik and $I$ a two-sided ideal in $A$ then $A/I$ is quadratik.  

{\rm (iii)} If $\cale$ is a cofull $A$-modul 
then $\Theta_A(\cale)  (\cale)$ is a quadratik subring of $A$. 

{\rm (iv)}
If $\cale$ is cofull 
then
%
$\cale \otimes_\pi B$ is cofull.

{\rm (v)} 
If $\cale$ is cofull  
then
$\calk_{A}(\cale) \otimes \idd 
\subseteq \calk_A ( \cale \otimes_\pi B)$

\end{lemma}

\begin{proof}

(i) and (ii) are obvious. 
(iii) 
Using cofullness of $\cale$, write
$\phi(\eta)= \sum_i \phi(\eta_{1,i} ) \phi_{2,i}(\eta_{2,i})$ 
for $\phi \in \Theta_A(\cale),\eta \in \cale$. 
(iv) 
Using this, in $\cale \otimes_\pi B$  we may write 
$$ \xi \phi(\eta) \otimes b = 
\xi  \otimes \pi(\phi(\eta)) b
= \sum_{i,j} \xi  \otimes \pi(\phi(\eta_{1,i})) 
\pi(\phi_{2,i}(\eta_{2,i,j})) \pi(\phi_{3,i,j} (\eta_{3,i,j})) b
$$
$$= \sum_{i,j}  \big ( \xi \otimes \pi(\phi(\eta_{1,i}))  \big )  
(\phi_{3,i,j}  \otimes \varrho_{i,j}) (\eta_{3,i,j} \otimes b)$$
with $\varrho_{i,j}(d)= \pi(\phi_{2,i}(\eta_{2,i,j})) d$. 
By cofullness of $\cale$ we are done.

(v) Just compute that 
$\theta_{\xi \phi(\eta),\psi} \otimes 1
= \theta_{\xi \otimes \pi(
\phi(\eta)), \psi \otimes 1}$. 
\end{proof}

\begin{lemma}				\label{lemma117}
Let $B$ be a quadratik ring and $\cale$ a 
right functional $B$-module. 

{\rm (i)} 
Then $\cale \oplus B$ is cofull if and only if 
for every $\xi \in \cale$ there are $\eta_i \in\cale, b_i \in B$ such that $\xi = \sum_i \eta_i b_i$.

{\rm (ii)}  
Let $\cale,\calf$ and $\pi$ be as in definition 
\ref{def110} and $\calf$ be cofull. 
Then 
$(\cale \otimes_\sigma \calf) \oplus B$ is cofull. 

{\rm (iii)}    
If $\cale \oplus B$ is cofull then also 
$\calf \oplus B$ for any functional $B$-submodule $\calf$ of $\cale$.

\end{lemma}

\begin{proof}
(i)
Just notice that 
$\xi = \sum_i \eta_i b_i
= \sum_{i,j} (\eta_i  \oplus 0_B) (0 \oplus \phi_{i,j})( 0_\cale \oplus y_{i,j})$ for the functionals  
$\phi_{i,j} \in \Theta_B(B)$ defined by 
$\phi_{i,j}(a)= x_{i,j} a$, 
where $b_i = \sum_j x_{i,j} y_{i,j}$ is by quadratik of $B$. 
 %
Cofullness of the $B$-summand is by lemma \ref{lemma114}. 

(ii) and (iii) are then easily derived from (i).
\end{proof}

It might be useful to call $\cale$ {\em weakly 
cofull} if the right condition of lemma \ref{lemma117}.(i) is satisfied. 
It has the desirable property that it is permanent under taking submodules by the last 
lemma (if $B$ is quadratik). 


Usually, a $B$-submodule $\calf \subseteq \cale$ 
of a right functional $B$-module $\cale$ 
is turned to a 
functional module by 
restriction of the functionals, 
that is, one sets 
$\Theta_B(\calf):= \{\phi|_\calf|\, \phi \in \Theta_B(\cale)\}$.


There presumably is no point in considering 
functional modules whose elements are not separated by the functionals. If so, we may 
call the module {\em separated by functionals}, 
and here is the procedure how we get it:  

\begin{lemma}			\label{lemma219} 
We may turn a right functional $A$-module 
$(\cale,S)$ to a right functional $A$-module 
$$\overline \cale := \cale /\{\xi \in \cale|\, 
\phi(\xi) = 0 , \, \forall \phi \in 
\Theta_A(\cale) \}$$
whose functional space 
$$\Theta_A (\overline \cale ) :=  
\{ \overline \phi|\, \phi \in \Theta_A (\cale ) \}$$ 
where $\overline \phi([\xi])= \phi(\xi)
$, 
separates the points. {\rm (i)} Also, there is a 
ring homomorphism 
$\pi: \call_A(\cale) \rightarrow \call_A(\overline \cale)$ which restricts to a ring homomorphism  $  \calk_A(\cale) \rightarrow \calk_A(\overline \cale)$. 

{\rm (ii)} Further, $\overline{\cale \otimes_\pi \calf} = \cale \otimes_\pi \overline \calf$. 

{\rm (iii)} 
We also 
have $\cale \otimes_{\id} A = \overline \cale$. 
\end{lemma}

\begin{proof}

We also define $A$-module structure by $[\xi]a:= [\xi a]$, and a $G$-action $\overline S$ by $\overline S_g([\xi]):= [S_g(\xi)]$. 
(i) 
If $T \in \call_A(\cale)$ then we set
$\pi(T)([\xi]):= [T(\xi)]$. 

(ii) The outline is that we have maps $\cale \otimes \calf \rightarrow \cale \otimes_\pi \calf \rightarrow 
\overline{\cale \otimes_\pi \calf}
\cong  (\cale \otimes \calf) / I$, where
$I$ means intersection of the kernels of all functionals as defined in definition 
\ref{def110}. 
\end{proof}

\begin{lemma}				\label{lemma220}
{\rm (i)} 
Let $A,B$ be rings and $\pi:A \rightarrow B \subseteq \call_B(B)$ a ring homomorphism. 
Then there is a functional $B$-module isomorphism 
$X: A \otimes_\pi B \rightarrow B_0$ to a functional $B$-submodule $B_0$ of $B$.

{\rm (ii)} If $B$ is quadratik, then $B \otimes_{\id} B \cong B$ as functional $B$-modules. 

{\rm (iii)} 
If $B$ is quadratik, then $B \otimes_{\id} (B \rtimes G) \cong B \rtimes G$ as functional $B \rtimes G$-modules.


{\rm (iv)} 
If $A$ is unital and $\pi:A \rightarrow \call_B(\cale)$ is a unital ring homomorphism 
then $A \otimes_\pi  \cale \cong \cale$ as
right functional $B$-modules. 

\end{lemma}

\begin{proof}

(i) 
We have a 
$B$-module functional isomorphism
$(X,l)$ defined by 
\begin{eqnarray*}
&&  X:A \otimes_\pi B \rightarrow B_0:= 
\sum 
\pi(A) B  \subseteq B
: X(a\otimes b) = 
\pi(a) b 				\\
&&  l: \Theta_B(A \otimes_\pi B) 
\rightarrow \Theta_B(B_0) : l(m_y \otimes m_x) = m_{x \pi(y)}
\end{eqnarray*} 
%
where $m_x$ denotes left multiplication operator ($x \in B,y \in A$) 
and $\Theta_B(B_0)$ is defined to be the image of $l$, into a (now functional) $B$-submodule $B_0$ of $B$, see also lemma \ref{lemma17}. 

If $X(\sum_i a_i \otimes b_i)=0$ then also 
$\sum_i\pi(x a_i) b_i=$ for all $x \in A$, and so  
$\sum_i a_i \otimes b_i=0$ by definition 
\ref{def110}, proving injectivity of $X$. 
 
(ii)
follows from (i). 
(iii) is proven similarly. 

(iv)
The 
functional $B$-module isomorphism is $p: A \otimes_{s_-} \cale \rightarrow \cale$ defined by 
$p(a \otimes \xi)= \pi(a)(\xi)$, and 
its inverse 
isomorphism defined by $p^{-1}  (\xi)= 1_A  \otimes \xi$, 
together with 
$f:\Theta_A( \cale) \rightarrow \Theta_A(A \otimes_{\pi} \cale)$ defined by 
$f(\phi)= \idd \otimes \phi$, thereby recalling 
definition \ref{def116} and lemma \ref{lemma17}
%
\end{proof}

\begin{lemma}			\label{lemma221} 
{\rm (i)} 
A direct sum $\cale = \oplus_i \cale_i$ of $A$-modules $\cale$ and $\cale_i$ is a direct sum 
of functional $A$-modules if and only if $\cale$ is a functional $A$-module and the $A_i$s are functional $A$-submodules of $\cale$ obtained by 
restriction of the functional space of $\cale$.

{\rm (ii)} 
The exterior and interior tensor products $\cale \otimes_\pi \calf$  
commute with direct sums of functional modules in both `variables' 
$\cale$ and $\calf$ (provided $\pi(A)$ acts on the summands $\calf_i$ of $\calf$).

\end{lemma}

\begin{proof}
(i) is an easy check, and (ii) may be deduced from (i) and well-known corresponding isomorphisms of $A$-modules. 
\end{proof}

For a non-equivariant ring $A$ 
we write $M_n(A)$ for the ring of $n \times n$-matrices with coefficients in $A$, 
and
also 
$M_\infty (A)$ 
for arbitrarily big infinite matrices.

\begin{lemma}		\label{lemma1172}   
Let $(A,\alpha)$ be a quadratik ring.  
Then there are ring isomorphisms 
$$
(A,\alpha) \cong (\calk_A(A),\Ad(\alpha)) \qquad 
(M_n(A), \alpha \otimes 1_{M_n})  \cong \calk_A((A,\alpha)^n)  	$$ 
Also, $M_n(A)$ is quadratik.  
\end{lemma} 

\begin{proof}

We just discuss the first isomorphism, the other one is done as in the next lemma. 
By lemmas \ref{lemma113} and \ref{lemma114}, 
$\pi$ of definition \ref{def111} is injective. 
The last claim follows also from these lemmas and the second isomorphism, or completely elementary by suitable matrix multiplication 
as in the proof of lemma \ref{lemma61}. 
\end{proof}

\begin{lemma}		\label{lemma1170}

Let $(A,\alpha)$ be a ring and $(\cale_i,S_i)$ 
right $(A,\alpha)$-modules
for $1 \le i \le n \le \infty$. 
Then there are ring isomorphisms 
%
$$(M_n(\calk_A(\cale)), \Ad( \oplus_{i=1}^n S_i))  \cong \calk_A(\oplus_{i=1}^n(\cale,S_i))  $$  
$$(M_n(\call_A(\cale)), \Ad( \oplus_{i=1}^n S_i))  \cong \call_A(\oplus_{i=1}^n(\cale,S_i))   $$  

where the $G$-action $\theta :=\Ad( \oplus_{i=1}^n S_i)$
on the  matrix algebras is defined by 
$$\theta_g((x_{ij}))=( S_{i,g} \circ x_{ij} \circ S_{j,g^{-1}})$$

\end{lemma}

\begin{proof}
The last isomorphism $\sigma$ is given by  
$\sigma((x_{ij})_{ij})((\xi_i)_i)= (\sum_{k=1}^n x_{ik}(\xi_k))_i$, and one easily checks $G$-equivariance,
and that it restricts to $M_n(\calk_A(\cale))$ and recall definition \ref{defcompact} for $G$-invariance. 
\end{proof}

The following lemma is elementary and well-known, and our sketched proof might appear superfluous.

\begin{lemma}		\label{lemma120} 
Let $A,B$ rings and $\cale,\calf,M,N$ be modules. Let $(\G,\triv)$ be a field. 
Write $\otimes^\G$ for the $\G$-vector space 
(i.e. $\G$-balanced) 
tensor product.   
Then we have isomorphisms 
$$(A \otimes B) \otimes \G
= (A \otimes \G)  \otimes^\G
(B \otimes \G)$$
$$(\cale \otimes_\pi \calf) \otimes \G
= (\cale \otimes \G)  \otimes^\G_{\pi \otimes \idd}  
(\calf \otimes \G)$$
$$(A \rtimes_\alpha  G) \otimes \G \cong 
(A \otimes \G) \rtimes_{\alpha \otimes \triv} 
G$$
$$ \calk_\G(\G^n) \cong M_n(\G) , \qquad
M_n(A) \otimes \G \cong M_n(A \otimes \G)$$
%
%

If $\F$ is a subfield of $\G$ and all rings and modules here are even vector spaces over $\F$ 
(so $\F$-algebras and their modules and homomorphisms, $\otimes$ replaced by $\otimes^\F$), then everything above is also true. 
\end{lemma}

\begin{proof}
(Sketch)
For the first identity we observe that 
by the universal property of the tensor product   
an obvious homomorphism and inverse homomorphism are well-defined. 
These homomorphisms respect certain quotients 
and so we arrive at 
the second line. 
The third line is because the 
tensor product exchanges with direct sums as abelian groups. 
  The fourth isomorphism is $\theta_{a e_i,b p_j} \mapsto ab e_{i,j}$ for $a,b \in \G$, $e_i\in \G^n, e_{i,j} \in M_n(\G)$ the canonical basis vectors and $p_j$ the canonical projection onto the $j$th coordinate. 
%
\end{proof}



{\em We impose now the following convention throughout: 
}

From now on, all rings are assumed to be 
quadratik $G$-rings if nothing else is said. 
All modules are understood to be right functional 
$G$-modules and they 
may not be cofull, 
but it is implicitly understood that they  
are 
whenever they are the underlying module $\cale$  
of the compact operators $\calk_B(\cale)$,
because of lemma \ref{lemma113},
if nothing else is said.

\section{$GK$-theory}

\label{sec3}

In this paper we write compositions of morphisms in a category and compositions of functions from 
left to right. That is, for instance, if $f:A \rightarrow B$ and $g : B \rightarrow C$ are maps, then we write
$f g$ for $g \circ f$, where composition operator $\circ$ is used in the usual sense from right to left.
This will go as far as that we write $fg(x)$ for $g(f(x))$.
In spaces of operators like $\call(\cale)$ we use the multiplication in the usual sense, that is,
$S T$ means $S \circ T$ for $S,T \in \call(\cale)$, but to avoid confusion, we mostly write
$S \circ T$.


\begin{definition}[Distinguished coordinate] 
\label{def21} 
{\rm

Let $(A,\alpha)$ be a ring. 
We write $\H_A$ for 
a direct sum 
$\H_A:= \bigoplus_{i \in I} A$ of copies of $A$,
where $I$ may have any 
cardinality above zero, 
and where the $G$-action on $\H_A$ is of the form $\alpha \oplus S$ on $\H_A = A \oplus \bigoplus_{i \in I \backslash \{i_0\}} A$, and the direct summand $(A,\alpha)$ on the position $i_0 \in I$ is called the {\em distinguished coordinate} 
of $\H_A$. 
Thereby, $S$ may be any $G$-action and need not to be a direct sum $G$-action.
}
\end{definition}

This $\H_A$ is merely an auxiliary space. 
It will become clear, that it is sufficient to 
take only $\H_A = A$, see remark \ref{rem12} below. 

\begin{definition}[Corner embedding] 		\label{def22}
{\rm

Let $(A,\alpha)$ be a ring.  
Let $(\cale,S)$ be a not necessarily cofull,  right functional $A$-module. 
Let $\H_A =(\oplus_{i \in I} A, \alpha \oplus T)$ be a direct sum. 
Assume that $\cale \oplus \H_A$ is cofull. 
(Cf. lemma \ref{lemma117}.) 

The injective 
ring homomorphism
$$e: (A,\alpha)  \rightarrow  \calk_A 
\big (
(\cale \oplus 
\H_A, S \oplus \alpha \oplus T) \big )$$
defined by 
$e(a) (\xi \oplus b \oplus \eta)
= 0 \oplus  a b \oplus 0$
for $\xi \in \cale, a \oplus \eta \in \H_A = A \oplus \bigoplus_{i \in I \backslash \{i_0\}} A$ 
is called a  {\em corner embedding}. 

}
\end{definition}

 That is,
$e(a)$ is the multiplication operator 
acting on the distinguished coordinate $(A,\alpha)$ 
of $\cale \oplus \H_A$.
Note that $e(a)= \theta_{a,\phi}$, where $\phi$ projects onto the distinguished coordinate.  
We call also ring homomorphisms which are corner embeddings up to ring isomorphism 
corner embeddings. 

There is also a corner embedding 
$e:A \rightarrow (M_\infty( A), \Ad( \alpha \oplus S))$ in the more usual sense
 defined by $e(a)=\diag(a,0,0,\ldots)$ 
by lemmas \ref{lemma1172} and \ref{lemma1170}. 


Let $C[0,1]:= (C([0,1],\C),\triv)$ be the complex-valued continuous functions on the unit interval, endowed with the trivial $G$-action.

For a ring $(B,\beta)$ we write $B[0,1] := (B \otimes C[0,1], \beta \otimes \triv)$ for the exterior tensor product with diagonal $G$-action. 
 
Gersten \cite{gersten} has defined a homotopy theory 
for rings by declaring polynomial functions 
as 
homotopy. For our purposes, logically 
this is 
a too 
strict notion and we rather take continuous functions 
in the complexified ring. 
(Replacing $\Z[x]$ in Gersten by $C[0,1]$ here.)

\begin{definition}
{\rm 

Let $(A,\alpha)$ and $(B,\beta)$ be rings. 
A {\em homotopy} is a $G$-equivariant ring homomorphism 
$f: (A,\alpha) \rightarrow (B \otimes C[0,1], \beta \otimes \triv)$ such that 
for the evaluation maps $v_t:B \otimes C[0,1] \rightarrow B \otimes \C$, $v_t(b \otimes z)= 
b \otimes z(t)$, the following holds true: 
the ranges of $f v_0$ and $f v_1$ are in the subring $B \otimes \Z$. 

The ring homomorphisms $f v_0 \tau$ and $f v_1 \tau$ from $A$ to $B$ are then called to be {\em homotopic}, where
$\tau:B \otimes \Z \rightarrow B$ is the canonical ring isomorphism $\tau(b \otimes n)= bn$.
 
}
\end{definition}

It is often useful to view homotopy as above 
by complexifying $B$ to $\hat B:=B \otimes \C$ at first and then do $\C$-algebra homotopy 
$f :A \rightarrow \hat B$  as usual. 
%
The choice of continuous homotopies is 
relatively arbitrary, having some $\sin$ and 
$\cos$ would be enough.

\begin{definition}
{\rm 
We fix a certain 
category $\ring G$ of $G$-equivariant, quadratik  rings (objects) and $G$-equivariant ring homomorphisms (morphisms),  
closed under all constructions 
needed in this paper. 
(In particular, the compact operators, 
adjointable operators, and corner embeddings.) 
}
\end{definition}

We are going to recall the definition of $GK^G$-theory
(``Generators and relations $KK$-theory" 
with ``$G$"-equivariance) 
for which we refer for more details to 
\cite{bgenerators}. 
The split exactness axiom is slightly but equivalently altered,
see \cite[Lemma 3.7]{bremarks}. 


\begin{definition}		\label{def121}
{\rm

Let $GK^G$ be the following 
category.
Object class of $GK^G$ is the object class of the class $\ring G$
(i.e. a given class of $G$-equivariant 
quadratik rings). 

Generator 
morphism class is the collection of all $G$-equivariant 
ring homomorphisms $f:A \rightarrow B$ (with obvious source and range objects) 
from $\ring G$ and the collection of the following ``synthetical" morphisms:

\begin{itemize}

\item
For every equivariant corner embedding $e:(A,\alpha) \rightarrow (\calk_A(\cale \oplus \H_A),\delta)$ as in definition \ref{def22} 
add a morphism called $e^{-1}: (
\calk_A(\cale \oplus \H_A),\delta) \rightarrow (A,\alpha)$.

\item
For every equivariant short split exact sequence
\begin{equation}			\label{splitexact}
\cals: \xymatrix{0 \ar[r] & (B,\beta)  \ar[r]^j & (M,\delta)  \ar[r]^f & (A,\alpha) \ar[r] \ar@<.5ex>[l]^s & 0}
\end{equation}
in $\ring G$ add a morphism called
$\Delta_{\cals}: (M,\delta) \rightarrow (B,\beta)$
or $\Delta_s$ if $\cals$ is understood.

\end{itemize}

Form the free category of the above generators
together with free addition and subtraction of morphisms having same range and source 
(formally this is like the free ring generated by these generator morphisms, but one can only add 
and multiply if 
source and range fit together)
and divide out the following relations
to turn it into the category $GK^G$:

\begin{itemize}

\item
({\em 
Ring category})
Set 
$g \circ f = f g$ for all $f \in \ring G(A,B)$
and $g \in \ring G(B,C)$.

\item
{(\em Unit})
For every object $A$,
${\rm id}_A$ is the unit morphism.

\item
({\em Additive category})
For all 
diagrams
$$\xymatrix{A  \ar[r]^{i_A} &  A \oplus B  \ar[r]^{p_B} 
\ar@<.5ex>[l]^{p_A}
&  B  \ar@<.5ex>[l]^{i_B} }$$
(canonical projections and injections) set
$1_{A \oplus B}= p_A i_A + p_B i_B$.

\item
({\em Homotopy invariance})
For all homotopies $f:A \rightarrow B[0,1]$ in $\ring G(A,B)$ set $f_0 = f_1$.

\item
({\em Stability})
All corner embeddings $e$ as in definition \ref{def22}  are invertible with inverse $e^{-1}$.

\item
({\em Split exactness})
For all split exact sequences (\ref{splitexact}) set
\begin{eqnarray*}
&& 1_B   =  j \Delta_s  \\
&& 1_M   =    \Delta_s j + f s
\end{eqnarray*}

\end{itemize}

}
\end{definition}

For more comments like these confer also 
\cite{bgenerators}.

\begin{remark}			\label{rem12}
{\rm

(i)
Set-theoretically $M= j(B) + s(A)$ in (\ref{splitexact}), and this is a direct 
 sum by the last point.

(ii) If we have given an additional homomorphism $u:A \rightarrow M$ in (\ref{splitexact}) then this is a second split for $f$ if and only if
$$u(a)-s(a) \in j(B) \quad \forall a \in A$$

(iii) Given (\ref{splitexact}), we have $s \Delta_s =0$
because
$s \Delta_s = s \Delta_s j \Delta_s =
s (1-fs) \Delta_s =0$.

(iv) The $GK^G$-theory as in Definition \ref{def121} remains completely unchanged if 
we would require stability only for corner embeddings with 
$\H_A = \oplus_{i \in  I} A$ with 
only one fixed cardinality $|I| \ge 1$. 
Indeed, if $|I|=1$ then $\cale \oplus \H_A 
= \cale \oplus A$, 
and so all other cases are obviously included. 
For the opposite direction write 
$$
\xymatrix{ 
A  \ar[r]^e    &  
\calk_A(\cale \oplus A) \ar[r]^-f  
 & M_\infty (\calk_A(\cale \oplus A)) \cong \calk_A( 
\oplus_{i \in I} (\cale \oplus A))   
\cong   
\calk_A( 
H_\cale \oplus \H_A)
}$$
where  $H_\cale:= \oplus_{i \in I} \cale$ and  
where $e$ and $f$ are corner embeddings in the sense of definition \ref{def22} (modulo ring isomorphisms). 
Since $f$ and $e f$ are invertible 
corner embeddings by our required axioms for fixed $|I|$, 
$e$ must also be invertible .

(v) 
By the analogous reasoning, 
the above analogous corner embedding $e$ is also 
invertible in $KK^G$-theory 
for $\cale$ a countably generated $A$-module,
because we have 
$\calk_A( 
H_\cale \oplus \H_A) \cong 
\calk_A( \H_A)$ by 
Kasparov's (equivariantly interpreted) stabilization theorem (cf. \cite[lemma 8.5]{baspects}) for $\H_A$ being the countably infinite direct 
Hilbert module sum of copies of $A$, and so 
$ef$ 
is also an invertible corner embedding 
in $KK^G$-theory.   

}

\end{remark}

We sometimes refer to the following notion
and variants of it (in particular $\call_B$ replaced by $\calk_B$). 
Typically, corner embeddings 
might be rotated to other corner embeddings. 

\begin{definition} 			\label{def26}  
{\rm 
Let $\pi: A \rightarrow \call_B((\cale,S) \oplus (\cale,S)) =:X$  be 
a ring homomorphism. 
Under a {\em rotation homotopy} 
we mean the homotopy 
$$\sigma:A \rightarrow X[0,1] \cong 
M_2( \call_B( \cale) \otimes C[0,1])  
$$
%
%

given by  
$\sigma(a) = (\idd_{\call_B(\cale)} \otimes V) \circ (\pi(a) \otimes \idd_{C[0,1]}) \circ (\idd_{\call_B(\cale)} \otimes  V^{-1} )$  
for the complex-valued homotopy 
$V_t :=\Big (\begin{matrix} \cos t & \sin t \\ - \sin t &  \cos t \end{matrix}
\Big ) \in M_2(\C)$ for $t \in [0,\pi/2]$, 
$V^{-1}_t=V_{-t}$.   
}
\end{definition}

\begin{definition}
{\rm
Given $P \in GK^G(A,B)$ we define 
$P^*:GK^G(B,X) \rightarrow GK^G(A,X)$ 
by $P^*(z):= P z$ and
$P_*:GK^G(X,A) \rightarrow GK^G(X,B)$ 
by $P_*(z):= z P$.
}
\end{definition} 

Let us finally remark that we 
restrict the object class $\ring G$ to a 
set by a particular selection such that
$GK^G$ turns to a small category and 
all Hom-classes $GK^G(A,B)$ are sets. 
See for example \cite{bgenerators}.

\section{Double split exact sequences}

\label{sec4}

We shall also recall in this section and in the next 
ones a couple of constructions and results 
from \cite{baspects}, where $GK$-theory for $C^*$-algebras is considered. 

Throughout, $(A,\alpha)$ and $(B,\beta)$ are $G$-rings.  

\begin{definition}			\label{def31}
{\rm
A {\em double split exact sequence}
is a diagram of the form
$$\xymatrix{
0 \ar[r] &  
(B,\beta)  \ar[r]^j & (M,\gamma) \ar[rrr]^f 
\ar[dr]^{e_{11}}
&
& &  (A,\alpha)  \ar@<.5ex>[lll]^{s}
\ar[ld]^{t} 		\ar[r]
& 0\\
& 
&
& (M_2(M), \theta) 
& (M,\delta) \ar[l]^{e_{22}}		
}
$$
where all morphisms in the diagram are equivariant 
ring homomorphisms,
the first line is a split exact sequence in 
$\ring G$,  
$t$ is another split in the sense that
$t f = 1_A$ in non-equivariant $\ring G$ 
and $e_{ii}$ are the 
corner embeddings.

}
\end{definition}

\begin{definition}			\label{def32}
{\rm
Consider a double split exact sequence as above.
We denote by $\mu_\theta$, or $\mu$ if $\theta$ is understood, the morphism
$$\mu: (M,\delta) \rightarrow (M,\gamma) :
\mu = e_{22} e_{11}^{-1}$$
in $GK^G$.
The {\em morphism in $GK^G$ associated to the 
double split exact sequence} is 
$$t \mu \Delta_{s}$$
}
\end{definition}

We use sloppy language and say for example ``the diagram is $t \mu \Delta_s$ in $GK$'', or two double split exact sequences are said to be ``equivalent" 
if their associated morphisms are.
Throughout, the short notation for the above double split exact sequence will be
$$\xymatrix{
(B,\beta)  \ar[r]^j & (M,\gamma) \ar[r]^f 
&  (A,\alpha)  \ar@<.5ex>[l]^{s,t}
} $$
Notating such a diagram, it is implicitly understood that this is a double split exact sequence as above if nothing else is said.
Often $s,t$ is stated as $s_\pm$, 
which has to be read as $s_-,s_+$.
The $G$-action $\theta$ of definition \ref{def31}
will sometimes be called the ``$M_2$-action of the double split exact sequence" 
for simplicity.

\begin{lemma}[{\cite[lemma 5.4]{baspects}}]			\label{lemma21}
Consider two double split exact sequences
which are connected by three morphisms
$b,\Phi,a$ in $GK^G$ as in this diagram:
$$\xymatrix{
B  \ar@<.5ex>[r]^i
\ar[ddd]^b
 & M \ar[rrr]^f 
\ar[rd]^{e_{11}}
\ar@{..>}[ddd]^\phi
 &
& & A  \ar@<.5ex>[lll]^{s_-}
\ar[ld]^{s_+} 
\ar[ddd]^a
\\
&& M_2(M) 	\ar[d]^\Phi 
& M \ar[l]^{e_{22}}		
\ar@{..>}[d]^\psi   \\
&& M_2(N)  
& N \ar[l]^{f_{22}}		\\
D  \ar@<.5ex>[r]^j & N  \ar[rrr]^g 
\ar[ur]^{f_{11}}
 &
& & C  \ar@<.5ex>[lll]^{t_-}
\ar[ul]^{t_+}
} 
$$

Here we have defined
$$\phi:= e_{11} \Phi f_{11}^{-1}
\qquad
\psi:= e_{22} \Phi f_{22}^{-1}
$$

(i)
Then
for the commutativity of the left rectangle of the diagram we note
$$\Delta_{s_-} b = \phi \Delta_{t_-}
\qquad \Leftarrow   
\qquad
i \phi = b j \quad {\rm and}
\quad 
f s_- \phi = \phi g t_-$$

(ii)
For 
commutativity within the right big square of the diagram we observe
$$s_+ \mu \phi = a t_+ \mu 
\qquad \Leftrightarrow   
\qquad
s_+ \psi = a t_+
$$

(iii)
Consequently, commutativity of double split exact sequences in this diagram can be decided as 
$$s_+ \mu \Delta_{s_-} b
= a t_+ \mu \Delta_{t_-} 
\qquad \Leftarrow   
\qquad
\mbox{Conditions of (i) and (ii) hold true}
$$

\end{lemma}

Let us revisit the last lemma and state for 
further reference: 


\begin{remark} 
\label{cor22}
{\rm

 $\Phi$ will always be of the form 
$\Phi= \phi \otimes 1_{M_2}$ for 
a non-equivariant 
ring 
homomorphism $\phi:M \rightarrow M$.
Then this $\phi$ is the $\phi$ and $\psi$ in the above diagram as non-equivariant maps, and both
are automatically equivariant as maps as 
entered in the diagram when $\Phi$ is.
So, 
to check commutativity of a diagram 
involving two exact double-split exact sequences with lemma 
\ref{lemma21}, we only have to show that 
$f s_- \phi = \phi g t_-, i \phi = b j$
and $s_+ \phi =a t_+$, and that
$\Phi$ is $G$-equivariant. 

}
\end{remark}


\section{The $M \square A$-construction}
%
%

\label{sec6}

We shall use the following standard procedure
to produce split exact sequences,
and this is in fact 
key:



\begin{definition}		\label{lemma35}
{\rm
Let $i: (B,\beta) \rightarrow (M,\gamma)$ be an equivariant injective 
ring homomorphism 
such that
the image of $i$ is a  
two-sided ideal in $M$.
Let $s: (A,\alpha) \rightarrow (M,\gamma)$ be an equivariant 
ring homomorphism. 
%
Then we 
define the equivariant 
$G$-subring   
$$M \square_s A := 
\{ (s(a)+i(b),  a) \in M \oplus A|\,a \in A, b \in B\} $$ 
of $(M \oplus A, \gamma \oplus \alpha)$. 
We also just write $M \square A$ if $s$ is understood. 
The $G$-action on $M \square A$ is denoted by
$\gamma \square \alpha$.
In particular we have a split exact sequence
$$\xymatrix{0 \ar[r] & B  \ar[r]^j & M \square A \ar[r]^f & A \ar[r] \ar@<.5ex>[l]^{s \square 1} & 0},$$
where $j(b)= (i(b),0)$, $f(m,a)=a$ and $(s \square 1)(a) :=(s(a),a)$ for all $a \in A, b \in B, m \in M$.

}
\end{definition}

If we have given a double split exact sequence 
as in definition \ref{def31} with $M$ of the form $M \square A$
then 
it is understood that $j,f$ and $s \square 1$ 
are always of the form as in the last definition.
Moreover, the construction of $M \square A$ refers always to the first notated split $s$, or the split indexed by minus (e.g. $s_- \square 1$)
if it appears in a double split exact sequence.
We denote elements of $M \square A$ by $m \square a :=(m,a)$. The operator $\square$ binds weakly, that is for example, $m+n \square a =(m+n)\square a$.

Observe that $M \square_s A = (s\square \idd)(A) + B$ is quadratik if $A$ and $B$ are. 

Non-equivariantly we have 
\begin{equation}				\label{isq}
M_2(M \square_s A) \cong
M_2(M) \square_{s \otimes 1} M_2(A) \subseteq M_2(M) \oplus M_2(A)
\end{equation}
with respect to
$i \otimes 1_{M_2}$ and $s \otimes 1_{M_2}$.

\begin{definition}
{\rm 
If we have $G$-algebras $(M_2(M),\gamma)$ and $(M_2(A),\delta)$ 
and $(M_2(M \square A), \theta)$
is canonically a $G$-invariant $G$-subalgebra of 
$(M_2(M) \oplus M_2(A),\gamma \oplus \delta)$ 
then we call $\theta$ also $\gamma \square \delta$.
} 
\end{definition}

\begin{remark} 		\label{remark43}
{\rm 
It will be important to observe, 
and is understood that the reader is aware of it in checking the validity of double split exact sequences, that 
the non-equivariant splits of the 
exact sequence of 
definition \ref{lemma35}
are exactly the maps of the form $t \square 1$, 
where $t : A \rightarrow M$ is any 
non-equivariant ring homomorphism 
such that, for all $a \in A$,  
$$t(a) - s(a) \in j(B)$$

Consequently, also then  non-equivariantly $M \square_s A = 
M \square_t A$. 
}
\end{remark} 

We may bring any 
double split exact sequence to 
the form of definition \ref{lemma35}, see 
%
%
for instance 
{\cite[lemma 6.3]{baspects}}, 
 or lemma \ref{lemma14} below, 
with $M_2$-action of the form 
$\gamma \square \delta$. 

\section{Actions on $M_2(A)$}

\label{sec7}

In this section we want 
to inspect closer how
a $M_2$-action of a double split exact sequences looks like.
%
%
This is a key lemma, and it is just a special 
case of lemma \ref{lemma1170}, but despite that let us restate it:

\begin{lemma}[{\cite[lemma 7.1]{baspects}}]  
			\label{lemma51}
Let $S,T$ be two $G$-actions on a 
right functional 
$(B,\beta)$-module $\cale$.
Then
$$
\alpha_g
\left (\begin{matrix} x & y\\
 z & w
\end{matrix}
\right ) =
\left (\begin{matrix} S_g x S_{g^{-1}} & S_g y T_{g^{-1}}\\
 T_g z S_{g^{-1}} & T_g w T_{g^{-1}}
\end{matrix} \right ) 
$$
defines a $G$-action $\alpha$ on $M_2\big ( \call_B(\cale) \big)$
which leaves the ideal $M_2\big ( \calk_B(\cale) \big)$ invariant. 


This is actually the inner 
action ${\rm Ad} (S \oplus T)$ 
on $\call_B \big ((\cale,S) \oplus (\cale,T) \big )$.

\end{lemma}

\begin{definition}
{\rm
The
$\alpha$ of the last lemma is also denoted by
${\rm Ad}(S \oplus T)$ or
${\rm Ad}(S,T)$.
}
\end{definition}


\begin{lemma}[Cf. {\cite[lemma 7.3]{baspects}}]  
		\label{lemma61}
Let $(A,\alpha)$ and $(A,\delta)$ be $G$-rings. 

Let $(M_2(A),\theta)$ be a $G$-ring 
and the corner embeddings
$e_{11} :(A,\alpha) \rightarrow
(M_2(A),\theta)$ and
$e_{22} :(A,\delta) \rightarrow
(M_2(A),\theta)$ be equivariant.

Then
$\theta$ is of the form

$$
\theta_g
\left (\begin{matrix} x & y\\
 z & w
\end{matrix}
\right ) =
\left (\begin{matrix} \alpha_g(x) & \beta_g(y)\\
 \gamma_g(z) & \delta_g(w)
\end{matrix} \right ) 
$$

Also:



{\rm (i)}
One has the relations
\begin{eqnarray}
&&\gamma_g(ax)= \delta_g(a)\gamma_g(x)	
\qquad \gamma_g(xb)= \gamma_g(x) \alpha_g(b)	
	\label{l2}  \\
&& \beta_g(ax)= \alpha_g(a) \beta_g(x)  
\qquad
\beta_g(xb)= \beta_g(x) \delta_g(b)	\\
&&\alpha_g(xy)= \beta_g(x) \gamma_g(y)	
\qquad
\delta_g(xy)= \gamma_g(x) \beta_g(y)	
	\label{l4}  \\
&&
\qquad \gamma_{gh}= \gamma_g \gamma_h
\qquad \beta_{gh}= \beta_g \beta_h 
\end{eqnarray}

{\rm (ii)}
$(A,\gamma)$ is 
a functional Morita equivalence 
$((A,\delta),(A,\alpha))$-bimodule, 
see definition \ref{def131} below, 
where the bimodule structure is 
multiplication in $A$,
and the
right module functional space $\Theta_{(A,\alpha)}((A,\gamma))$ is given by left multiplication $\theta_a(b) = a b$, and the left module functional space $\Theta_{(A,\delta)}((A,\gamma))$ is 
given by right multiplication 
$\varrho_b( a) =a b$.

{\rm (iii)}
Analogously,
$(A,\beta)$ is 
a functional Morita equivalence 
$((A,\alpha),(A,\delta))$-bimodule, which is inverse to the later one.

{\rm (iv)}
Let
$\chi : A \rightarrow 
\call_A(A)$ be the natural embedding 
given by $\chi(a)(b) = ab$. 

Then $\alpha$ and $\gamma$ are $G$-actions
on the 
right $(A,\alpha)$-module $A$.

Consequently we have the $G$-action
${\rm Ad}(\alpha \oplus \gamma)$
on the matrix algebra $M_2(\call_{(A,\alpha)}(A))$.

The map
$$\chi \otimes 1_{M_2}:(M_2(A), \theta) 
\rightarrow 
(M_2(\call_{(A,\alpha)}(A)), {\rm Ad}(\alpha \oplus \gamma))$$
is a $G$-equivariant injective 
ring homomorphism.

{\rm (v)}
$\alpha$ and  
$\gamma$, or $\beta$ and $\gamma$, determine 
 $\theta$ uniquely and completely.

{\rm (vi)}
$\theta$ determines $\alpha, \beta, \gamma$ and $\delta$ uniquely.

{\rm (vii)}
In general, $\alpha$ and $\delta$ do not determine $\gamma$ and thus not $\theta$.

{\rm (viii)}
If we drop all assumptions then we may add:

A $G$-ring $(A,\alpha)$ and a 
right functional $(A,\alpha)$-module action $\gamma$ on $(A,\Theta_A(A))$ 
alone 
(implicitly meaning also that $\Theta_A(A)$ becomes $\Ad(\gamma,\alpha)$-invariant)
ensure the existence of the above $\theta$
with all assumptions and assertions of this lemma.
%


\end{lemma}

\begin{proof}
By assumption, the corner rings must be $G$-invariant. Hence, by quadratik we write an element of $A$ as a sum of 
products $abc$ and then 
 by 
$$
\left (\begin{matrix} 0 & abc \\
 0 & 0
\end{matrix}
\right )   
= 
\left (\begin{matrix} a & 0\\
 0 & 0
\end{matrix}
\right )
\left (\begin{matrix} 0 & b\\
 0 & 0
\end{matrix} \right ) 
\left (\begin{matrix} 0 & 0\\
 0 & c
\end{matrix}
\right )
$$
we see that 
$\theta_g$ applied to the 
first matrix has again the form of the 
first  matrix.

(i) One computes 
expressions like
$$
\left (\begin{matrix} 0 & 0\\
 \gamma_g(z) & 0
\end{matrix}
\right )
\left (\begin{matrix} \alpha_g(x) & 0\\
 0 & 0
\end{matrix} \right ) 
\quad
\left (\begin{matrix} 0 & 0\\
 0 & \delta_g(x)
\end{matrix}
\right )
\left (\begin{matrix} 0 & 0\\
 \gamma_g(z) & 0
\end{matrix} \right ) 
$$
and uses the fact that $\theta$ is a $G$-action on a ring.  

(ii) 
The $G$-action property of modules is by 
line (\ref{l2}), and the $G$-invariance 
of the functional spaces follows from 
(\ref{l4}).

(iv) By (ii), $\alpha$ and $\gamma$ are $G$-actions as claimed, so that the existence of ${\rm Ad}(\alpha \oplus \gamma)$ is by lemma \ref{lemma51}.
By relations (i) one can deduce
\begin{equation}			\label{l7}
\left (\begin{matrix} \alpha_g(x) x' & \beta_g(y)y' \\
 \gamma_g(z)z' & \delta_g(w)w'
\end{matrix} \right ) 
= 
\left (\begin{matrix} \alpha_g(x \alpha_{g^{-1}} 
(x'))  & \alpha_g(y \gamma_{g^{-1}}(y')) \\
 \gamma_g(z \alpha_{g^{-1}}(z')) & \gamma_g(w
\gamma_{g^{-1}}(w'))
\end{matrix} \right ) 
\end{equation}
for all $x,...,w' \in A$, which shows $G$-equivariance of $\chi \otimes 1$.
In fact, the second matrix line follows directly from (\ref{l2}), and the upper right corner from the first relation of (\ref{l4}).

(v) $\alpha$ and $\delta$  are determined by  $\beta$ and $\gamma$ by (\ref{l4}). 
The version $\alpha$ and $\gamma$ follows 
from (iv). 
 
(viii) By (iv), 
we can 
construct 
${\rm Ad}(\alpha \oplus \gamma)$ and aim to 
define $\theta$ by its restriction. 
To show that the image of $\chi \otimes 1$ is $G$-invariant, we consider the right hand side of (\ref{l7})
and want to construct identity with the left hand side.
For the first column this is clear 
by the $\gamma$- and $\alpha$-actions. 
For the upper right corner we use 
definition \ref{def12} that $X:=\Theta_{(A,\alpha)} ((A,\gamma)$ is $G$-invariant. So 
as left multiplication $m_y$ by $y$ is in $X$, 
$\alpha_g \circ m_y \circ \gamma_{g^{-1}} =:\beta_g(y)$ is also on $X$. 
For the lower right corner we write 
$w=w_1 w_2$ by quadratik and then 
$\gamma_g(w \gamma_{g^{-1}}(w')) 
= \gamma_g(w_1) \alpha_g(w_2 \gamma_{g^{-1}}(w'))= \gamma_g(w_1) \beta_g(w_2) w' 
=: \delta_g(w) w'$ (under 
quotation marks). 
%

(vii)
Take for example $G=\Z/2$, $A=\C$ (or any $A$), $\alpha=\delta$ the trivial action. Then $\gamma_g(x) = x$ and $\gamma_g(x) = (-1)^g x$ are two valid choices.
%
\end{proof}

\begin{corollary}[{\cite[corollary 7.5]{baspects}}]  
			\label{cor79}
Consider the double split exact sequence of definition \ref{def31}.

{\rm (i)} Then 
the ideal 
$M_2(j(B))$
is invariant under the action $\theta$.

{\rm (ii)}
The map $f \otimes 1: (M_2(M),\theta)
\rightarrow (M_2(A),\delta)$
is equivariant for the quotient $G$-action 
$\delta$ on
$M_2(A) \cong M_2(M) / M_2(j(B))$.
%
\end{corollary}

\begin{lemma}[{\cite[lemma 7.8]{baspects}}]   
	\label{lemma272}
Let $S,T$ be two $G$-actions on a 
$(B,\beta)$-module $\cale$.

Consider a diagram 
$$\xymatrix{
\calk_B(\cale)  
\ar[r] 
& 
\call_B(\cale) \square A \ar[r]    & A 
\ar@<.5ex>[l]^{s_\pm \square \idd} 
}$$
which is double split exact except that
we have not found a $M_2$-action yet.
But we know that $s_-$ is equivariant with respect to 
${\rm Ad}(S)$ on $\call(\cale)$, and $s_+$ is equivariant with respect to
${\rm Ad}(T)$ on $\call(\cale)$.

Equip $M_2(\call_B(\cale) \oplus A)
\cong \call_B(\cale \oplus \cale) \oplus M_2(A)$ with the 
$G$-action 
$${\rm Ad}(S \oplus T) \oplus (\alpha \otimes 1_{M_2}) 
$$

Then the following assertions are equivalent:
\begin{itemize}

\item[(i)]

 $s_-(a) \big( S_g T_{g^{-1}} - S_g S_{g^{-1}} \big) $ and 
 $s_-(a) \big( T_g S_{g^{-1}} - T_g T_{g^{-1}} \big)$ 
are in $\calk_B(\cale)$ for all $g \in G$, $a \in A$.


\item[(ii)]
$S_g s_-(a) T_{g^{-1}} - s_-(\alpha_g(a))$
and $T_g s_-(a) S_{g^{-1}} - s_-(\alpha_g(a))$  
are in $\calk_B(\cale)$
 for all $g \in G$, $a \in A$.

\item[(iii)]
$M_2(\call_B(\cale) \square A)$
is a $G$-invariant subalgebra.

\end{itemize}

In case that there is a 
invertible operator $U \in \call(\cale)$ such that $T_g \circ U = U \circ S_g$ for all $g \in G$, 
these conditions are
also equivalent to

\begin{itemize}

\item[(iv)]
$s_- (a) \big (g(U) - U 
\big ) \in \calk_B(\cale)$ for all $g \in G$, $a \in A$
($G$-action is ${\rm Ad}(S)$).

\end{itemize}

\end{lemma}

\begin{proof}
Analogous proof as in lemma 
\cite[lemma 7.8]{baspects}.  
One uses the fact that $S_g  k  T_{g^{-1}}$ 
is compact for compact operators $k$, 
see definition \ref{defcompact}.  
The additional occurrences of $T_g S_{g^{-1}}$
are necessary here as we have no involution as
in $C^*$-algebras. 
\end{proof}

\section{Computations with double split exact sequences}


From now on, if nothing else is said, the $M_2$-action on $M \square A$ is always understood to be of the form
$\gamma \square (\alpha \otimes 1_{M_2})$ for $G$-algebras
$(M_2(M),\gamma)$ and 
$(A ,\alpha)$, 
cf. the last paragraph of section \ref{sec6}.

Actions on $\call_B(\cale)$ 
will always be of the form ${\rm Ad}(S)$ for a $G$-action $S$ on $\cale$.

\begin{lemma}[{\cite[lemma 8.1]{baspects}}]		\label{lemma100}
Given 
a ring homomorphism 
$f:A \rightarrow B$ 
we get a double split exact sequence
$$\xymatrix{B \ar[r]^i & B \oplus A \ar[r]^g 
 & A  \ar@<.5ex>[l]^{s_\pm} }$$
with $s_-(a)=(0,a), s_+(a)=(f(a),a)$
and one has $f = s_+ \mu \Delta_{s_-}$ in $GK$. 
\end{lemma}

\begin{lemma}[{\cite[lemma 8.2]{baspects}}]  
		\label{lemma215}

Given the first line and an
equivariant 
ring homomorphism $\varphi$ 
as in this diagram it can be completed to this diagram
$$\xymatrix{
B  
\ar[r]^{i \square 0} 
& M 
\square A 
\ar[r]    & A 
\ar@<.5ex>[l]^{s_\pm \square 1} 
\\
B \ar[r]^{i \square 0}  \ar[u] &
M  \ar[r] \ar[u]^{\phi} \square X & X \ar@<.5ex>[l]^{t_\pm} \ar[u]^\varphi
}$$
such that
$\varphi (s_+ \square 1) \mu \Delta_{s_- \square 1} = t_+ \mu \Delta_{t_-}$
in $GK$.

We assume here that the $G$-action on $M_2(M \square A)$ is of the form $\theta \square (\alpha \otimes 1)$.
\end{lemma}

\begin{lemma}[{\cite[lemma 8.3]{baspects}}]		\label{lemma14} 
Every double split exact sequence as in the first line
is isomorphic to the one of the second line
as indicated in this diagram:
$$\xymatrix{
B  \ar[r]^i \ar[d]  & M  \ar[d]^\phi  \ar[r]^f & A \ar@<.5ex>[l]^{s_\pm}  
\ar[d]   \\
B  \ar[r]^j  & \call_B(B) \square  A \ar[r]  & A \ar@<.5ex>[l]^{t_\pm} 
}$$
That is, 
$s_+ \mu \Delta_{s_-} = t_+ \mu \Delta_{t_-}$ in GK.



The $G$-action on $M_2(\call_B(B) \square A)$ is of the form ${\rm Ad}(S \oplus T) \square \delta$,  
where 
$(M_2(A),\delta)$ and $(M_2(\call_B(B)), {\rm Ad}(S \oplus T))$
are 
$G$-algebras.
\end{lemma}

For useful and important lemmas with respect 
to homotopy see {\cite[lemmas 8.10, 8.11]{baspects}}, which work 
analogously also in the ring setting.

\begin{lemma}			\label{lemma101}  

Each $\Delta_t$ can be we written as a double split exact sequence in $GK^G$. More precisely, 
$\Delta_t = (1 \square 1) \Delta_{gt \square 1}$
with respect to the diagram 
$$\xymatrix{
J 
\ar[r]_j 
& M \ar@<-.5ex>[l]_{\Delta_t} 
\ar@<.5ex>[r]^g 
& X \ar[l]^t     \\
J  \ar[u]  \ar[r]^{j \square 0} 
& M \square M 
\ar[u]^p 
\ar@<.5ex>[r]^f  
& M \ar[l]^{ gt \square 1, 1 \square 1 }  \ar[u]^g 
}
$$

\end{lemma}

\begin{proof}

Note that $(1-gt)(m) \in j(J)$ 
by the definition of the first line of the above diagram so that the lower line is double 
split by remark \ref{remark43}.
  

Set $p(m \square n)= m$ and $f(m \square n)=n$. 
By lemma \ref{lemma21}.(i),  $p  \Delta_t = \Delta_{gt \square 1}$.
(Note that $M \square_{gt\square \idd} M
= M \square_{\idd\square \idd} M = \{(m +j)\oplus m| m \in M, j \in J\}$ and so 
$f (gt\square \idd)p= p gt$ on $M \square M$.) 
 Thus 
$$\Delta_t = (1 \square 1) p \Delta_t
= (1 \square 1) \Delta_{gt \square 1}$$

We set the $M_2$-action on the second line 
of the above diagram to 
$(\theta \square \theta ) \otimes 1_{M_2}$ for $(M,\theta)$ being the given $G$-action of the first line of the above diagram. 
Note that $j(B)$ is invariant under the $\theta$-action by lemma \ref{cor79}, and so the $M_2$-action of the second line is valid. 
%
\end{proof}

The next lemma shows how we may 
unitize the `starter' 
ring of a double split exact sequence.


\begin{lemma}			\label{lemma67} 
Let the left part of the first line of the following diagram be a given double split exact sequence, 
and $\pi$ be the identical embedding. 
Then 
these data can be completed to this diagram
$$\xymatrix{ B \ar[r] \ar[d] & M \square 
A  \ar[r] \ar[d]^\phi  &
(A,\alpha) \ar@<.5ex>[l]^{s_\pm}  \ar[d]^\pi  
\ar[r]^e   
&  M_\infty((A,\alpha))   
 \ar[d]^{\pi \otimes \idd}   \\  
B \ar[r] & \tilde M \square 
\tilde A \ar[r] &
(\tilde A, \tilde \alpha) \ar@<.5ex>[l]^{\tilde s_\pm}    
\ar[r]^f  
& M_\infty(
(\tilde A,\tilde \alpha))   
}$$
such that $s_+ \mu \Delta_{s_-} =  
\pi \tilde s_+ \mu \Delta_{\tilde s_-}$. 
If $M$ is unital, then one may replace $\tilde M$ by $M$ here. 

\end{lemma}

\begin{proof}

Set
$$\tilde s_\pm (a + \lambda 1_{ \tilde A})= 
s_\pm (a) + \lambda 1_{\tilde M} \square 
a + \lambda 1_{\tilde A}$$
and 
$\phi$ 
to be the identical embedding. 
The $M_2$-action of the second line of the diagram is $\tilde \theta \square \tilde \alpha \otimes 1_{M_2}$ for being it $\theta \square \alpha \otimes 1_{M_2}$ of the first line. 
Verify the claim with lemma \ref{lemma21} for $\Phi:= \phi \otimes 1_{M_2}$, thereby 
recalling remarks \ref{cor22} and \ref{remark43}. 
\end{proof}

\section{Calculations with corner embeddings}

The next proposition shows how 
compact operators on a module over 
compact 
operators may be regarded as compact operators 
of a plain module. 


\begin{proposition}			\label{lemma71}  

Let $(B,\beta)$ be a ring, $(\cale,X)$ a 
weakly cofull $(B,\beta)$-module 
and $(\calf,Y)$ a 
cofull $\calk_B\big ((\cale \oplus B,X \oplus \beta) \big)$-module.


Then 
there is a ring isomorphism
$$\pi:\calk_{\calk_B(\cale \oplus B)}(\calf) \rightarrow
\compacts B {\calf 
M_B }: \pi(T)= T 
|_{\calf 
M_B}$$
where 
$M_B \subseteq \compacts B {\cale \oplus B}$
is the subalgebra of corner multiplication operators $m_b$ defined by 
$m_b(\eta \oplus d)= 0 \oplus bd$ 
for $b,d \in B, \eta \in \cale$. 

The additive subgroup $\calf  M_B: = \{\xi m_b \in \calf | \,\xi \in \calf, m_b \in M_B\}$ of $\calf$ 
is turned into a $B$-module by setting
$\eta \cdot b := \eta m_b$ 
for $b \in B$ and $\eta \in \cale M_B$.

Turn the the right $(B,\beta) \cong (M_B, \Ad(\beta))$ module $\calf M_B$ into a cofull functional 
module by defining    
$$\Theta_{B}(\calf M_B) := \{ m \circ \phi|_{\calf M_B}| \, 
m \in M_B, \phi \in \Theta_{\calk_B(\cale \oplus \H_B)} (\calf) \}$$

The $G$-action on $\calf M_B$ is
the one induced by restriction
of $Y$. 

\end{proposition}

\begin{proof}

(a) 
We claim that for each 
$z \in A:=\calk_B(\cale \oplus B)$ 
there are $x_i \in A_1:=\calk_B(B, \cale \oplus B) \subseteq A,
y_i \in A_2:= \calk_B( \cale \oplus B,B) \subseteq A, b_i \in B$
such that $z = \sum_{i=1}^n x_i m_{b_i} y_i$. 

Indeed, let $z= \theta_{\xi,\phi}$. 
Write $\xi= \sum_{i=1}^n \zeta_i b_i d_i$ 
for $\zeta \in \cale \oplus B, b_j,d_j \in B$ 
by cofullness of $\cale \oplus B$ and quadratik of $B$. 
Set $\psi \in \Theta_B(\cale \oplus B)$ to be the canonical projection onto the coordinate $B$ (said in the sloppy sense). 
Then, for $\eta \in \cale \oplus B$, 
\begin{eqnarray*}
z(\eta) &=& \xi \phi(\eta)
= \sum_{i=1}^n \zeta_j b_j d_j \phi(\eta)
= \sum_{i=1}^n \theta_{\zeta_i,\psi} \circ m_{b_i}
\circ \theta_{(0 \oplus d_i),\phi} 
(\eta)
\end{eqnarray*}

The last formula 
also precises what is 
meant by $A_1$ and $A_2$. 
Note that 
the product $A_2 A_1 \subseteq M_B$.

(b) 
Let us view $\calf M_B$ as a $M_B$-module. 
As $m \phi(\xi n)= m \phi(\xi) n \in M_B$ for $ 
\phi \in \Theta_B(\calf), 
\xi \in \calf, m,n \in M_B$, we see that the functional $m \phi$ of the functional space of $\calf M_B$ maps into $M_B$. 
To show cofullness of $\calf M_B$ we write,
using cofullness of $\calf$ two times and
quadratik of $B$ and (a),  
$$
\xi m = \sum_i \xi_i \phi_i(\eta_i) \phi_i(\zeta_i) m 
= \sum_{i,j=1} \xi_i x_{i,j} m_{i,j} 
n_{i,j} y_{i,j}
  \phi_i(\zeta_i m)  $$ 
 for 
$\xi,\xi_i,\eta_i,\zeta_i \in\calf,  m, 
m_{i,j},n_{i,j} \in M_B, x_{i,j} \in A_1, y_{i,j} \in A_2$.

(c) 
Next we prove that $\pi(T)$ is a compact operator as claimed. 

By claim (a), the cofullness of $\calf$ and quadratik of $B$ we may write a $\xi \in \calf$ as 
$\xi = \sum_{i=1}^n \zeta_i x_i m_{b_i c_i}  y_i$ 
for $\zeta_i \in \calf, x_i \in A_1, y_i \in A_2, b_i, c_i \in B$.  

Then, for $\eta \in \calf$ and 
$\theta_{\xi,\phi} \in \calk_A(\calf)$, 
$$\theta_{\xi,\phi} (\eta m_b) 
= \sum_{i=1}^n \zeta_i x_i m_{b_i c_i} y_i \phi(\eta) m_b = \sum_{i=1}^n \zeta_i' \psi_i(\eta m_b)
$$
for $\zeta_i':= \zeta_i x_i m_{b_i}$ and
$\psi_i := m_{c_i} y_i \phi|_{\calf M_B}$. 

(d)
We claim that $\pi$ is injective. 

Using again claim (a) as before for $\xi \in \calf$,  assuming $\pi(T)=0$ 
gives 

$$T(\xi)= \sum_i T(\zeta_i x_i m_{b_i} y_i)= 
\sum_i T(\zeta_i x_i m_{b_i}) y_i 
= 0 $$



(e) We claim that $\pi$ is surjective. 

Indeed, 
given $S \in \compacts B {\calf M_B}$
we try to define its preimage  $T \in \compacts {\compacts B {\cale \oplus B}} \calf$ by
$$T \Big(\sum_{i=1}^n \xi_i  x_i m_{b_i} y_i \Big):= \sum_{i=1}^n (\sigma \circ S)(\xi_i x_i m_{b_i}) y_i$$
for all $\xi_i \in \calf, x_i \in A_1,y_i \in A_2, b_i \in B$. 
Here, $\sigma: \calf M_B \rightarrow \calf$ is the identity embedding.

If we multiply both sides of the above identity with $z= \sum_{j=1}^m p_j m_{d_j} q_j \in A$ ($p_j \in A_1, q_j \in A_2, d_j \in B$ as in claim (a)), then,  because $y_i p_j$ 
is an element in $M_B$,  
we get
$$T \Big(\sum_{i=1}^n \xi_i  x_i m_{b_i} y_i \Big) z = \sum_{j=1}^m (\sigma \circ S) \Big (\sum_{i=1}^n \xi_i x_i m_{b_i} y_i
p_j  \Big ) m_{d_j} q_j = 0$$

Because $z \in A$ was arbitrary, and 
$\compacts B {\calf 
M_B }$ is an ideal in itself by 
lemma \ref{lemma113} and (b), we conclude 
that $T$ is well-defined.
%
%
Notice that by claim (a), $T$ is fully defined. 
\end{proof}

The next 
corollary shows that the composition
of two corner embeddings is again one.

\begin{corollary}  

Let the first line of the following diagram be given,
where $e$ and $f$ are corner embeddings.
$$\xymatrix{ B \ar[r]^e 
\ar[d]^g
& \calk_B(\cale \oplus \H_B) \ar[rr]^f  & &  \calk_{\calk_B(\cale \oplus \H_B)} \Big(\calf \oplus \H_{\calk_B(\cale \oplus \H_B)} \Big) \ar[d]^\phi			\\
\calk_B(
\calz \oplus \H_B)
& & & \calk_B \Big(
 \big(\calf \oplus \H_{\calk_B(\cale \oplus \H_B)} \big ) M_B \Big)
 \ar[lll]^\psi		
}$$

Then one can make the above commuting diagram
where $\phi$ and $\psi$ are ring isomorphisms 
and $g$ is a corner embedding.
\end{corollary}   

\begin{proof}
Let $M_B \subseteq A:= \calk_B(\cale \oplus \H_B)$ be the subring of corner multiplication operators as in 
proposition \ref{lemma71}  acting on the distinguished coordinate $B$ of $\H_B$. 

Then there is a $B$-module isomorphism
($B$ identified with $M_B$), 
recalling also lemma \ref{lemma221}.(i), 
$$A M_B \cong \calk_B(B,\cale \oplus \H_B)
\cong 
\calk_B(B,\cale) \oplus \H_B
$$

Hence, using this isomorphism for the distinguished first coordinate $A$ of $\H_A$ we get a $B$-module isomorphism 
$$(\calf \oplus \H_A) M_B \cong \calf M_B
\oplus (\H_A \ominus A) M_B \oplus \calk_B(B,\cale) \oplus \H_B = \calz \oplus \H_B$$
for $\calz$ obviously defined. 
This yields $\psi$ and the above commuting diagram. 
%
%
%
%
\end{proof}

In the following lemma we show how the inverses of corner embeddings can skip ring homomorphisms.

\begin{lemma}[Cf. {\cite[lemma 8.8]{baspects}}]  
			\label{lemma43}
Let a corner embedding $e$ and
a ring homomorphism $\pi$ as in the following
diagram be given:

\begin{equation}	\label{diag}
\xymatrix{ A \ar[rr]^e \ar[d]^\pi & &  \calk_A(\cale \oplus \H_A)  \ar[d]^\phi   \\
B \ar[rr]^{f} & & \calk_B((\cale  \oplus \H_A)\otimes_\pi B \oplus \H_B)
}
\end{equation}

Then we draw the above commuting diagram 
where $f$ is a corner embedding.

In particular, $e^{-1} \pi = \phi f^{-1}$
in $GK^G$.

\end{lemma}

\begin{proof}
Define $\phi(T)=T \otimes 1 \oplus 0_{\H_B}$. 
At first we write 
$$\calk_B((\cale  \oplus \H_A)\otimes_\pi B \oplus \H_B)
\cong \calk_B(\cale \otimes_\pi B  \oplus \H_{A \otimes_\pi B} \oplus \H_B)$$
such that 
$e \phi (a)$ 
acts by $a$-multiplication on the summand $(A \otimes_\pi B,\alpha \otimes \beta)$.

By lemma \ref{lemma220}.(i) 
we have a 
$B$-module 
isomorphism
$$
X:A \otimes_\pi B \rightarrow B_0:= 
\sum 
\pi(A) B  \subseteq B
: X(a\otimes b) = 
\pi(a) b 				\\
$$
%

We have an equivariant ring homomorphism
$$h:M_2(\pi (A)) \rightarrow \calk_{B}(B_0 \oplus B) \subseteq \calk_{B}(B_0 \oplus \H_B)$$
by matrix-vector multiplication,
where the summand $B$ means here the 
distinguished first coordinate $(B,\beta)$ of $\H_B$.

(Here we use cofullness of $A$, i.e.  
$\pi(a) b= \sum_i \theta_{\pi(a_i), \pi(a_i')} 
b$.)

Note that $e \phi(a) = h(\diag (\pi(a),0))$. 
We may rotate this to $h(\diag (0,\pi(a))) = \pi f (a)$ 
by a 
rotation homotopy in the domain of $h$, 
see definition \ref{def26}.  
 
Hence we get $e \phi = \pi f$ in $GK^G$ by homotopy. 
\end{proof}

\begin{definition}
{\rm

Define $L_0 GK^G$ (`{\em level-$0$ morphisms}') to be the smallest additive subcategory of $GK^G$ generated by all ring homomorphisms. 

}
\end{definition}

So the morphisms in $L_0 GK^G(A,B)$ are exactly those 
of the form
$\pm f_1 \pm f_2 \pm \ldots \pm f_n$
for some 
ring homomorphisms $f_i:A \rightarrow B$ for all $1 \le i \le n$. 

The next lemma shows how inverses of corner embeddings can skip such morphisms: 

\begin{lemma}
If $e \in GK^G(A,D)$ is a corner embedding and $z \in L_0 GK^G(A,B)$ 
then 
there is a corner embedding $f \in GK^G(B,E)$ and a $w \in L_0 GK^G(D,E)$ such that
$$e^{-1} z = w f^{-1}$$

\end{lemma}

\begin{proof}

If $z= \pi_1 \pm 
\ldots \pm \pi_n$ then we
do the diagram (\ref{diag}) 
for each $\pi_i$ but with one constant 
corner embedding $f$ by replacing the 
lower right corner of 
diagram (\ref{diag})  by 
$\calk_B \big (\oplus_{i=1}^n (\cale  \oplus \H_A)\otimes_{\pi_i} B \oplus \H_B \big )$, which
works for all $\pi_i$.  
We obtain then by (the proof of) lemma 
\ref{lemma43} that 
$e^{-1} \pi_i =  \phi_i f^{-1}$ for all $1 \le i \le n$. 
%
\end{proof}

\section{Calculations with extended double split exact sequences}
							\label{sec9}

From now on, we make the convention that for double split exact sequences as in 
definition \ref{def31} the ring homomorphism 
 $f \otimes 1_{M_2}: (M_2(M),\theta) \rightarrow (M_2(A), \alpha \otimes 1_{M_2})$ 
is $G$-equivariant. In other words, 
$\delta = \alpha \otimes 1_{M_2}$ in 
corollary \ref{cor79}. 

\begin{definition}
{\rm 

By an {\em extended double split exact sequence } 
we mean a diagram of the form 
$$\xymatrix{
(Z,\zeta) \ar[r]^{e} & (B,\beta)  \ar[r]^j & (M,\gamma) \ar[r]^f 
&  (A,\alpha)  \ar@<.5ex>[l]^{s,t}
} $$
where $e$ is a corner embedding and the other part is a double split exact sequence as 
in definition \ref{def31}. 
The {\em associated 
morphism}  in $GK^G$ to this extended double split exact sequence is 
$s \mu \Delta_t e^{-1}$. 
}
\end{definition}

We shall also abbreviate 
$s \mu \Delta_t e^{-1}=s \nabla_{t}$ by setting $\nabla_{t}:= \mu \Delta_t e^{-1}$.  

The next lemma shows how we may add or subtract superfluous modules involved in extended 
double split exact sequences. 

\begin{lemma}    \label{lemma10}

Let the first line of the the following diagram be given and $(\calf,T)$ be any 
weakly cofull $B$-module. 
Then we can draw this diagram 
\begin{equation} 			\label{standardspe}
\xymatrix{ B \ar[r]^e \ar[d]  
&  \calk_B(\cale \oplus \H_B)  \ar[d]^\psi    \ar[r]
& \call_B(\cale \oplus \H_B) \square A  \ar[d]^{\phi}  \ar[r]  
& A  \ar@<.5ex>[l]^{s_\pm }  \ar[d] \\
B \ar[r]^-{f} & \calk_B(\cale  \oplus \calf \oplus \H_B )   \ar[r]
&
\call_B(\cale  \oplus  \calf \oplus \H_B) \square A    \ar[r] 
& A  \ar@<.5ex>[l]^-{t_\pm} 
}
\end{equation}
such that 
$s_+ \mu \Delta_{s_-} e^{-1} = t_+ \mu \Delta_{t_-} f^{-1}$. 

\end{lemma}

\begin{proof}
Set 
$\phi(T \square a)= T \oplus 0_\calf \square a$ and $\psi(T)=T \oplus 0_\calf$.   
%
Define $t_\pm = s_\pm \phi$ 
and verify the claim with lemma \ref{lemma21}. 
The $M_2$-action of the second line, 
to be checked by lemma \ref{lemma272}, 
 is 
$\Ad(S \oplus T \oplus  X\oplus T) \square \alpha$ if of the first line it is $\Ad(S \oplus X) \square \alpha$. 
Check the claim with lemma \ref{lemma14}
and recall remark \ref{cor22}.  
\end{proof}

\begin{lemma}			\label{lemma89} 

Every double split exact sequence can be turned to an extended double split exact sequence, 
even to the form (\ref{standardspe}) 
with $s_\pm$ being of ``standard form", 
i.e. $s_\pm = s_\pm' \oplus 0_{\H_B} \square \idd$.   
\end{lemma}

\begin{proof}
Every double split exact sequence 
can be turned to an extended one by 
composing with the inverse of the identity corner embedding $e: B \rightarrow B \cong \calk_B(B)$.  
But
by application of 
lemmas  \ref{lemma14} and \ref{lemma10} 
we can also achieve this for 
corner embeddings of the form 
$e: B \rightarrow B \cong \calk_B(\H_B)$ 
for bigger direct sums $\H_B$.
%
\end{proof}

\begin{lemma}
A double split exact sequence 
(\ref{standardspe}) 
can be equivalently 
turned to the same one but with 
$\cale$ replaced by $\overline \cale$. 
In other words, all involved modules 
then are  separated by the functionals. 
\end{lemma}

\begin{proof}
Do an 
analogous proof as in 
lemma \ref{lemma10}, but where in the second line 
of the diagram $\cale \oplus \calf$ is replaced 
by $\overline \cale$, and $\phi$ and $\psi$
are induced by the map of  lemma 
\ref{lemma219}.(i). 
\end{proof}


In this lemma we show how an extended double split exact sequence may absorb a ring homomorphism 
from the other side than in  lemma 
\ref{lemma215}. 

\begin{lemma}			\label{lemma83}   

Let an extended double split exact sequence 
as in the first line of the following diagram 
and
a ring homomorphism $\pi$ as in the following
diagram be given:

$$\xymatrix{ B \ar[r]^e \ar[d]^\pi &  \calk_B(\cale \oplus \H_B)  \ar[d]^{\psi}    \ar[r]
& \call_B(\cale \oplus \H_B) \square A  \ar[d]^{\phi}		\ar[r] 
& A  \ar@<.5ex>[l]^{s_\pm}  \ar[d] \\
D \ar[r]^-{f} & \calk_D \Big ((\cale  \oplus \H_B)\otimes_\pi D \oplus \H_D \Big )   \ar[r]
&
\call_D \Big((\cale  \oplus \H_B)\otimes_\pi D \oplus \H_D \Big ) \square A   \ar[r]
& A  \ar@<.5ex>[l]^-{t_\pm}    
\\
D \ar[r]^g \ar[u]  
& \calk_D  ( \cale\otimes_\pi D \oplus \H_D )   \ar[r]  \ar[u] 
&
\call_D ( \cale  \otimes_\pi D \oplus \H_D  ) \square A  \ar[u]   \ar[r]
& A  \ar@<.5ex>[l]^-{v_\pm}  \ar[u] 
}$$
%

{\rm (i)} 
Then we can complete these data to the above diagram where the second line is extended double split and
$s_+ \mu \Delta_{s_-} e^{-1} \pi =
t_+ \mu \Delta_{t_-} f^{-1}$.
 
{\rm (ii)}  
If $s_\pm$ is standard (see lemma \ref{lemma89})   
then this morphism 
is also the one of the third line, that is, 
$s_+ \mu \Delta_{s_-} e^{-1} \pi =
v_+ \mu \Delta_{v_-} g^{-1}$. 
\end{lemma}

\begin{proof}  
{\rm (i)} 
Set 
$\phi(T \square a)=T \otimes 1 \oplus 0_{\H_D} \square a$, 
$\psi(T)=T \otimes 1 \oplus 0_{\H_D}$  and 
$t_\pm = {s_\pm \phi}$. 

Because $s_+(a)-s_-(a)$ is in the ideal of compact operators by remark \ref{remark43}, 
$t_+(a)-t_-(a)$ is in the ideal of compact operators by lemma 
\ref{lemma116}.(v), and so the second line 
of the above diagram is double split. 



If the $M_2$-action of the first line of the diagram is $\Ad(S \oplus T)$, then on the second line we set it to 
$\Ad(S \otimes \delta \oplus R,  T \otimes \delta \oplus R)$, where $D= (D,\delta)$ and $\H_D=(\H_D,R)$. 

Since for all $a \in A$ 
$$s_-(a) \big ((S_g \otimes \delta_g) (T_{g^{-1}} \otimes \delta_{g^{-1}}) - 
(S_g \otimes \delta_g) (S_{g^{-1}} \otimes \delta_{g^{-1}}) \big ) \in \calk_D( \cale \otimes_\pi D)$$
by lemmas \ref{lemma272} 
and  
\ref{lemma116}.(v),  
the $M_2$-action of the second line is valid 
by lemma \ref{lemma272} again. 
The first claim is then verified by 
lemma \ref{lemma43}, and 
lemma   \ref{lemma21} with $\Phi= \phi \otimes 1_{M_2}$ and recalling remark \ref{cor22}.


(ii)
The last claim with the third line simply follows 
 from lemma \ref{lemma10} applied 
to $\calf:= \H_B \otimes_\pi D$. 
\end{proof}

\begin{definition}			\label{def84}  
{\rm
The maps $t_\pm$ of the last lemma (or the $v_\pm$ if 
applicable) 
are denoted by
$s_\pm^{\pi} := t_\pm$ (or $s_\pm^{\pi} := v_\pm$). 
So one has, under the assumptions of the last lemma,
$$s_+ \nabla_{s_-} \pi = s_+^\pi \nabla_{s_-^\pi} $$  
 }
\end{definition}

The next proposition shows 
that the inverse of a plain corner embedding can be 
fused with a double split exact sequence.

\begin{proposition}			\label{lemma85}  

Let $A$ be unital. Let the first line, an extended double split exact sequence, and the plain matrix embedding $g$ (i.e. without occurrence of $\cale$, but still with any $G$-action)  
of the 
following diagram be given. 
Thereby, the cardinalities of the index sets of the direct sums $\H_A$ and $\H_B$ may differ. 
We also need to assume that 
$s_-(1)=s_+(1)=1_\cale$. 

$$\xymatrix{
B \ar[r] 
\ar[d]   &
\calk_B(\cale \oplus \H_B) \ar[r] 
\ar[d] 
& \call_B( \cale \oplus \H_B) \square  A 
\ar[d]^\phi  \ar[r]   &
 A \ar@<.5ex>[l]^{s_\pm \oplus 0 \square 1}  \ar[d]^g    \\
B \ar[r] 
& \calk_B( H_\cale \oplus \H_B 
 ) \ar[r] 
& \call_B
\Big (  H_\cale \oplus \H_B 
\Big ) \square  \calk_A(\H_A)  \ar[r]  
&
 \calk_A(\H_A) \ar@<.5ex>[l]^-{t_\pm }  
}$$

Then we complete these data to the above diagram such that, in $GK^G$,  
$$g^{-1}  ( 
s_+ \oplus 0 \square 1)  
\nabla_{ s_- \oplus 0 \square 1} 
= t_+ \nabla_{t_-}$$  

\end{proposition}

\begin{proof}
We at first ignore any $G$-action. 

We use the 
functional $B$-module isomorphism $p_-: A \otimes_{s_-} \cale \rightarrow \cale$ defined by 
$p_-(a \otimes \xi)= s_-(a)(\xi)$, and 
its inverse 
isomorphism defined by $p_-^{-1}  (\xi)= 1 \otimes \xi$ 
of lemma \ref{lemma220}.(iv), 
  and an analogous 
$p_+$ for $s_+$. 
 
Let $\pi_\pm: H_\cale  \rightarrow \H_A \otimes_{s_\pm} \cale$ 
be the canonical isomorphisms induced by 
$p_-^{-1}$ and $p_+^{-1}$, respectively, 
where $H_\cale := \oplus_{i \in I} \cale$. 

Set
$$\phi(T \square a)= T \oplus 0 \oplus 0 \oplus  \cdots  \square g(a)$$

That means, more precisely, $T$ canonically operates on the right hand side of the identity only at the first summand  of $H_\cale$
and on $\H_B$. 

We set
$$
t_\pm(T) = \pi_\pm \circ ( T \otimes 1_\cale ) 
\circ \pi_\pm^{-1} \oplus 0_{\H_B} 
\square T \\
$$

Let us write $e_i$ for the canonical generators of $\H_A$ with entry $1_A$ in coordinate $i$ and otherwise zero, and let $T(e_i)_j \in A$ denote the $j$th coordinate of $T(e_i)$. 

Then
$$ \big (t_-(T)- t_+(T) \big )(\oplus_i \xi_i) =
\bigoplus_j \sum_i \Big ( \big (s_- - s_+ \big ) \big ( T(e_i)_j \big ) \Big )(\xi_i) 
 \square 0 = \bigoplus_j \sum_i k_{i,j}(\xi_i) 
\square 0$$

for $\xi_i \in \cale, T \in \calk_A(\H_A)$ and  
operators $k_{i,j} \in \call_B(\cale, A)$ obviously defined. 
But they are compact operators, because the first line of the above diagram is split exact. 
Hence the second line of the diagram is verified to be 
non-equivariantly double split, because $T$ is a compact operator 
and so almost all $k_{i,j}$ are zero.


If the $M_2$-action of the first line of the diagram is $\Ad (S \oplus R , T \oplus R) \square (\alpha \otimes 1)$
then on the second line it is put to 
$$\Ad ( \pi_-^{-1} \circ (V \otimes S) 
\circ \pi_- \oplus R,
\pi_+^{-1} \circ (V \otimes T) 
\circ \pi_+ \oplus R)
 \square \Ad(V)$$
where $S,T$ act on $\cale$, $R$ on $\H_B$ and $V$ on $\H_A$.

Since $(\H_A,V)= (\H_A, \alpha \oplus W)$ by definition \ref{def21}, the 
map $\Phi:=\phi \otimes 1_{M_2}$ 
between the $M_2$-spaces of the first and second line of the above diagram 
 is seen to be $G$-equivariant, because 
$\pi_-^{-1} \circ (\alpha \otimes S) \circ \pi_- = S$ on the first summand $\cale$ 
of $H_\cale$. 
Also recall the formulas of lemma 
 \ref{lemma51}. 

By lemma \ref{lemma272}, applied to the first 
line of the above diagram, we have 
$X_g := T_g T_{g -1} 
-  T_g S_{g -1} \in \calk_A(\cale)$. 

If we multiply 
\begin{equation}   \label{e33}
\pi_-^{-1} \circ (V \otimes S_g) 
\circ \pi_- \circ 
\pi_+^{-1} \circ (V \otimes T_{g^{-1}}) 
\circ \pi_+  
\end{equation}
from the right hand side with the compact 
operator 
$X_g=
\pi_+^{-1} \circ (1 \otimes X_g)  \circ \pi_+$ 
then we get (\ref{e33}) minus  (\ref{e33}) again, but where $T_g^{-1}$ is replaced by 
$S_g^{-1}$. 

By lemma \ref{lemma272} this verifies  
the validity of the $M_2$-action of the second line of the above diagram.

Set $\tilde s_\pm := s_+ \oplus 0 \square 1 $. 
Notice that we have $\tilde s_\pm \phi = g t_\pm$. 
 Finalize the proof by checking 
lemma   \ref{lemma21} with $a:=g, b:=\idd_B,  
 \Phi:=\phi \otimes 1_{M_2}$, $\phi:=\psi:=\phi$, to
obtain $\tilde s_+ \nabla_{\tilde s_-} = g t_+ \nabla_{t_-}$. 
%
%
%
\end{proof}

This is the counterpart to {\cite[lemma 13.1]{baspects}}, now with another proof, as the cited one would not work in our ring setting: 

\begin{corollary}

The inverse of a  plain matrix embedding 
$e: (A,\alpha) \rightarrow  M_\infty 
((A,\alpha))$ 
%
can 
be fused with an extended double split exact 
sequence, i.e. if $s_+ \nabla_{s_-}$ is given then there is $t_+ \nabla_{t_-}$ such that $e^{-1} s_+ \nabla_{s_-} 
= t_+ \nabla_{t_-}$. 

\end{corollary}

\begin{proof}

We do the diagram of lemma \ref{lemma67}, such
that 
$$e^{-1} s_+ \nabla_{s_-}
= (\pi \otimes 1) f^{-1} s_+ \nabla_{s_-}
= (\pi \otimes 1) t_+ \nabla_{t_-}
= v_+ \nabla_{v_-}
$$

where the last two identities are by 
proposition \ref{lemma85}  and 
lemma \ref{lemma215}. 
%
%
%
\end{proof}

\begin{lemma}[{\cite[lemma 14.1]{baspects}}]  
			\label{lemma191}
Let two 
extended double split exact sequences 
like in line (\ref{standardspe}) 
be given, say for actions $S,T$ and homomorphisms
$s_\pm, t_\pm$. 
Then we can sum up these diagrams to 
$$ 
\xymatrix{
B \ar[r] &
\calk_B( \cale \oplus \calf  \oplus \H_B 
) \square A  
%
\ar[r] 
& 
\call_B \big ( (\cale ,S) \oplus (\calf \oplus \H_B ,T) \big ) \square A \ar[r]    & A
\ar@<.5ex>[l]^-{s_- \oplus t_-, 
s_+ \oplus t_+ } 
}
$$ 
and this corresponds to the sum of the associated elements in $GK^G$, i.e.
$$s_+ \mu \Delta_{s_-} e^{-1}+
t_+ \mu \Delta_{t_-} e^{-1}
=
(s_+ \oplus t_+) \mu \Delta_{s_- \oplus t_-} e^{-1}
$$

The $M_2$-action is ${\rm Ad}(V \oplus W) \square (\alpha \otimes 1)$
for the two $M_2$-actions 
${\rm Ad}(V) \square (\alpha \otimes 1)$
and 
${\rm Ad}(W) \square (\alpha \otimes 1)$
of the given diagrams.  

(Note that we used the wrong but suggestive notation $s_- \oplus t_- := x \oplus y \square 1$ for $s_- = x \square 1, t_-=y \square 1$. 
Also, one has actually different $e$s.)

\end{lemma}

\begin{corollary}[{\cite[corollary 14.2]{baspects}}]  
					\label{cor142}
Consider the 
extended double split exact sequence of line 
(\ref{standardspe}). 
Make its `negative' diagram
where we exchange $s_-$ and $s_+$ and
transform the $M_2$-action under coordinate flip.
Then its associated element in $GK^G$ is 
the negative, that is, 
$$- s_+ \mu \Delta_{s_-} e^{-1} =
s_- \mu \Delta_{s_+} e^{-1}$$

\end{corollary}


\begin{lemma}			\label{lemma88}  
Every generator morphism (e.g. $\varphi,f^{-1},\Delta_s$) of $GK^G$ can be expressed as 
an extended double split exact sequence. 
\end{lemma}

\begin{proof}

By lemmas \ref{lemma100} and \ref{lemma101} 
we can write ring homomorphisms $\pi$ and 
synthetical splits $\Delta_s$ as double 
split exact sequences, 
and to them we apply lemma \ref{lemma89}.
For a corner embedding we write 
$e^{-1}=1 e^{-1} $ 
for the double split exact sequence $1$ 
by lemma \ref{lemma100}. 
%
%
\end{proof}

\section{$\nabla$-Calculation}
						\label{sec10}

The idea of this section is to take 
an arbitrarily complicated expression $x$ in $GK^G$, and 'solve' (decide the truth of) the equation $x=0$ in $GK^G$ by equivalently reformulating it as 
$y=0$ in $GK^G$, where $y$ is a level-0 morphism .

\begin{definition}
{\rm 

A {\em term of $N$ elements} in a ring (or additive category) $A$ is an 
expression in 
$A$
using plus, minus and multiplication signs, and no brackets and 
$N$ elements $x_1,...,x_N$ (copies are counted again). 
}
\end{definition}

For example, $
x_1 x_4 + x_2 x_3 x_4 - x_5 x_6 x_7 + x_1 x_7 x_3 x_5$ is a term of 12 elements.

\begin{lemma}			\label{lemma92}  
Every morphism $z$ in $GK^G$ can be expressed as a term of extended double split exact sequences. 
This term may even be chosen to be the product
$z= x_1 x_2 \ldots x_n$
of $n$ extended split exact sequences $x_1, \ldots , x_n$.
\end{lemma}

\begin{proof} 
The first assertion follows simply by lemma 
\ref{lemma88} by expressing $z$ as a term of generators and then expressing all those by extended double split exact sequences. 
 
By lemma \ref{lemma88} write any identity ring homomorphism $1_A$ as an extended split exact sequence and then form its negative 
by corollary \ref{cor142} 
such that 
$-1_A$ is an extended double split exact sequence, so a word. 
In this way we may rewrite any morphism $z$ 
as a positive sum $+w_1+w_2 + \ldots +w_n$ 
of words $w_i$ in $GK^G$. 
By \cite[lemma 3.6]{bremarks} such a morphism can be expressed as a single words $w$.
Now apply lemma 
\ref{lemma88}. 
 %
\end{proof}

\begin{definition}
{\rm
A morphism $P \in GK^G(A,B)$ is called {\em right-invertible in $GK^G$} 
(or 
{\em $GK^G$-right-invertible}) if there is a morphism $Q \in GK^G(B,A)$ such that 
$P Q = 
1 \in GK^G(A,A)$. 
}
\end{definition}

\begin{lemma}				\label{lemma94}  
Let 
$$x= s_+ \nabla_{s_-} =s_+ \mu \Delta_{s_-} e^{-1} 
= s_+ e_{22} e_{11}^{-1} \Delta_{s_-} e^{-1}$$ 
be the associated morphism of 
an extended double split exact sequence, where $\mu = e_{22} e_{11}^{-1}$. 
Then there is a ring homomorphism $P:B \rightarrow D$, which is right-invertible in $GK^G$, such that 
$$x 
P = s_+ e_{22} - s_- e_{11} $$  
in $GK^G$. 
Actually, $P= e j_{s_-} e_{11}$ 
and $D=M_2(M)$  
%
%
is the `$M_2$-space' of $x$. 

\end{lemma}
  
\begin{proof}
Indeed, $P= e j_{s_-} e_{11}$ is a ring homomorphism with 
right inverse $e_{11}^{-1} \Delta_{s_-} e^{-1}$.  

Consider the valid ring homomorphism $f \otimes 1_{M_2}: D \rightarrow (M_2(A),\alpha \otimes 1_{M_2})$ by corollary \ref{cor79} 
and the convention for 
double split exact sequences
that $\delta=\alpha \otimes 1$. 
Then 
$e_{11} (f \otimes 1) = f E_{11}$ for the corner embedding $E_{11}:A \rightarrow M_2((A,\alpha))$. 
Consequently, $e_{22} e_{11}^{-1} f= f' E_{22} E_{11}^{-1} = f'$, the last identity by a rotation homotopy, 
and where $f'=f$ non-equivariantly, but
with a different source coming from the lower 
right corner of $D$.  
Consequently, 
$$s_+ e_{22} e_{11}^{-1} \Delta_{s_-} e^{-1}
\cdot e j_{s_-} e_{11} 
= s_+ e_{22} e_{11}^{-1} (1- f s_- ) e_{11}
= s_+ e_{22} - s_- e_{11}$$


\end{proof} 

\begin{definition}
{\rm
For $n \ge 1$ 
define $L_n GK^G(A,B)$ 
(`{\em level-$n$ morphisms}')
to be the 
abelian subgroup 
of the abelian group $GK^G(A,B)$ 
%
which is generated by 
(so $\Z$-linear span of) all 
elements which are 
expressible as 
a product $x_1 \ldots x_n$ 
of 
$n$ 
extended double split exact sequences 
$x_1, \ldots , x_n$ 
 in $GK^G$. 

}
\end{definition}

As $1 \in GK^G(B,B)$ is expressible as a double split exact sequence by lemma   
\ref{lemma88}, it is clear that 
a level-$n$ morphism $x_1 \ldots x_n$ can be written as a level-$(n+1)$ morphism $x_1 \ldots x_n 1$, so that we 
get a filtration of subgroups 
$$L_0 GK^G(A,B) \subseteq L_1 GK^G(A,B) 
\subseteq L_2 GK^G(A,B) \subseteq \ldots $$
of $GK^G(A,B)$ 
whose union is the whole group $GK^G(A,B)$  
by lemma \ref{lemma92}.

\begin{lemma}
If $L_1 GK^G(A,B)$ is an infinite set
then the cardinalities of $L_1 GK^G(A,B)$
and $GK^G(A,B)$ coincide.


\end{lemma}

\begin{proof}
By the last filtration this is clear 
as $L_n GK^G(A,B)$ has only at most 
countably many times the number of elements of $L_1 GK^G(A,B)$. 
\end{proof}


\begin{lemma}			\label{lemma97} 
Let $z$ be a morphism in $GK$ which is 
a term of $N$ extended split exact sequences.

{\rm (i)} 
Then there is a ring homomorphism $P$, 
which is right-invertible in $GK^G$, 
such that $z P$ 
is a term of $N-1$ extended split exact sequences.

{\rm (ii)}
There is also a ring homomorphism $P$, right invertible in $GK^G$,
such that $zP$ is a (level-0) morphism in $L_0 GK^G$.

\end{lemma}

\begin{proof}

(i) 
Consider at first a term $z=X_1 \cdots X_N$ which is just a product of extended split exact sequences $X_i$.

Say, $X_N = s_+ \mu \Delta_{s_-} e^{-1}$, where $\mu = e_{22} e_{11}^{-1}$.  
Set $P=e j_{s_-} e_{11}$.  
Then  by lemma \ref{lemma94}  we get 
$X_N P = s_+ e_{22} - s_- e_{11}$. 
   

If $N=1$ we are done. Otherwise 
proceed by writing
$$X_1 \ldots X_{N-2} X_{N-1} (s_+ e_{22} - s_- e_{11})
= X_1 \ldots X_{N-2} (Y - Z) = X_1 \ldots X_{N-2} W$$
with lemmas \ref{lemma83}, \ref{lemma191} and corollary \ref{cor142},   
where $Y,Z,W$ are extended double split exact sequences.
We have achieved that
$z P$ 
is a product of just $N-1$ extended double split exact sequences.

By lemma \ref{lemma92} this is already enough, but let us show that 
if we have an arbitrary term, say $z=X_1 \cdots X_N + Y_1 \cdots Y_n$ we may do it analogously.
We multiply it by $P$ to eliminate $X_N$
and fuse $Y_n P$ to a single double split exact sequence $Y_n'$ by lemma \ref{lemma83}.  
So we get $z P= X_1 \cdots X_{N-2} W + 
Y_1 \cdots Y_{n-1} Y_n'$. 

(ii) 
Applying (i) successively $N$ times we achieve the last claim.  
So we set $P=P_N \cdots P_1$, where $P_N,\ldots , P_1$ are the ring homomorphisms repeatedly chosen by (i). 
%
\end{proof}

\begin{theorem}			\label{theorem108}

Let $F:GK^G \rightarrow D$ be a (not necessarily additive) functor into a category $D$. 
Then $F$ is 
faithful if and only if $F$
is 
faithful on the subcategory 
$L_0 GK^G$.
\end{theorem}			

\begin{proof}
Let $z,w \in GK^G(A,B)$ be given such that $F(z)=F(w)$.  
Choose a ring homomorphism $P$ and a right-inverse $Q$ in $GK^G$ such that both 
$zP,wP$ are level-0 morphisms 
by lemma \ref{lemma97}.(ii). 
(Say $P=P_1 P_2$ are successively chosen such that $z P_1$ and $w P_1 P_2$ are level-0.)

Then $F(z)F(P)=F(z P)=F(w P)$. 
Hence, by 
faithfulness of $F$ on $L_0 GK^G$ 
we get $z P = w P$. Thus $z P Q =z = w$.
%
\end{proof}

\begin{example}
{\rm 

We ask whether we could turn the product of two 
extended double split exact sequences $s_+ \nabla_{s_-}$ and $t_+ \nabla_{t_-}$ 
to 
one extended double split exact sequence 
$u_+ \nabla_{u_-}$ and make the following 
ansatz and equivalent reformulations to this end	
(here and below, $e_+:=e_{22}, e_-:=e_{11}, E_+:= E_{22}$, etc. for corner embeddings $e,E,F$ into $M_2$-spaces involving $\mu$), 
\begin{eqnarray*}
s_+ \nabla_{s_-} t_+ \nabla_{t_-} 
- u_+ \nabla_{u_-} 
	&=& 0 
 \quad 
\cdot P:=e_{t_-} j_{t_-} e_{+}    \\  
s_+ \nabla_{s_-} (t_+ e_+ - t_- e_-) 
- u_+ \nabla_{u_-}  P 
	&=& 0  \quad 
  \mbox{by lemma \ref{lemma94}}		\\  
s_+^{t_+ e_+} \nabla_{s_-^{t_+ e_+}}  - 
s_+^{t_- e_-} \nabla_{s_-^{t_- e_-}} 
- u_+^{P} \nabla_{u_-^{P}} 
	&=& 0  \quad 
  \mbox{by lemma \ref{lemma83}, \ref{def84}} 		\\  
s_+^{t_+ e_+} \nabla_{s_-^{t_+ e_+}}  + 
s_-^{t_- e_-} \nabla_{s_+^{t_- e_-}} 
+ u_-^{P} \nabla_{u_+^{P}} 
	&=& 0  \quad 
 \mbox{by corollary \ref{cor142}}		\\  
%
(s_+^{t_+ e_+} \oplus {s_-^{t_- e_-}} 
\oplus u_-^{P}) \nabla_{{s_-^{t_+ e_+}} \oplus 
s_+^{t_- e_-} \oplus u_+^{P}}  
	&=& 0 
\quad  \mbox{by lemma \ref{lemma191}}		\\  
\big (s_+^{t_+ e_+} \oplus {s_-^{t_- e_-}} 
\oplus u_-^{P} \big ) E_+ - \big ( {{s_-^{t_+ e_+}} \oplus 
s_+^{t_- e_-} \oplus u_+^{P}} \big ) E_-  
	&=& 0 
\quad   \cdot P' \mbox{ and then \ref{lemma94}}  
\end{eqnarray*}

All steps done above can be reversed, as $P Q=1$ for a right-inverse $Q$ in $GK^G$, and so
all identities are equivalent.
Note that the last line is an identity 
in $L_0 GK^G$. 

}
\end{example}

\begin{example}				\label{example2}
{\rm 

The above pattern repeats. 
Let us look at another example. By similar 
arguments as above we have  
\begin{eqnarray*}
s_+ \nabla_{s_-} t_+ \nabla_{t_-} 
u_+ \nabla_{u_-}   &=& 0			\qquad \Leftrightarrow   
\\  
s_+ \nabla_{s_-} \Big ( (t_+^{u_+ e_+} \oplus 
{{t_-}^{u_- e_-}}) E_+ - ({{t_-}^{u_+ e_+}
\oplus t_+^{u_- e_-}})  E_- \Big ) 
&=& 0		\qquad \Leftrightarrow 
\\ 
\Big (
s_+^{(t_+^{u_+ e_+} \oplus 
{{t_-}^{u_- e_-}}) E_+}
\oplus 
{
s_-^{({{t_-}^{u_+ e_+}
\oplus t_+^{u_- e_-}})  E_- }
} \Big ) 
F_+	&&  \\
-
\Big (
{{s_-}^{(t_+^{u_+ e_+} \oplus 
{{t_-}^{u_- e_-}}) E_+}}
\oplus
s_+^{ ({{t_-}^{u_+ e_+}
\oplus t_+^{u_- e_-}})  E_-}
\Big)
F_-
&=& 0 
\end{eqnarray*}

%

}
\end{example}

\begin{lemma}
Let $A$ and $B$ be given rings. 
Suppose that $GK^G(A,B)$ 
is a countable abelian group.
Then there 
are rings $D_n$ ($n\ge 1$)  
and 
abelian group  homomorphisms 
$$\xymatrix{
GK^G(A,B) \ar[r]^-\varphi &
\varinjlim_{n 
\to \infty} L_0 GK^G(A,D_n) 
\ar[r]^f &  L_0 GK^G(A, 
\varinjlim_{n \rightarrow \infty} D_n)
}$$
%
where $\varphi$ is injective. 

\end{lemma}

\begin{proof} 
(a)
Write $X:= GK^G(A,B) = \{x_1,x_2,x_3,\ldots\}$,
where $x_i \in X$.

Set $D_0:= B$ and $P_0:= Q_0:= 1 \in GK^G(B,B)$. 
By induction by $n \ge 0$, we assume that we have already chosen 
rings $D_n$ and ring homomorphisms $P_n$
$$\xymatrix{ D_0 \ar[r]^{P_0} & 
D_1  \ar[r]^{P_1} & D_2  \ar[r]^{P_2} 
& \ldots  & \ar[r]^{P_n} & D_n   }$$
and right-inverses $Q_k \in GK^G(D_{k+1},D_k)$
for $P_k$, that is 
$P_k Q_k = 1$ in $GK^G$, for all $1 \le k \le n$ 
such that for $V_n:= P_1 P_2 \ldots P_n 
\in GK^G(B,D_n)$ we have 
%
%
$$x_1 V_n ,\ldots x_n V_n \in L_0 GK^G(A,D_n).$$

Then by lemma \ref{lemma97}.(ii) 
 select a  ring homomorphism $P_{n+1} : D_n \rightarrow D_{n+1}$ 
%
and a right-inverse $Q_{n+1} \in GK^G(D_{n+1},D_n)$ for it such that 
$x_{n+1} V_{n} P_{n+1} \in L_0 GK^G(A, D_{n+1})$. 
This completes the induction step. 

(b) 
Form the abelian groups inductive limit
$M:= 
\varinjlim_{n \rightarrow \infty} L_0 GK^G(A,D_n)$ 
induced by the group homomorphisms $(P_n)_*$. 
%
Let $\varphi_n: L_0 GK^G(A,D_n) \rightarrow M$ be the standard maps satisfying $(P_{n+1})_* \varphi_{n+1} = \varphi_n$. 

Define the desired $\varphi$ on the generating set by
$\varphi(x_n)= \varphi_n(x_n V_n)$.  
 
We have $\varphi(x_n)= 
\varphi_{n+1}(x_n V_n P_{n+1}) = 
\varphi_{n+1}(x_n V_{n+1})$,
whence a definition $\varphi(x_n+x_m):= \varphi(x_n)+\varphi(x_m)$ becomes well-defined.
Injectivity: As the connecting maps $(P_n)_*$ of the direct limit of $M$ are right-invertible by $(Q_n)_*$, and so injective, the maps $\varphi_n$ are injective. 
Thus, if $\varphi(z)=\varphi_n(zV_n)=0$ then 
$z V_n=0$, then $z= V_n Q_n \ldots Q_1 = 0$.

(c) 
Form the direct limit ring
$D:= \stackrel{\rightarrow}{\lim}_{n \rightarrow \infty} D_n$
with inductive limit connection ring homomorphism maps $P_n : D_n \rightarrow D_{n+1}$.
Let $f_n:D_n \rightarrow D$ the standard ring homomorphisms satisfying $P_n f_{n+1}= f_n$ associated to this direct limit.

Then there is a group homomorphism $f:M \rightarrow GK^G(A,D)$ induced by the maps $(f_n)_*:GK^G(A,D_n) \rightarrow GK^G(A,D)$.    
%
%
%
%
%
\end{proof}

\begin{corollary}					\label{lemma98}  
If $H \subseteq GK^G(A,B)$ 
is a finitely generated abelian subgroup
then there is a $GK^G$-right-invertible ring homomorphism $P: B \rightarrow D$ such that 
$$P_* : H \rightarrow L_0 GK^G(A,D):
P_*(z)= zP$$ is an injective group homomorphism. 
\end{corollary}

\begin{proof}
Assume that $H$ is the $\Z$-linear span of finitely many generators  $\{x_1, \ldots ,x_n\}  \subseteq GK^G(A,B)$. 
Do the induction as in the last proof 
and then set the required $P:=V_n$.

\end{proof}

\section{Tensor product Functor}   
\label{section10} 
				\label{sec11}

Let $(\F,\triv)$ be a 
field. 
In this section, all rings are assumed 
to  be  actually algebras over $\F$, 
so the object class of $\ring G$ 
 consists of $\F$-algebras only.  
This affects also all constructions: 
all ring homomorphisms are assumed to be actually $\F$-algebra homomorphisms. 
All tensor products are $\F$-balanced. 
All modules are $\F$-vector spaces, and
all module homomorphisms are $\F$-linear,
etc.

This in particular affects the homotopy axiom 
of $GK$-theory, as the involved exterior tensor product is now $\F$-balanced. 

We still shall notoriously say `ring' and `ring homomorphism' but mean `algebra' and `algebra homomorphism' 
to be in notation in accordance with 
everything 
said so far.

\begin{lemma}		\label{lemma1323} 
The $\F$-algebra homomorphism (\ref{eqh}) restricts 
to an 
isomorphism 
$\calk_A(\cale 
) \otimes \calk_B(\calf)
\rightarrow 
\calk_{A \otimes B}(\cale \otimes \calf)  
$
and 
injection 
$ \call_A(\cale 
) \otimes \call_B(\calf)
\rightarrow 
\call_{A \otimes B}(\cale \otimes \calf)  $.

\end{lemma}

\begin{proof}

The injection (\ref{eqh}) is well known from algebra. 

As $\pi(\theta_{\xi,\phi} \otimes \theta_{\eta,\psi}) = \theta_{\xi \otimes \eta, \phi \otimes \psi}$, we see the bijection with respect to the compact operators, see definition \ref{def19}. 
\end{proof}

\begin{definition}
{\rm 

Let $(D ,\delta)$ be a ring (= $\F$-algebra). 
Define an additive functor 
$$\tau^G : GK^G \rightarrow GK^G$$ 
by
\begin{itemize}

\item
For an object $(A,\alpha)$ in $GK^G$ 
set $\tau^G((A,\alpha)):= (A \otimes D, \alpha \otimes \delta)$. 

\item
For an equivariant ring homomorphism $\pi:A \rightarrow B$
set $\tau^G(\pi) = \pi \otimes 1$. 

\item
For a corner embedding $e \in GK^G$
define
$$\tau^G (e^{-1}) := 
\Big (\tau^G (e) \Big )^{-1}$$

\item
For a split exact sequence as in (\ref{splitexact}) 
set $\tau^G(\Delta_s) := \Delta_{\tau^G(s 
)}$   
with respect to the following split exact sequence
$$\xymatrix{ 0 \ar[r]  & B \otimes D  \ar[r]^{
\tau^G(j) }  
 & M \otimes D  \ar[r]^{\tau^G( f) }  
& A  \otimes D \ar[r]  \ar@<.5ex>[l]^{\tau^G( s) }     
& 0
}
$$

\end{itemize}

}
\end{definition}

Now   
$\tau^G(e) = e \otimes 1$ is indeed invertible as it
is essentially a corner embedding 
by lemmas \ref{lemma1172} and \ref{lemma1323}: 
$$\xymatrix{  A  \otimes D \ar[r]^-{e \otimes 1}  
 & \calk_A(\cale \oplus \H_A)  \otimes \calk_D(D)  \ar[r]^\cong 
& 
\calk_{A \otimes D} (\cale \otimes D \oplus \H_{A   \otimes D} )
}$$

\section{Descent functor}  
						\label{sec12}

In this section we are going to define a 
descent functor in analogy to 
the descent homomorphism 
  \cite[Theorem 3.11]{kasparov1988}. 

\begin{definition}
{\rm 

Let $(A,\alpha)$ be a $G$-equivariant ring. 
Define the {\em crossed product} $A \rtimes_\alpha G$ to be 
the set of functions $f:G \rightarrow A$ with finite support, with pointwise addition. 
Such an $f$ is written as a formal sum 
$f =:\sum_{g \in G} f(g) \rtimes g$. 

A multiplication in $A \rtimes G$ is declared 
by convolution: 
$(\sum_{g \in G} a_g \rtimes g) 
(\sum_{h \in G} b_h \rtimes h) :=
\sum_{g,h \in G} a_g \beta_g(b_h) \rtimes g h$,
which turns it to a non-equivariant ring. 

}
\end{definition}

Note that if a $A$ is quadratik then also its crossed  product.  
Indeed, if $a = \sum_i b_i c_i$ in $A$ 
then 
$ a \rtimes g= \sum_i (b_i \rtimes g) (g^{-1} (c_i) \rtimes 1)$.

\begin{definition}
{\rm 

Define an additive functor, called 
{\em descent functor},  
$$j^G : GK^G \rightarrow GK$$ 
by
\begin{itemize}

\item
For an object $(A,\alpha)$ in $GK^G$ 
set $j^G((A,\alpha)):= A \rtimes_\alpha G$. 

\item
For an equivariant ring homomorphism $\pi:(A,\alpha) \rightarrow (B,\beta)$
set $j^G(\pi):A \rtimes_\alpha G \rightarrow 
B \rtimes_\beta G$ to be the non-equivariant ring homomorphism
defined by
$j^G(\pi) (a \rtimes g):= \pi(a) \rtimes g$.


\item
For a corner embedding $e \in GK^G$
define
$$j^G (e^{-1}) := 
\Big (j^G (e) \Big )^{-1}$$

\item
For a split exact sequence as in (\ref{splitexact}) 
set $j^G(\Delta_s) := \Delta_{j^G(s)}$ 
with respect to the following split exact sequence
$$\xymatrix{ 0 \ar[r]  & B \rtimes_\beta G  \ar[r]^{j^G(j)}
& M \rtimes_\gamma G  \ar[r]^{j^G(f)}
& A \rtimes_\alpha G  \ar[r]  
\ar@<.5ex>[l]^{j^G(s)} 
& 0
}
$$

\end{itemize}

}
\end{definition}

Instead of $j^G(m)$ we often also write $m \rtimes \idd$ for morphisms $m$ in $GK^G$. 
 
All relations of $GK^G$ are effortlessly seen to go through the descent functor, so it is  well defined, excepting the corner embedding. 
Here, 
the 
corollary \ref{lemma121} below 
%
shows that 
$j^G(e)$ is indeed invertible for a corner embedding $e$.  

\begin{proposition}		\label{lemma122}

For any cofull $(B,\beta)$-module $(\cale,S)$ 
there is a non-equivariant ring isomorphism 
$$
\sigma: \calk_B(\cale ) \rtimes_{\Ad(S)} G 
\rightarrow  
\calk_{B \rtimes_\beta G} \Big(\cale \otimes_B (B \rtimes_\beta G) \Big) 
$$
\end{proposition}			

\begin{proof}

Equip the ring $R:=B \rtimes_\beta G$ with the trivial 
$G$-action and the right $(R,\triv)$-module $R$ 
with the $G$-action 
$$V_g( b \rtimes h) = \beta_g(b) \rtimes g h$$ 

Define now 
the indicated ring homomorphism by  
\begin{equation}  \label{dsigma} 
\sigma(T \rtimes g)= (T \otimes 1) \circ (S_g \otimes V_g) 
\end{equation}
where the right hand side are operators acting on 
$\cale  \otimes_B (B \rtimes_\beta G)$. 

By cofullness of $\cale \otimes_B B$ 
by lemma  \ref{lemma116}.(iv) 
we 
may express $\xi = \sum_i (\eta_i \otimes b_i)
(\phi_i \otimes \psi_i) (
\xi_i \otimes c_i)
$
and so
$$\xi \rtimes g = \sum_i (\eta_i \otimes (b_i \rtimes 1))
(\phi_i \otimes (\psi_i \rtimes 1)) (
\xi_i  \otimes (c_i \rtimes g) )
$$ 
for $\xi,\eta_i,\xi_i  \in \cale, b_i,c_i \in B$,
which shows cofullness of $\cale \otimes_B (B \rtimes G)$. 

For computing the surjectivity of $\sigma$, 
take 
$\phi \in \Theta_B(\cale)$ and $\psi \in \Theta_{B \rtimes_\beta G}(B \rtimes_\beta G)$.


Suppose $\psi(z)= w z$ for some 
$w \in R$. Let us say $w = \sum_{n=1}^N a_n \rtimes y_n$. 

For $\theta = (\psi \otimes_{B \rtimes_\beta G}  \phi)$, $\xi, E \in \cale$ and $x,F \in B \rtimes_\beta G$ we get 
\begin{eqnarray*} 
\xi \otimes x \;  \theta(E \otimes F)
& =& \xi \otimes x \; \psi \Big (\pi\big(\phi(E)\big) F \Big) \\
&=&  \sum_n \xi \otimes  x \; a_n y_n \; \pi\big(\phi(E)\big) F \\ 
&=& \sum_n \xi    \beta_x  (a_n) \beta_{x y_n} \big(\phi(E)\big) \otimes x y_n F \\
&=& \sum_n \Big (S_{xy_n} \otimes V_{x y_n}  \Big )  \Big ( S_{(xy_n)^{-1}} \big (\xi    \beta_x  (a_n) \big)  \phi(E) \otimes  F \Big )  \\ 
&=& \sum_n (S_{xy_n} \otimes V_{x y_n}) \circ \big ( T_n \otimes  1 \big ) (E \otimes F)   \\ 
&=& \sum_n \sigma \Big ( {xy_n}(T_n) \rtimes x y_n \Big ) (E \otimes F) 
\end{eqnarray*}

for the compact operators $T_n \in \calk_B(\cale)$ 
given by 
$$T_n(E) = S_{(xy_n)^{-1}} \big (\xi    \beta_x  (a_n) \big)  \phi(E)$$ 

For the other set inclusion, to show that $\sigma$ maps really into the indicated ring, 
we read the above computation from the bottom to 
the top, by setting $N=1, y_1=1$, 
and letting $T_1$ according to the formula 
freely given so to say, and thus we may also replace it by $T_1= x^{-1}( T)$, 
and by cofullness of $\cale$ we achieve it
for every $T$. 

To prove 
injectivity of $\sigma$,		
assume that  
%
$$\sigma \Big (\sum_{g \in G} T_g \otimes g \Big ) \big (\xi \otimes (b \rtimes 1) \big )
= \sum_g T_g S_g(\xi) \otimes (\beta_g(b) \rtimes g)=0$$

As $\cale \otimes_B (B \rtimes G) \cong \oplus_{g \in G} \cale \otimes_B B \cong \oplus_{g \in G} \overline \cale$ as abelian (additive) groups by lemma \ref{lemma219}.(iii), 
we get 
$$T_g S_g(\xi b)=0 \qquad \forall g \in G$$
in $\overline \cale$. 
By cofullness of $\overline \cale$ 
inherited from $\cale$,  
we 
conclude $T_g=0$.
\end{proof}

\begin{corollary} 			
			\label{lemma121}
Let $e:B \rightarrow \calk_B \big ( (\cale \oplus \H_B ,S) \big) 
$ be an equivariant corner embedding.

Then there is a commuting diagram of non-equivariant ring homomorphisms
$$\xymatrix{ 
\calk_B(\cale \oplus \H_B) \rtimes_{\Ad(S)} G \ar[r]^\sigma  &
\calk_{B \rtimes_\beta G} \Big((\cale  \oplus \H_B) \otimes_B (B \rtimes_\beta G) \Big) 
\ar[d]^\pi \\
B \rtimes_\beta G  \ar[r]^f  \ar[u]^{e \rtimes 1} &
\calk_{B \rtimes_\beta G} \Big(\cale   \otimes_B (B \rtimes_\beta G) 
\oplus \H_{B \rtimes_\beta G} \Big)
}$$
where $\sigma$ is an isomorphism, $\pi$ is the obvious canonical
isomorphism
by exchanging direct sums with the tensor product, and $f$ is a corner embedding.

In particular, $e \rtimes \idd = f \pi^{-1} \sigma^{-1}$ is 
a corner embedding up to isomorphism.  
\end{corollary}			

\begin{proof}

It is easy to realize by definition 
(\ref{dsigma}) of $\sigma$ 
that the diagram commutes. 
Recall also lemma \ref{lemma220}.(iii). 
\end{proof}

\section{Exactness}

The title of this section may be promising too much as we only have little observation with respect to exactness. 
Recall that $KK$-theory for $C^*$-algebras is only known to be half-exact if there is a split by a homomorphism or weaker by a completely positive linear map, see \cite{kasparov1981} and \cite{skandalis}.  
Exactness may even fail if there is not such a split. 

The next lemma shows how exactness of the contravariant functor $GK^G(-,X)$ can be deduced from exactness of the contravariant functor $L_0 GK^G(-,X)$.

\begin{lemma}

Given a short exact sequence 
(or only $fg=0$)  
$$
\xymatrix{
A \ar[r]^f 
& B \ar[r]^g & C 
}  
$$
in $\ring G$ 
we may consider two sequences of abelian groups as follows
\begin{equation} 
	\label{l5}
\xymatrix{  	
 GK^G(A,X)  \ar[d]^{P_*}
& GK^G(B,X) 
\ar[l]_{f^*} \ar[d]^{P_*}  
& GK^G(C,X)   \ar[d]^{P_*} 
\ar[l]_{g^*}   \\
  L_0 GK^G(A,Y)    & 
   L_0 GK^G(B,Y) 
\ar[l]_{f^*} &   L_0 GK^G(C,Y)    
\ar[l]_{g^*}  
}
\end{equation}

The 
vertical arrows have to be ignored. If 
the lower line of the diagram 
is exact for every object $Y$ in $GK^G$, then 
the upper line of the diagram 
is exact for every object $X$ 
in $GK^G$. 
(Exact means 
$\ker {f^*} = \image {g^*}$.) 

\end{lemma}

\begin{proof}
Go to line (\ref{l5}). 
Clearly $g^* f^*=0$. To show 
$\ker f^* \subseteq \image g^*$, we consider a morphism $z$ in $\ker f^*$ and choose 
a ring homomorphism $P:X \rightarrow Y$ with 
$GK^G$-right-inverse $Q$ such that 
$z P$ is a level-0 morphism 
by lemma \ref{lemma97}.(ii). 
Clearly, $f^*(zP)= fz P=0$, and so there is 
a $w \in L_0 GK^G(C,Y)$ such that $g^*(w)= zP$ 
by the assumed exactness of the second line 
of the above diagram. 
Hence, $g^*(w Q)= g w Q=g zP Q= z$. 
\end{proof}

Tautologically by the split-exactness definition of $GK^G$-theory we record: 

\begin{lemma} 
Both functors $GK^G(-,X)$ and $GK^G(X,-)$ are split exact. 
That means, given a short split exact sequence  
(\ref{splitexact}),  
the sequence 
$$
\xymatrix{0  & GK^G(B,X)  \ar[l] & GK^G(M,X)   \ar[l]_{j^*} \ar@<-.5ex>[r]_{s^*} & GK^G(A,X)   \ar[l]_{f^*} & 0  \ar[l]}
$$
and its 
analogy with $GK^G(X,-)$ are split-exact. 
\end{lemma}

\begin{proof}

(a)
If $z \in GK^G(A,X)$ and $f^*(z) = f z=0$ then
$s f z = z =0$, proving injectivity of $f^*$. 
If $z \in GK^G(B,X)$ then $z=j \Delta_s z 
= j^*(\Delta_s z)$ 
by split-exactness of definition 
\ref{def121}, 
showing surjectivity of $j^*$.

If $z \in \ker (j^*)$, then $j^* 
(z) = j z=0$, then $\Delta_s j z=0	
= (1-fs) z$ 
by split-exactness of definition 
\ref{def121}, 
then $z=fs z$, so $z \in \image {f^*}$.

(b)
If $z \in \ker (f_*)$, then $f_*(z) = z f=0$, then $ z f s=0	
= z (1-\Delta_s j) $, then $z= z \Delta_s j$, so $z \in \image {j_*}$.
\end{proof}

\section{Morita Equivalence}
				\label{sec14}
 
Let $A,B$ be unital general rings. 
In 
\cite{morita} Morita shows that if a category of  
right $A$-modules 
containing $A$ is 
equivalent to a category of right $B$-modules 
containing $B$ by a 
natural transformation $D_1$ and inverse $D_2$, 
then this 
transformation can be computed by 
the formula $M \mapsto M \otimes_A \cale$, and its inverse by $N \mapsto N \otimes_B \calf$, for a certain fixed $A,B$-module $\cale$
 and $B,A$-module $\calf$. 
Actually $\cale= D_1(A)$ and $\calf=D_2(B)$, 
and the $A$-module structure on $\cale$ comes from left multiplication in $A$. 
Now the left multiplication action of $A$ on $A$ is by compact operators ($A \rightarrow \calk_A(A)$) in our terminology.  
Ideals of $\Hom A(A)$ are in one-to-one correspondences 
to ideals in $\Hom A(\cale)$, and perhaps the $A$-action on $\cale$ is also by compact operators. 
We have not clarified this, but it should 
justify our next definition. 
We take it for granted 
that it is a suitable definition 
and it is designed so   that our next  
theorem works.

\begin{definition}				\label{def131}
{\rm 

Two rings $A$ and $B$ are called {\em functional Morita
equivalent} if there are a right functional 
$B$-module $\cale$ and a right functional  $A$-module 
$\calf$, 
ring homomorphisms $m: A \rightarrow \calk_B(\cale)$, 
$n :B \rightarrow \calk_A(\calf)$ 
(turning $\cale,\calf$ to non-functional bimodules) 
such that 
$\cale \otimes_n \calf \cong A$ 
as right functional, left ordinary  $A,A$-bimodules,
and
$\calf \otimes_m \cale \cong B$ 
as right functional, left ordinary  
$B,B$-bimodules.  
}
\end{definition}

\begin{theorem}   

If two rings $A$ and $B$ are functional Morita equivalent then they are $GK^G$-equivalent.

\end{theorem}    

\begin{proof}

Let $\cale$ be an $A,B$-bimodule and $\calf$
a $B,A$-module which realize the Morita equivalence as described in the last definition, including $m$ and $n$.  

We may assume that $\cale$ and $\calf$ 
are separated by the functionals. Otherwise 
we replace them by $\overline \cale$ and 
$\overline \calf$ of lemma \ref{lemma219}, and $n$ and $m$
by $\pi \circ n$ and $\pi \circ m$ for 
$\pi$ 
of lemma 
\ref{lemma219}. 
Also use lemma \ref{lemma219}.(ii) 
to see $\overline{\cale \otimes_\pi \calf} =
\overline A = A = \overline \cale \otimes_\pi \overline \calf$. 

The $M$ and $N$ in the diagram below are the ring homomorphisms into the canonical corner operators induced by $m$ and $n$. 
That is, 
set $M(a)(\xi \oplus \eta):= m(a)(\xi) \oplus 0$ for $\xi \in \cale, \eta \in \H_B$ 
and 
$N(b)(\xi \oplus \eta):= n(b)(\xi) \oplus 0$.

In this diagram, $e,f,g,h$ are the obvious 
corner embeddings. 
The ring homomorphism $\kappa$ is just the $\phi$ of lemma \ref{lemma43}, so the upper right rectangle of the above diagram commutes.

$$\xymatrix{ 
A \ar[d]^M \\
 \calk_B(\cale \oplus \H_B)   \ar[r]^{e^{-1}}   \ar[d]^{\kappa}  
& 
B  \ar[d]^N 
\\
\calk_{\calk_A(\calf \oplus \H_A)} \Big(
\big (\cale \oplus \H_B \big ) \otimes_N \calk_A(\calf \oplus \H_A) \oplus \H_{\calk_A(\calf \oplus \H_A)} \Big )  
\ar[r]^-{g^{-1}}  \ar[d]^\pi   
& 
\calk_A(\calf \oplus \H_A)
\ar[dd]^{f^{-1}}  
\\
\calk_{A} \Big (
\Big ( 
\cale \otimes_N \calk_A(A, \calf ) 
\oplus  
X 
 \oplus \H_{\calk_A(A, \calf  \oplus \H_A)} \Big )  M_A 	
\Big ) 
\ar[d]^\sigma  
\\
\calk_{A} \Big (\cale \otimes_n \calf  
\oplus X 
\oplus \H_{\calk_A(A,\calf)} 
\oplus   \H_A  \Big )  
\ar[r]^{h^{-1}}  
& 
A    
}$$

The ring homomorphism $\pi$ is that of proposition \ref{lemma71}. 
In the third line of the above diagram we have already canceled some $\calf$s in the domain of $\calk_A$ because of the occurrence of $M_A$,  
and also the $\calk_A(A,\H_A)$-summand 
in the first internal tensor product because 
the map $N$ cancels it anyway. 
We have also abbreviated $X:= \H_B  \otimes_N 
\calk_A(\calf \oplus \H_A)
$.

Now consider $K:=\calk_A(A ,\calf)$ 
as an $B,A$-bimodule by $b \cdot k \cdot a=  n(b) \circ  k \circ m_a $, where $m_a(x) = ax$, 
$a,x \in A$,  
as in proposition \ref{lemma71}.  

Turn it to a right functional module by setting 
$\Theta_A(K) = \calk_A(\calf,A)=:L$ defined by multiplying an element $k \in K$ with an element in  
$l \in  L$ to $l \circ k$.

We get then an $B,A$-bimodule, and even right functional $A$-module isomorphism by 
$$r:\calf \rightarrow 
\calk_A(A,\calf): r(\eta b)(a)= \eta b a = \theta_{\eta, \phi_b}(a)$$
$$f: \Theta_A(\calf) \rightarrow  \Theta_A(K) 
:f(\psi)(k) 
= \psi \circ k $$ 
for $\eta \in \calf, a,b \in A, k \in K, \psi \in \Theta_A(\calf)$, $\phi_b \in \Theta_A(A)$,
$\phi_b(a)=ba$. 
  
Note that $r$ is $A$-linear, as $r(\eta b x)(a)=r (\eta b) \circ  m_x(a)$. 
Injectivity of $r$ follows from
$\xi b a=0$ for all $a \in B$ implies 
$\phi(\xi b) a=0$ for all $\phi \in \Theta_B(\calf),a \in B$ 
implies $\xi b=0$ by 
lemmas \ref{lemma114}, \ref{lemma113} 
and that $\calf$ separates the functionals.

Hence, by the isomorphism $(r,f)$ we get 
$$\cale \otimes_N \calk_A(A, \calf)
\cong \cale \otimes_n \calf  
\cong A 
$$

arriving at the fourth line of the above diagram by a ring homomorphism $\sigma$.  

It is easy to see that $f g \pi \sigma = h$, so that the lower right rectangle of the above 
diagram commutes. 

By commutativity of the above diagram we 
get 
$$M e^{-1} N f^{-1} 
= M \kappa \pi \sigma h^{-1} 
= \idd_A$$
where the last identity we obtain by rotating 
the corner embedding $h$ acting on the distinguished coordinate of $\H_A$ to a corner embedding acting on $\cale \otimes_n \calf \cong A$. 
 
Analogously we get $N f^{-1} M e^{-1} = \idd_B$, and so $A$ and $B$ are $GK^G$-equivalent. 
\end{proof}

\section{Induction Functor}
			\label{sec15}

\begin{definition}
		\label{def1510}
{\rm  
Let $G$ be a discrete group and $H \subseteq G$ 
a subgroup.
Let $(A,\alpha)$ be a $H$-equivariant ring.  
Let $(\cale,S)$ be a $H$-equivariant right functional  $A$-module. 
Define the {\em induced module} 
\begin{eqnarray*} 
\ind H G (\cale) &=& \{f \in c(G,\cale)|
\, f(gh)= S_{h^{-1}} (f(g)) 
\,
\forall h \in H , \\
&&
\qquad  f(gH)=\{0\} \mbox{ for almost all } gH 
\in G/H
\}
\end{eqnarray*}

(i.e. continuous functions $f:G \rightarrow \cale$ 
on the discrete set $G$), 
where this a is 
$G$-equivariant 
right functional 
$\ind H G (A)$-module,  
and the
ring $\ind H G (A)$ is analogously defined 
with pointwise ring multiplication.  
A $G$-action $T$ on $\ind H G (\cale)$ is defined by
$T_g(f)(x):= f(g^{-1}x)$. 

The functional space of this module is 
set to be   
$\ind H G (\Theta_A(\cale))$
 (elements of it understood to be pointwise evaluated against the module), 
defined analogously as above as 
a left $\ind H G (A)$-module.  
}
\end{definition}

\begin{definition} 
{\rm 

Let $H \subseteq G$ be a 
subgroup. 
Define a functor, called {\em induction functor},  
$$\ind H G : GK^H \rightarrow GK^G$$ 
by
\begin{itemize}

\item
For an object $(A,\alpha)$ in $GK^H$ 
set $\ind H G(A):= \ind H G (A)$ 
of definition \ref{def1510}. 

\item
For a ring homomorphism $\pi:A \rightarrow B$
set $\ind H G (\pi):\ind H G (A) \rightarrow 
\ind H G (B)$ to be the ring homomorphism
defined by $\ind H G (\pi)(f)(g)= \pi(f(g))$.

\item
For a corner embedding $e \in GK^H(A,\calk_B(\cale \oplus \H_B))$
define
$$\ind H G (e^{-1}) := \Big (\ind H G (e) \Big )^{-1}$$

\item
For a split exact sequence as in (\ref{splitexact}) 
set $\ind H G(\Delta_s) := \Delta_{\ind H G (s)}$ 
with respect to the following split exact sequence
$$\xymatrix{ 0 \ar[r]  & \ind H G (B)  \ar[r]^{\ind H G (j)}
& \ind H G (M)  \ar[r]^{\ind H G (f)}
& \ind H G (A)  \ar[r]
\ar@<0.5ex>[l]^{\ind H G (s)}
& 0
}
$$

\end{itemize}

}
\end{definition} 

\begin{definition} 
{\rm 
Let $H \subseteq G$ be a 
subgroup. 
Define a {\em restriction functor} 
$\res G H : GK^G \rightarrow GK^H$  
by restricting $G$-equivariance of rings and 
ring homomorphisms to $H$-equivariance, and
otherwise analogously as the induction functor.
  
} 
\end{definition}

The next lemma shows that induction commutes with 
operators, and the proof is straightforward. 

\begin{lemma}

There is a $G$-equivariant ring isomorphism
$$\sigma: \ind H G \Big (\Hom A(\cale )
\Big)
\rightarrow 
\Hom {\ind H G (A)} \Big( \ind H G(\cale ) 
\Big)
$$
defined by 
$\sigma(f 
) (\eta)(g) =  f(g)\big (\eta(g)\big )$, 
where 
$f$ is in the domain of $\sigma$,   
$\eta \in\ind H G(\cale )$ 
and $g \in G$.

This $\sigma$ restricts to 
isomorphisms for $\Hom {}$ replaced by $\call$
and $\Hom {}$ by $\calk$, respectively. 

\end{lemma}

We shall also work in intermediate steps 
with the non-quadratik ring $\Z$, 
and notate
$c_0(G/H)=\ell^2(G/H) = \oplus_{[g] \in G/H} \Z$ 
for the direct sum of abelian groups, 
and $\ell^2(G/H) \otimes B \cong \oplus_{[g] \in G/H} B$ for the direct sum functional $B$-module. 

\begin{lemma}			\label{lemma144}
Let $(A,\alpha)$ be a $H$-equivariant ring and $(B,\beta)$ a 
$G$-equivariant ring.
Then there is a $G$-equivariant ring homomorphism $\mu$ 
$$\xymatrix{\ind H G (A \otimes \res G H (B)) 
\ar[r]^{x} & \oplus_{g \in G_0} A \otimes B 
\ar[r]^{y}   & \ind H G (A) \otimes B
}
$$
defined by $\mu(f ) (g) = (\id_A \otimes \beta_g)(f(g))$ (sloppily said), where $f$ is in 
the domain of $\mu$ 
and $g \in G$. 
 
This also holds for $A=(\Z,\triv)$ and then we may identify $\ind H G (\Z) \cong  c_0(G/H)$. 


\end{lemma}

\begin{proof}

By the universal property of the tensor product we may easily write 
down the well-defined inverse map $\mu^{-1}$ 
as $\mu^{-1}(a \otimes b)(g)= a(g) \otimes \beta_{g^{-1}}(b)$. 
To see that the non-equivariant $\mu$ is well-defined,
we choose an arbitrary  complete representation $G_0 \subseteq G$ of $G/H$ such that 
no two elements $g$ of  $G_0$ are in the same class. 
Then set $x(f):= f|_{G_0}$, 
and 
$$y(\oplus_{g \in G_0} a_g \otimes b_g)= \sum_{g \in G_0}
\big ( k 
\mapsto 1_{\{k=gh \in gH\}} h^{-1}(a_g) 
\big ) \otimes \beta_{g} 
 (b_g)$$  
  and 
check that $\mu =xy$, inverse to $\mu^{-1}$.
\end{proof}

Roughly, 
this type of Frobenius reciprocity 
and its proof are well-known. 

\begin{proposition}  

Let $H \subseteq G$ be a 
subgroup. 

Then 
the induction functor is left adjoint to the restriction functor with respect to $GK$-theory. 
That means we have
$$GK^G \Big ( \ind H G (A),B \Big ) = GK^H 
\Big (A,\res G H (B) \Big)$$
and this identity is natural in $A$ and $B$. 

\end{proposition}

\begin{proof}

We shall consider the unit and counit of adjunction. 
First we have the natural transformation
$\iota$  
of the identity functor $\id_{GK^H}$
on $GK^H$ and the functor $\res G H \ind H G$ 
on $GK^H$ 
given by the family of ring homomorphisms
$$\iota_A: A \rightarrow \res G H \ind H G(A)
: 
\iota_A(a) ( g)= 1_{\{g \in H\}} g^{-1}(a) $$

Second we have the natural transformation
$\pi$ between the functors $\id_{GK^G}$ 
and $\ind H G \res G H $ on $GK^G$
defined by the family of ring homomorphisms
$$\xymatrix{ \pi_B : 
\ind H G \res G H (B)
\ar[r]^\mu &
c_0(G/H) \otimes B
\ar[r]^\lambda &
\call_B ( \ell^2(G/H) \otimes B 
\oplus \H_B)
\ar[r]^-{e^{-1}} &
B
}$$


where $\lambda$ 
sends to the coordinate-wise diagonal multiplication operators, but to zero on $\H_B$,
and $e$ is the corner embedding acting on $\H_B$.   

It is sufficient to show that
\begin{equation}   \label{ident1}
\pi_{\ind H G (A)} \circ \ind H G (\iota_A) = \mbox{id}_{\ind H G (A)}
\end{equation}
in $KK^G$ and
\begin{equation}   \label{ident2}
{\res G H (\pi_B)} \circ \iota_{\res G H (B)} = \id_{\res  G H(B)}
\end{equation}
in $KK^H$ by \cite[IV.1 Theorem 2.(v)]{maclane}.

Define $P_{[g]}$ to be the projection acting on 
$\ind H G(A)$ which leaves an element on $gH$
unmodified and sets it to zero outside of it. 
 
Then, for $\mu$ of Lemma \ref{lemma144} 
and 
$a \in \ind H G (A)$, one has  
\begin{eqnarray}
\label{c1}
&& \mu \big(\ind H G (\iota_A)(a) \big ) \\
\label{c2}		
 &=& 
 \sum_{[g] \in G/H} 1_{[g]} \otimes 
P_{[g]}(a)
\quad
\in   \quad  c_0(G/H) \otimes \ind H G (A) 
\end{eqnarray}

Now $\pi_{\ind H G (A)} \circ  \ind H G (\iota_A)
= \ind H G (\iota_A) \mu \lambda e^{-1}$  
 is a morphism
$ \ind H G  (A) \rightarrow
\ind H G  (A)$. 
Set $B:= \ind H G  (A)$. 

We have a direct sum decomposition 
$X:=\ell^2(G/H) \otimes \ind H G (A)\cong  M \oplus N$ 
of $G$-modules,   
where $M$ is just the set of all elements 
(\ref{c2}) when $a$ is varied over all $a \in B$.

We have an isomorphism 
$$u:M \rightarrow B  : u \Big (  
\sum_{[g] \in G/H} 
1_{[g]} \otimes 
P_{[g]} (a)
\Big ) = a
$$ 
of $G$-$\ind H G  (A)$-modules 
for $a \in B$.

Since by the above computation (\ref{c1}), $z:= \ind H G ( \iota_A) \mu \lambda$ acts only on $M$, 
if we compose it with $\Ad(u \oplus 1 \oplus 1)$ it becomes the multiplication operator on $B$,
so a corner embedding $f$ and we get $z e^{-1} = f e_2^{-1} =\idd$ by a rotation homotopy, see the following commuting diagram: 

$$\xymatrix{ B \ar[r]^-z \ar[d] &  \calk_B(M \oplus N \oplus \H_B)  \ar[d]^{ \Ad (u \oplus 1 \oplus 1)} 
&  B  \ar[d]  \ar[l]^-e
\\
B \ar[r]^-{f} & \calk_B( B  \oplus N \oplus \H_B )   
&    B   \ar[l]^-{e_2} 
}$$

The case (\ref{ident2}) is 
proven similarly.				
\end{proof}

\section{Green-Julg Isomorphism}
				\label{sec16}

In this section we shall prove the 
well-known Green-Julg isomorphism theorem 
\cite{julg} 
in the 
framework  
of $GK$-theory. 

In this section, if nothing else is said, $G$ denotes a 
finite, discrete group 
and we 
put $n:= |G|$.


Moreover, if nothing else is said, $(\F,\triv)$ denotes any given commutative, associative field with 
$\character(\F) \neq n$ (because we shall divide by $n$). 

All rings and modules are now $\F$-algebras 
and $\F$-vector spaces, respectively,  
as explained in section \ref{section10}.  
In particular, all tensor products are $\F$-balanced. 

 Typically, $\delta_g$ for $g \in G$ denote canonical 
basis elements. 

\begin{definition}   	\label{def151}
{\rm 

Let $G$ be a finite group. 
Write $\ell^2(G):= \oplus_{g \in G} \F$ 
for 
the direct sum functional right (and also left notated)  $\F$-module.
We also write $\langle \xi,\eta\rangle := \sum_{g \in G} \xi_g \eta_g$ for the ``inner product". 

On it define  
the right, ``{\em right regular}" $G$-action
$V$ by 
$V_g(\delta_h) = \delta_{h g}$ 
for $\xi \in \ell^2(G), g,h \in G$. 

Equip $\ell^2(G)$ with the left $G$-action $V^{-1}$ 
defined by $V_g^{-1}:= V_{g^{-1}}$.  

Write $\K:= \big (\calk_\F(\ell^2(G)), \Ad(V ^{-1}
)\big )$. 

Set $m:=\sum_{g \in G} \delta_g \in \ell^2(G)$. 
A corner embedding $f: \F \rightarrow \K= \calk_\F( \F^n 
)$ is given by 
$$f(x)(\xi)=x m \frac{1}{n} \langle \xi,m\rangle  = \theta_{x m, \phi_f}(\xi)$$ 
for all $x \in \F, \xi \in \ell^2(G), h \in G$. 
Here, $\phi_f \in \Theta_\F(\ell^2(G))$ is $\phi_f(\xi)=n^{-1} \langle \xi,m\rangle$. 

}
\end{definition}

\begin{lemma}			\label{lemma152}  	

{\rm (i)} $f$ is indeed a corner embedding.

{\rm (ii)}
If $A$ is 
a ring then 
 $1 \otimes f: A \cong A \otimes \F \rightarrow A \otimes \K$ 
is a corner embedding. 
\end{lemma}

\begin{proof}
(i) 
Observe that $f(xy)= f(x)f(y)$. 
%
Choosing a $\F$-vector space basis in $\ker \phi$ we obtain a non-equivariant isomorphism 
$$T: \ell^2(G) \cong \F \oplus \ker(\phi_f)
\rightarrow  
\F \oplus \F^{n-1} = \F^n$$ 
of vector spaces, 
such that $T(m)=m \oplus 0^{n-1}$ for $m$ of definition \ref{def151}. 
Select the $G$-action $W$ on $\F^n$ such that $T$ becomes $G$-equivariant. 

Observe that $m$ is $V$-invariant (i.e. $V_g(m)=m$), such that $T(m)$ is $W$-invariant, and so the first copy $\F$ of the direct sum $(\F^n,W)$ is a distinguished first coordinate $(\F,\triv)$. 

Then $(T \circ f(x) \circ T^{-1})(\eta)= x \eta_1 \oplus 0^{n-1}$ for $\eta \in \F^n$, so $f$ is a corner embedding. 

(ii)
This is just the tensor product functor, 
but let us observe it once more: 
$$(A,\alpha) \otimes \calk_\F((\F^n,W)) \cong \calk_A(A) \otimes \calk_\F(\F^n) \cong \calk_{A \otimes \F}( (A \otimes \F^n,\alpha \otimes W))
\cong \calk_A(A^n)$$

where the first copy of $A$ in $A^n$ is a distinguished first coordinate $(A,\alpha)$,
and the map $1 \otimes f$ 
turns to the obvious 
canonical corner embedding $A \rightarrow \calk_A(A^n)$. 
\end{proof}

\begin{definition}   		\label{def153} 
{
\rm 

Let $(A,\alpha)$ be a ring. 
We often identify 
$$\big (A \otimes \K,\alpha \otimes \Ad(V^{-1}) 
\big ) \cong \Big (\calk_A \big (A \otimes \ell^2(G) \big ),
\Ad(\alpha \otimes V^{-1} )\Big )$$
and also abbreviate the $G$-action 
$\rho:= \Ad(V^{-1})$ on $\K$. 


Let $\pi:A \rightarrow A \otimes \K$ be the ring homomorphism defined by $\pi(a)(b \otimes \delta_g)
= \alpha_{g^{-1}}(a) b \otimes \delta_g$. 

Let $U: G \rightarrow A \otimes \K$ be 
the group homomorphism  
defined by
 $U_g(a \otimes \delta_h)
= a \otimes \delta_{g h}$. 
(``Left regular action", but it is not a group action as we have the $G$-action $\alpha$ on $A$.)  

Let $\lambda_A: (A \rtimes_\alpha G, \triv) \rightarrow 
(A \otimes \K, \alpha \otimes \rho)$ 
be the 
injective, $G$-equivariant ring homomorphism 
(``{\em left regular representation}") given by  
$\lambda_A(a \rtimes g)= \pi(a) \circ U_g$.

}
\end{definition}

This is the Green-Julg map: 

\begin{definition}    \label{defgj}
{
Let $(A,\alpha)$ be a ring. 
Define an abelian group homomorphism 
$$S:GK^G(\F,A) \rightarrow GK(\F,A \rtimes_\alpha G)$$
by
$$S(\xymatrix{\F \ar[r]^{Y} & A} ) =
\xymatrix{ \F \ar[r]^M & \F \rtimes G \ar[r]^{Y \rtimes \idd} & A \rtimes_\alpha G} 
$$
where $M$ is the ring homomorphism (``averaging map") given by  
$M(x) = 
 x n^{-1} \sum_{g \in G}  1 \rtimes g$.
}
\end{definition}

$M$ is chosen such that $M \lambda_\F = 1 \otimes f$, and this is the only reason for its choice. 

This is the inverse Green-Julg map: 

\begin{definition}    \label{definvgj}
{
Let $(A,\alpha)$ be a ring. 
Define an abelian group homomorphism
$$T:GK(\F,A \rtimes_\alpha G) \rightarrow GK^G(\F,A)$$
by 
$$T(\xymatrix{\F \ar[r]^{L} & A \rtimes_\alpha  G} ) =
$$
$$
\xymatrix{ (\F,\triv) \ar[r]^{L} & (A \rtimes_\alpha G,\triv)  \ar[r]^{\lambda_A} &  (A \otimes \K,  
\alpha \otimes \rho 
) \ar[rr]^{(\idd \otimes f)^{-1}} && (A \otimes \F, \alpha \otimes 1)}$$

}
\end{definition}

\begin{lemma}		\label{lemma1160} 

Let $\varphi:(A,\alpha) \rightarrow (B,\beta)$ be a morphism in $GK^G$. Then the following diagram commutes:

$$
\xymatrix@u{ 
B \rtimes_\beta 
G  \ar[d]^{\lambda_B}  & A \rtimes_\alpha 
G  \ar[d]^{\lambda_A} 
\ar[l]^{\varphi \rtimes 1}  \\
B \otimes \K  \ar[d]^{(1 \otimes f)^{-1}} & A \otimes \K  \ar[d]^{(1 \otimes f)^{-1}}  
\ar[l]^{\varphi \otimes 1} \\
B \otimes \F   
  & A \otimes \F   \ar[l]^{\varphi \otimes 1}
}
$$

\end{lemma}

\begin{proof}
For $\varphi$ a ring homomorphism this is obvious, for $\varphi$ the inverse of a corner embedding also, by proving commutativity with 
the reversed arrows. 
For $\varphi=\Delta_s$ one makes obvious 
diagrams with split exact sequences 
and verifies with lemma \ref{lemma21}.   
\end{proof}

\begin{proposition}
We have $T \circ S = \id_{GK^G(\F,A)}$.
\end{proposition}

\begin{proof}
We have
$$T \circ S(Y) =
\xymatrix{
\F \ar[r]^M & \F \rtimes G \ar[r]^{Y \rtimes \idd} & A \rtimes G
\ar[r]^{\lambda_A} &  A \otimes \K \ar[rr]^{(\idd \otimes f)^{-1}} && A \otimes \F}
$$

Hence, by lemma \ref{lemma1160},  
$$T \circ S(Y) = 
\xymatrix{
\F \ar[r]^M & \F \rtimes G \ar[r]^{\lambda_\F} &  \F \otimes   \K
\ar[rr]^{(\idd \otimes f)^{-1}} && \F \otimes \F   \ar[r]^{\idd \otimes Y}  & \F \otimes A 
}
$$

Since $M \lambda_\F = f \otimes \idd$, the proposition is proved.
%
%
\end{proof}

\begin{lemma}		\label{lemma118}

Let $(A,\alpha)$ be a unital $C^*$-algebra with a $G$-action. Let $A^\alpha$ be its fixed point algebra under the $G$-action. 
If $u \in A^\alpha$ is a unitary element 
non-equivariantly 
homotopic to $1 \in A^\alpha$ 
by a unitary path in $A$ and with spectrum not equal to the torus $\T$, 
then $u$ is non-equivariantly homotopic to $1$ by a unitary path
also in $A^\alpha$. 

\end{lemma}

\begin{proof}

Let $\nu:A \rightarrow A^\alpha: \nu(a)=n^{-1} \sum_{g \in G} \alpha_g(a)$ be the averaging map, which is a linear continuous projection. 

Let $(u_t)$ be the given homotopy in $A$. 
Since $\spec(u) \neq \T$, there is a fixed scalar $\lambda \in \C$ (direction through the gap of $\spec(u)$) such that the straight line in $A^\alpha$ connecting $u$ with $u + \lambda 1$ consists of only invertible elements. We move further from $u + \lambda 1$ by the homotopy $t \mapsto u_t +\lambda 1$ of invertible elements in $A$ and end at $(\lambda+1_\C)1$ in $A^\alpha$. 
From here, we move by a straight line to $1 \in A^\sigma$. 
If we choose $|\lambda|$ big enough (by Neumann series), even 
the image under     
$\nu$ of this sketched path in $A$ from $u$ to $1$ remains to consist of invertible elements only, and is a path in $A^\alpha$. We may change it to a unitary path in 
$A^\alpha$. 
\end{proof}

\begin{proposition}[
{\cite[Proposition 4.6]{rieffel}}] 
				\label{prop159}
Let $(A,\alpha)$ be a 
$\F$-algebra with $\F \subseteq \C$ being a subfield. 
There is a 
natural ring isomorphism 
$\zeta_A: A \rtimes_\alpha G \rightarrow 
(A \otimes \K)^{\alpha \otimes \rho}$
(which is not $\lambda_A$). 

In the $C^*$-algebra setting 
this holds also for compact groups $G$. 

\end{proposition}

\begin{proof}

In the $C^*$-algebra setting this is 
\cite{rieffel} and the remark thereafter. 
Inspection of the proof of \cite{rieffel} 
shows that no $C^*$-typical analysis is involved and consists of relatively canonical 
homomorphisms, which appear to do not only work for 
$\C$-algebras, but even $\F$-algebras. 
\end{proof}

\begin{proposition} 

Assume  that $\F$ is a subfield of $\C$. 

Then 
we have $S \circ T = \id_{GK(\C,A \rtimes_\alpha G)}$.
\end{proposition}

\begin{proof}
(a)
Let $\xymatrix{ \C \ar[r]^L  & A \rtimes_\alpha G}$ in $KK(\C , A \rtimes G)$ be given.

Consider the diagram
$$\xymatrix{
\C \ar[r]^M  
\ar[rd]^f 
& \C \rtimes_\triv G  \ar[r]^{L \rtimes \idd}  
\ar[d]^{\lambda_\C}
&  ( A \rtimes_\alpha G)  \rtimes_\triv G \ar[r]^{\lambda_{A} \rtimes \idd }
\ar[d]^{\lambda_{A \rtimes_\alpha G}} 
&
(A \otimes \K) \rtimes_{\alpha \otimes \rho} 
 G  \ar[rr]^{(\idd \otimes f)^{-1} \rtimes \idd} 
\ar[d]^{ \lambda_{A \otimes \K} }  
&&  A \rtimes_\alpha G  
\ar[d]^{\lambda_A}  \\
& \C \otimes \K  \ar[r]^{L \otimes \idd}  &  (A \rtimes_\alpha G)  \otimes \K    
\ar[r]^{\lambda_A \otimes \idd}
& 
A \otimes \K \otimes \K  
 &&   A \otimes \K     
 \ar[ll]^{(\idd \otimes f) \otimes \idd} 
\\
}
$$

The top line is exactly $S \circ T(L)$. 
It is not ad hoc $L$, but one may observe 
that it was 
if only the two occurring copies of $G$ would be exchanged. 

To make the flip, we go down with the $\lambda$s everywhere as notated in the diagram. 
The whole diagram commutes  
by lemma \ref{lemma1160}.  

We restrict now 
the last line of the above diagram to the image spaces of the down-going $\lambda$-arrows, so they are now all 
bijective. 

Obviously, again everything in the diagram commutes, where we define the restricted 
$L \otimes 1$ as 
as 
$\lambda_\C^{-1} (L \rtimes 1) \lambda_{A \rtimes G}$.

(b) 
Assume for the moment that $\F=\C$. 
On the 
ring $B:=A \otimes \K \otimes \K$ 
the flip 
ring endomorphism $F$ defined by $F(a \otimes k \otimes l)= a  \otimes l \otimes k$  
is homotopic to the identical 
ring homomorphism 
$\idd$ on that space, because $F=\idd \otimes \Ad(H)$ for the unitary 
flip operator $H \in \K \otimes \K \cong  
B(\C^n \otimes \C^n)$, and $H$ is homotopic to $\idd$ by a unitary path 
$(H^t)_t$, as is well known. 

Set $u:= 1 \otimes H$. Observe that $u \in B^G:=B^{\alpha \otimes \rho \otimes \rho}$
(fixed point algebra), and that 
$F=\Ad(u)$ can be restricted to $B^G$. 

By lemma 
 \ref{lemma118} applied  
to $(\K \otimes \K, \rho \otimes \rho)$ and  $H$, 
$F|_{B^G}$ is homotopic to the identity ring homomorphism on $B^G$  
by a homotopy $F'_t:=\Ad(u_t)$ for
$u_t \in B^G$.  

(c) 
Let $X$ be the 
range of $\lambda_{A \otimes \K}$. 
Now our desired ``flip" is 
$$W:X \rightarrow X: W=\lambda_{A \otimes \K}^{-1} \zeta \cdot 
F|_{B^G} \cdot 
\zeta^{-1} \lambda_{A \otimes \K},$$  
which is $\idd$ in $GK$.  
Here, $\zeta$ is from proposition 
\ref{prop159} applied to $A:=(A\otimes \K,\alpha \otimes \rho)$.

Let us take for granted, that $\zeta$ is so `natural' that $W(a \otimes U_g \otimes U_h)=
a \otimes U_h \otimes U_g$ (note that $U_g$ is $\rho$-invariant), 
which might only be seen by inspection of the proof of 
\cite[Proposition 4.6]{rieffel}. 

Thus, 
as the image of $f$ is in the image of $\lambda_\C$ by the first 
triangle of the above diagram, 
and the $\lambda$s are realized by the operators $U$, 
the operator $W$ exchanges the coordinates in such 
a way that $\lambda_A ((1 \otimes f) \otimes 1) W = \lambda_A \otimes f$ 
in the left area of the above diagram.

We compute now $S \circ T(L)$ by going the bottom path in the above diagram, and implementing $W=1$ at $X$. 
Using the granted 'naturality' of $W$, 
we thus have 
$$S \circ T(L) = f(L \otimes 1) (\lambda_A \rtimes 1) W (1 \otimes f \otimes 1)^{-1} \lambda_A^{-1} = L$$

proving the proposition for 
$\F=\C$.

(d)
Now let $\F$ be an arbitrary subfield of $\C$. 
We do everything as above and need only 
explain why $F|_{B^G}$ is still homotopic to 
$1$ as endomorphisms on $B^G$. 
To this end we at first complexify 
$B$ and 
obtain the following $\C$-algebra isomorphism 
by lemma \ref{lemma120}, 
$$
\tau :
 B \otimes \C
=
   \big ( A \otimes \K \otimes \K  
\big )  \otimes \C 
\rightarrow     
\hat A \otimes^\C \hat \K \otimes^\C \hat \K   
$$

%
%
%
%
%

Because $\F \subseteq \C$ is a subfield, 
we get $\F \otimes \C 
\cong \C$ as   
algebras over $\C$, because we use the 
$\F$-balanced tensor product everywhere. 
Consequently, we get an isomorphism 
of $\C$-algebras 
$$\hat \K = M_n(\F^n) \otimes \C \cong 
M_n(\F^n \otimes \C) \cong 
M_n(\C^n) 
\cong \calk_\C (\ell^2(G,\C))  
$$

The map $\tau$ 
is also equivariant 
with respect to $(\alpha \otimes \rho \otimes \rho) \otimes \idd$ and $\alpha \otimes \rho \otimes \rho$, respectively, whence we
can restrict 
$\tau$ to the fixed 
point algebras.  
%
We can 
thus perform the desired and already verified homotopy 
of endomorphisms 
on the fixed point algebra of the range of 
$\tau$ 
and go back to 
endomorphisms on $B^G$ 
at $0$ and $1$ only.  
\end{proof}

\begin{corollary}[Green-Julg Theorem] 

Let $G$ be a finite group 
and $\F$ a subfield of $\C$. 
Then the Green-Julg map $S$ of definition 
\ref{defgj} 
is an isomorphism of abelian groups 
with inverse isomorphism $T$ as defined in definition 
\ref{definvgj}. 

\end{corollary}

\begin{lemma}			\label{lemma1512}
The Green-Julg isomorphism $S$ 
is functorial, that is, 
for every morphism $Z: A \rightarrow B$ one has a commuting diagram 
$$\xymatrix{ 
KK^G(\F,A ) \ar[r]^{S_A}    \ar[d]^{Z_*}  
 & KK(\F,A \rtimes G)   \ar[d]^{(Z \rtimes 1)_*}   
 \\
KK^G(\F,B) \ar[r]^{S_B}  
 & KK(\F,B \rtimes G)	}
$$
\end{lemma}

This is actually a corollary of the above 
proof:

\begin{corollary}[Green-Julg Theorem]

Let $G$ be a compact group. 
Let $(A,\alpha)$ be a 
separable $G$-$C^*$-algebra. 
Then there is an isomorphism of abelian groups 
$$S:KK^G(\C, A) \rightarrow KK(\C,A \rtimes_\alpha G)$$
with inverse map $T$, analogously defined as above. 
\end{corollary}

\begin{proof}

We may identify $KK^G \cong GK^G$ by \cite{bgenerators}. 

One replaces sums over $G$ by integrals over $G$. 
The vector space 
$\ell^2(G)$ is 
replaced by $L^2(G)$ (Haar measure) with inner product 
$\langle \xi,\eta\rangle = \int_G \xi(g) \overline{ \eta(g)} dg$. 
The averaging factor $n^{-1}$ can be omitted 
as the Haar measure of $G$ is one. 

For example, in definition \ref{def151}, one sets 
$m:= 1 \in L^2(G)$ (constant $1$ function) 
and 
$f: \C \rightarrow \K$, $f(x)= x \langle \xi,m\rangle $.  

The function-formulas involving $\delta_g$ must be one-to-one reformulated in function-evaluation writing. 
For example, $U_g(a)(\xi)(h) := a(\xi(g^{-1}h))$. 

The finite basis of the proof of lemma \ref{lemma152} 
has to be replaced by a possibly countably infinite,
orthonormal basis. 
It is well known that 
the left regular representation of 
definition \ref{def153} maps 
into the compact operators for $G$ compact.    

Otherwise the proof 
is formally completely unchanged. 
\end{proof} 

We remark that the Green-Julg theorem 
holds also in a `weaker' theory than $GK^G$, where we drop the split-exactness axiom 
involving the $\Delta$s, because 
the $\Delta$s are nowhere essentially used in the proof, 
but do also not disturb.

\section{Baum-Connes map}
				\label{sec17}

We continue the last section by extending 
the Green-Julg map to the Baum-Connes map. 
We let the field be $\F:=\C$. 
 To incorporate the $C^*$-case, 
we also allow 
 $G$ to be a locally compact, second-countable group $G$. 
 
We very briefly recall some notions to 
define a Baum-Connes map with respect to 
$GK$-theory. 
For more details on these notions see
\cite{baumconneshigson1994}, or for instance 
\cite{guentnerhigsontrout}. 

\begin{definition}
{\rm 
A locally compact space $X$ equipped with a continuous $G$-action 
is called {\em proper}
if the map $\chi: G \times X \rightarrow X \times X$ defined by $\chi(g,x)=(g x,x)$ is proper; that means, if for all compact subsets 
$K \subseteq X\times X$, $\chi^{-1}(K)$ is compact.

A locally compact $G$-space $X$ is called {\em $G$-compact} if the quotient space $G\backslash X$ is compact (if and only if 
$X$ is the $G$-saturation $G \cdot  K$ 
for a compact subset $K \subseteq X$).
}
\end{definition}

\begin{definition}
{\rm 
If $X$ is a locally compact, proper $G$-compact $G$-space then there
is a {\em cut-off function} $c$ for it,
that means a continuous function 
$c:X \rightarrow \R$ with compact support 
such that
$$
\int_G c(g^{-1} x)^2 dg=1 \qquad 
\forall x \in X$$ 
with respect to the Haar measure on $G$.

A projection $p_X \in C_c(X) \rtimes G$  is defined by 
$$p_X(g)(x) 
= m_G(g)^{-1/2} c(g^{-1}x) c(x)$$
where $m_G$ is the modular function of $G$,   
and a ring homomorphism $M_X:\C \rightarrow C_0(X) \rtimes G$ by $M_X(z)= z p_X$.

}
\end{definition}

Let $\underline E G$ be an universal example 
for proper actions 
by $G$,
see \cite{baumconneshigson1994}. 

\begin{definition}
{\rm 
If $X$ is a  locally compact, $G$-compact $G$-space then define the {\em assembly map} 
as the abelian group homomorphism  
$$\nu_X^{G,A}:GK^G(C_0(X),A) \rightarrow GK(\C,A \rtimes G)$$
$$\nu_X^{G,A}(\xymatrix{C_0(X) \ar[r]^{V} & A} ) =
\xymatrix{ \C \ar[r]^{M_X} & C_0(X) \rtimes G \ar[r]^{V \rtimes \idd} & A \rtimes G},$$
}
\end{definition}

\begin{definition}
{\rm 
Define the {\em Baum-Connes assembly map} 
as the abelian group homomorphism  
$$\nu^{G,A}: \lim_{X \subseteq \underline E G}  GK^G(C_0(X),A) \rightarrow GK(\C,A \rtimes G)$$

where the direct limit of abelian groups is indexed over all $G$-compact, second-countable, locally compact $G$-invariant subsets $X \subseteq \underline E G$ with direct limit connecting 
abelian group homomorphisms  
$\pi^*$ 
for all restriction ring homomorphisms  
$\pi:C_0(Y) \rightarrow C_0(X)$, $\pi(f)(x)=f(x)$,
whenever $Y \supseteq X$.
The abelian group homomorphism $\nu^{G,A}$ is then defined to be  canonically induced by the abelian group homomorphisms
$\nu_X^{G,A}$.
 
}
\end{definition}

\begin{lemma}				\label{lemma167}
The Baum-Connes map is functorial 
in the coefficient algebra  $A$, that is a
completely analogous statement as in 
lemma \ref{lemma1512} holds.
 
More precisely, $\nu^{G,A}_X \circ (Z \rtimes 1)_* = Z_* \circ \nu^{G,B}_X$ 
and $\nu^{G,A} \circ (Z \rtimes 1)_* = Z_* \circ \nu^{G,B}$ 
for all morphisms $Z :A \rightarrow B$. 

\end{lemma}

If we restrict everything 
on the domain of the 
Baum-Connes map 
 to level-0 morphisms, then
we get a 
restricted abelian group homomorphism   
$$L_0\nu^{G,A}: \lim_{X \subseteq \underline E G}  L_0 GK^G(C_0(X),A) \rightarrow L_0 GK(\C,A \rtimes G)$$

\begin{proposition}			\label{prop176}

The Baum-Connes maps $\nu^{G,A}$ are injective 
for all coefficient algebras $A$ if and only 
if 
$L_0\nu^{G,A}$ is injective for all coefficient algebras $A$. 
\end{proposition}

\begin{proof} 
Let $\psi_X^{G,A}$ the 
abelian 
group homomorphisms from the domain of $\nu_X^{G,A}$ 
to the domain of the Baum-Connes map $\nu^{G,A}$ 
associated to the direct limit. 

Suppose $\nu^{G,A}(\psi_X^{G,A}(z))
=\nu^{G,A}_X(z)=0$ 
for a morphism $z:C_0(X) \rightarrow A$. 
Then 
$z P$ is a level-0 morphism for some 
$GK^G$-right-invertible 
ring homomorphism $P:A \rightarrow D$ with right-inverse $Q$.
Thus, by lemma \ref{lemma167},  
$$\nu^{G,D}(\psi_X^{G,D}(z P)) 
= \nu^{G,D}_X(z P)= \nu^{G,D}_X(z) (P \rtimes 1)=0$$ 

By injectivity of $L_0\nu^{G,D}$ we conclude 
that $\psi_X^{G,D}(z P)=0$. 

But then $w z P=0$ in $GK^G$ for some restriction map 
$w:C_0(Y) \rightarrow C_0(X)$ for $Y \supseteq X$. 
Hence $w z = w z PQ=0$. 
Thus $\psi_X^{G,A}(z)=\psi_Y^{G,A}(wz)=0$. 
\end{proof}

Note that this holds also with respect to $C^*$-algebras and $KK$-theory by analogy. That is, 
the classical Baum-Connes map is injective 
if and only if it is injective on 
$L_0 KK^G$, 
or more precisely,  
on the abelian subgroup 
generated by 
the $*$-homomorphisms.

\bibliographystyle{plain}
\bibliography{references}

\if  1 


\section{}

In the present paper we shall consider a class $\calr$ of rings, which may be topologized, but we discuss mainly the case where these rings are equipped with the discrete topology.

\begin{definition}[Rings]
{\rm

Let $\calr$ be a given class of associative, not necessarily commutative and possibly non-unital rings $(A,\tau,\alpha,\chi)$, 
all of which are euipped with a Hausdorff topology $\tau$ and an approximate unit.
That means, for each ring $A$ in $\calr$ there exists a net $(p_i)_{i \in I}$ in $A$ such that for all $a \in A$, the nets $a - p_i a$ and $a - a p_i$ converge to zero. 
Moreover, 
each 
$A$ is equipped with a group action $\alpha: G \rightarrow \Aut(A)$. 

The $\chi$ stands for some possible extra information like for instance a vector space structure on $A$, a norm on $A$, or an involution on $A$. 
%
}
\end{definition} 

kann die ringe nicht diskret nehmen, 
wenn man auf operatorräumen keine diskrete topology hat .

am besten überall vorgschriebene frei gewählte topology, auch auf $\call$ etc .

\begin{definition}
{\rm
For each ring $A$ in $\calr$ let us be given 
a class $\calh_A$ of infinite 
}
\end{definition}

\begin{definition}[Modules]
{\rm

For each ring $A$ in $\calr$ we have given a class of right $A$-modules $\calm_A$, which is typically closed under the following elementary constructions:

\begin{itemize}

\item
$A$ is in $\calm_A$;

\item
the direct sum $M \oplus N$, a version of an exterior tensor product $M \otimes N$, a version of an internal tensor product $M \otimes_f N$ is in $\calm_A$ whenever $M,N$ are modules in $\calm_A$ 

\end{itemize}

Let $\calm$ be a given class of modules. 
More precisely, each module in $M$ should be a right module over at least one ring $A$ of $\calr$.

}
\end{definition}

\begin{definition}[Functionals]
{\rm

To each module $(M,S)$ is associated an abelian group $\Theta_A(M)$ 
consisting of right $A$-linear module maps $\phi:M \rightarrow A$ ($A$-linear functionals).

}
\end{definition}

Als topology in operator räumen nehme 
punktweise konvergenz: $\phi_i (x) \rightarrow \phi(x)$ für alle $x$

-

die frage ob nicht alle compa ct werden ->
testen auf $\call_A(A)$

n, denn norm convergence -> pktweise convergenz, dh normabschluss größer (?)

n, umgekehrte implication : pktweise abschluss > normabschluss

tatsächlich operatortopology abschluss con compacten auf h , K(H) -> abschluss B(H)

\begin{definition}[Compact operators]
{\rm

To each module $M$ in $\calm_A$ is associated 
a space of compact opertators $\compacts A M \subseteq \call_A(M)$
which consists as a dense subspace the
finite sums  of elementary compact operators 
$\theta_{\xi,\phi}$ defined by
$\theta_{\xi,\phi}(\eta)= \xi \phi(\eta)$ for $\xi,\eta\in M$ and $\phi \in \Theta_A(M)$.

}
\end{definition}

wegen der idealbedingung muss gelten:

wenn $\phi \in \Theta_A(\cale)$ dann $\phi \circ T \in \Theta_A(\cale)$ für alle $T \in \call_A(\cale)$ 

klären ob $\Hom A(A)$ nur multiplikationsoperatoren 

---

Es macht keinen unterschied ob man
$\Theta_A(\cale)$ oder $A^+ \Theta_A(\cale)$ 
nimmt in der def von kompakten

\begin{definition}[Infinite direct sums]
{\rm

To each ring $(A,\alpha)$ in $\calr$ is a associated 
a class $\calh_A$ consisting of objects in $\calm_A$ which are infinite direct sums.
  That is, an object in $\calh$, typically denotes by $\H_A$ is an right $A$-module which has the infnite direct sum $\bigoplus_{n=1}^\infty A$ as a dense sub-A-module.
The $G$-action on $\H_A = A \oplus A \oplus A \oplus ...$ must be of the form $(\alpha \oplus S)$, where $S$ may be any $G$-action $S$ on 
$A \oplus A \oplus A \oplus \ldots$ 
(and need not to be the canonical driect sum $G$-action $ \alpha \oplus \alpha \oplus \ldots$).

---

vllt besser nur $\bigoplus A$ is untermodule of $\H_A$ 

eigentlich brauche nur $A \oplus A$, um die rotation machen zu können 

vllt besser: $A \oplus  A$ directer summand von $\H_A$ 

---

eigentlih brauche nur eine koordinate $A$ in $\H_A$ 

}
\end{definition}

Sind 2 ringe mit ungerschiedlich topology gleich?
 
---

Normally, a module homomoprhism is an $A$-linear map. But we need to require that this map also intertwinnes the functional spaces:

\begin{definition}
{\rm
A module isomorphism between two $A$-modules
$(\cale,\Theta_A(\cale))$ and 
$(\calf,\Theta_A(\calf))$
is an $A$-linear map $T:\cale \rightarrow \calf$ and a 
additive 
set-theoretical bijektion $f:\Theta_A(\cale)  \rightarrow \Theta_A(\calf)$ 
such that for every $\varphi \in \Theta_A(\cale)$ 

$f(\varphi) (T \xi)= \varphi(\xi)$
 
vllt auch $f$ links-$A$-module homomoprhism 

$$T(\xi a \phi()) = T(\xi) a f(\phi())$$

---

man müsste hier erwähnen, dass es egal ob $\Theta_A$ oder $A^+ \Theta_A$ und dass die def funktioniert 

vllt verlangen, dass $\Theta_A$ link-$A$ module 

---

$f:A^+ \Theta_A(\cale) \rightarrow A^+ \Theta_A (\calf)$

}
\end{definition}

\section{}

The first coordinate $(A,\alpha)$ of $(\H_A,\alpha \oplus S)$ is often referred to 
as the {\em distinguished coordinate} of $\H_A$.

approx einheit von $\compacts A {\cale \otimes_f \calf}$ etc. 

$\theta_{\xi,\phi} = \xi \phi$
 
$\theta_{\xi,\phi} \theta_{\eta,\psi}
= \xi \phi(\eta) \psi = \theta_{\xi \phi(\eta),\psi}$

Für jedes $\eta \in \cale$ gebe es ein netz $\phi_i \in \Theta_A(\cale)$ sodass 
$\phi_i(\eta)$ approx einheit von $A$

$\xi a_i - \xi$

sodass $\eta \phi_i(\eta) \rightarrow \eta$

Für jedes $\phi \in \Theta_A(\cale)$ gebe es ein netz $\eta_i \in \cale$ 

sodass $\phi(\eta_i) \phi \rightarrow \phi$

$$\xi \otimes \eta \;
\psi \Big (\pi\big(\phi(\xi)\big) \eta \Big)$$

$$\xi \phi(\xi) \otimes \eta \;
\psi \Big (\pi\big(\phi(\xi)\big) \eta \Big)$$

$$
\psi \Big (\pi\big(\phi(\xi)\big) \eta \Big)
\psi \Big (\pi\big(\phi(a)\big) b \Big)$$

$$
\psi \Big (\pi\big(\phi(\xi)\big) \Big( \eta 
\psi \Big (\pi\big(\phi(a)\big) \Big ) b \Big)
\Big)
$$

vllt reicht links approximierende einheit

\section{}

$\calk(E\otimes F)  \rightarrow \calk(E \otimes_f F)$

$T_i \otimes S_i$

$T(x b) \otimes S(y)= T(x) \otimes f(b) S(y)
=T(x) \otimes S(f(b)y)$
 
wenn approx einheiten quasizentral

setze $T \otimes S$ null auf kernel 
($\otimes_f$) 
sodass wohldef

\section{}

$\xi \phi(x) \otimes \eta \psi (y)$

$\xi \phi( \eta \psi(x)) = \xi \phi(\eta) \psi(x)$

\section{}

brauche apporx in $\calk$ denn brauche
in $\call \square A$

auch nicht ?


erstens für $M_2$-aktion 

zweitens für $\nabla$-rechung

\section{}

$\phi:A \rightarrow \calk_A(A)$
$\phi(a)(b)= a b = a s(b)$ mit $s=id \in \Theta_A(A)$

\section{}

für alle ringe $A$ ist
$\sum_i a_i b_i$ dicht in $A$ (endliche summen, $a_i,b_i \in A$)

für alle $A$-module ist $\cale A$ dicht in $\cale$

\section{}

$< \xi , \eta>^* a = < \eta , \xi> a$

$< \eta, T(a)> = < \eta, \xi a> $

\section{}

A $G$-action on a ring $A$ is 

A $G$-action on a $(A,\alpha)$-module $M$
is a group homomorphism
$S:G \rightarrow {\rm Aut-group}(M)$
such that $S_g(m a)= S_g(m) \alpha_g(a)$

$T_g(\theta)\big (S_g(x) \big ) = \alpha_g\big (\theta (x)\big)$

\begin{example}

$A$ ring, $\theta(a)=x a$

$T_g(\theta)(a)= \alpha_g(x) a$

$T_g(\theta)(\alpha_g(a)) = \alpha_g(T_g(a))$

\end{example}

genauso inneres produkt

\section{}

action on $\calk_B(\cale)$

$S_g T S_{g^{-1}}$

$S_g \big (x \phi( S_{g^{-1}} \xi) \big)
= S_g (x)  T_g(\phi)( \xi) $

\section{}

\begin{itemize}

\item 
$\calf$ eine familiy von ringen

\item
$\calm$ eine familie von modulen über den
ringen von $\calf$

abeschlossen unter $R \in \calf$ $\Rightarrow $
$R \in \calm$, $M,N\in \calm$ $\Rightarrow $
$M \oplus N$ in $\calm$, auch unednliche direkte summen, gewissen,
ud
$M,N \in \calm$ dann auch $M \otimes_f N$ in $\calm$

\end{itemize}

\begin{lemma}
$\calk_B(\cale)$ has approximate unit. 
$\rightarrow$ axiom 
\end{lemma}

\begin{proof}
Let $X$ be a generating set of $\cale$ as a $B$-module.

Let $F$ be the set of finite subsetes of $X$. 

Let $\phi \in \Theta_B(\cale)$.

Set $p_Y:= \sum_{x \in Y} \theta_{x,\phi}$
for $Y \in F$.

\end{proof}

\begin{definition}
{\rm
A {\em corner embedding} is 
an injective 
ring homomorphism
$$e: A \rightarrow  \calk_A(
\cale \oplus 
\H_A)$$
defined by 
$e(a) (\xi \oplus a_1 \oplus a_2 \oplus a_3 \oplus \ldots)
= 0_\cale \oplus  a a_1 \oplus 0_A \oplus 0_A \oplus \ldots $. That is,
$e(a)$ is the multiplication operator 
acting on the distingushed coordinate $(A,\alpha)$ 
of $(\cale \oplus \H_A,S)$.

}
\end{definition}

\begin{definition}
{\rm
A {\em corner embedding} is 
an injective 
ring homomorphism
$$e: A \rightarrow  \calk_A(
\H_A\oplus \cale)$$
where 
for all $a \in A$, $e(a)$ 
is the topological limit of one-dimensional compact operators 
of the form $e_i(a)= \xi_i a \phi_i(\pi(\cdot))$, where
$\phi_i \in \Theta_A(\H_A)$ and $\xi_i \in \H_A$, and $\pi:\H_A\oplus \cale \rightarrow \H_A$ is the canonical projection. 
So $e(a)= \lim_{i \rightarrow \infty} e_i(a)$.
}
\end{definition}

The typical example is the corner embedding
where $e$ is the limit of the operators $e_i$, where 
$e_i(a)= (p_i,0,0,\ldots) a \phi(\cdot )$, for an approximate unit $(p_i) \subseteq A$
and where $\phi:\H_A \oplus \cale \rightarrow A$ is projection onto the first coordinate $A$
of $\H_A$.
%


\begin{definition}

To every module $\cale$ in $\moduleright B$ is associated a subset
of $B$-linear functionals  $\functionals B \cale
\subseteq \modulehomomorphisms B \cale B$

We define the ideal of komact operators to be
$\calk_B(\cale)$ consisting of all operators of the form 
$T(\xi)= \eta \phi(\xi)$ for all $\xi,\eta \in \cale$ and all $\phi \in \functionals B \cale$.

\end{definition}

\begin{definition}
{\rm

Let $\cale$ in $\moduleright B$
and $\calf$ in $\moduleleft C$
and   
$\pi:B \rightarrow \moduleendomorphisms C \calf$ a ring homomorphism.

Then we set 
$\functionals C {\cale \otimes_\pi \calf}
$
to be the set of all $\theta$ in $\modulehomomorphisms C
{\cale \otimes_\pi \calf_C} C$
of the form
$$\theta(\xi \otimes \eta) = \psi \Big (\pi\big(\phi(\xi)\big) \eta \Big)$$
for some $\phi$ in $\functionals B {\cale}$ and $\psi$ in $\functionals C {\calf}$.

$$V_g(\theta)(\xi \otimes \eta) = T_g(\psi) \Big (\pi\big(S_g(\phi)(\xi)\big) \eta \Big)$$

}
\end{definition}

\section{}

erlaube $\Theta_A(A)$ für ring $A$ nur multiplikationsoperatoren mit multiplieren von multiplier algebra 

oder besser nur multipliers in $A$ und $1$

---

was wenn etwa $\cale \otimes_f \calf$ zufällig ring $A$ wird, und $\Theta_A(A)$ davon nicht 
$\Theta_A(\cale \otimes_f \calf)$

man könnte $\Theta$ für ringe und moudle separieren, jedoch wird gerade dieser 
fall in morita equivalnz verwendet 

---

zwei module sind gleich wenn bijecktion $\phi$
und dieses $\phi$ auch bijektion auf $\Theta$ macht, mit eigener bijektion.

hier unten definiere 
$B \rtimes G$ , das $\Theta$ darauf wie üblich


$a_g \rtimes g=0$ in a ring 

$\cale \otimes_B (\oplus^n \calf_n )  \rightarrow \oplus^n \;\; \cale \otimes_B \calf_n$ 
als abelsche gruppen 


\begin{lemma}
Let $A,B$ be rings where $A$ is a subring of $B$.

Then there is a ring homomorphism
$$\phi:\moduleendomorphisms B \cale \rightarrow
\moduleendomorphisms A {\cale \cdot A}: \phi(T)= T \circ E$$
where $E(\xi)= \lim_i \xi a_i$
for an approximate unit $(a_i)$ of $A$,
and $E \in \moduleendomorphisms A \cale$.

\end{lemma}

We use this lemma for 

\begin{lemma}

Let $B$ have an approxiamte unit. 

for every $\xi$ there is a $b \in B, \zeta \in \cale$ such that
$\xi = \zeta b^2$

oder mit summen von solchen

oder $\xi = \zeta b_1 b_2$

---

There is a ring isomorphism
$$\pi:\calk_{\calk_B(\cale \oplus B)}(\calf) \rightarrow
\compacts B {\calf 
E }: \phi(T)= T 
|_{\calf 
E}$$
where 
$E \in \compacts B {\cale \oplus B}$
is the canocial projection onto the coordinate $B$, 
and the additive subgroup $\calf  E: = \{\xi E \in \calf | \xi \in \calf\}$ of $\calf$ 
is turned into a $B$-module by setting
$\xi E \cdot b := \xi E m_b$ 
for $b \in B$ and where $m_b \in \calk_B(\cale \oplus B)$ is the corner operator
$m_b(\eta \oplus d)= 0 \oplus bd$.  
The $G$-action on $\calf E$ is
the one induced by restriction. 
%

%

---

vllt besser:

$\calf M_B$, wo $M_B$ die corner operator algebra ist $m_b$

->

$z(\eta)= \xi \phi(\eta)
= \zeta b^2 \phi(\eta)
= \zeta m_b( b \phi(\eta)  \oplus 0_\cale)
$

$
=  (\zeta \psi) \circ m_b \circ ((b \oplus 0_\cale)\phi) (\eta)
= x m_b y
$

\end{lemma}

\begin{proof}

Set $\calf E:= \calf M$, wo $M$ bild of corner embedding 
$e: B \rightarrow \call_B(\cale \oplus B)$

Sei $z \in \compacts B {\cale \oplus B}$

Dann
$z(\eta)= \xi \phi(\eta)
\approx \xi b^2 \phi(\eta)
= \xi m_b( b \phi(\eta)  \oplus 0_\cale)
$

$= \xi \psi (\phi(\eta) \oplus 0_\cale)
\approx \xi \psi( (b \oplus 0_\cale)\phi(\eta))
\approx \xi \psi (m_b( (b \oplus 0_\cale)\phi(\eta)) ) $

$
=  (\xi \psi) \circ m_b \circ ((b \oplus 0_\cale)\phi) (\eta)
= x m_b y
$

wo $\psi \in \Theta_B(\cale \oplus B)$ projektion auf koordinate $B$ 


jedoch scheint alles mit approx einheit, 
da für alle $\eta$ gelten muss 

---

man müsste sogar limes $x m_b y$ nehmen 

\end{proof}

$b^2 in \calk_B(\cale)$ :

$a b (\eta)= 
\xi \phi(\xi' \phi(\eta))
= \xi \phi(\xi') \phi(\eta)$  

\begin{proof}

We have 
$$\xi  \phi(\eta )=
\lim_i  
(\xi \oplus b) E( b_i \phi(\eta ))$$
 
for all $\xi,\eta \in \cale \oplus B$ such that
every $z \in \compacts B {\cale \oplus B}$
is presentable as a sum of products of the form $x E y$ for some
$x,y \in \compacts B {\cale \oplus B}$

Injective: Assume that $\pi(T)= T \circ E=0$. Then
$T(\xi E)= 0$ for all $\xi \in \calf$.

Given any $\xi$,
Choose a $z \in \compacts B {\cale \oplus B}$
such that $\xi z = \xi$.

Then $T(\xi z)= 0$ because $T(\xi x E y)=
T(\xi x E)y =0$ by what we said above.

Thus $T=0$.

Surjective:

Given $S \in \compacts B {\calf E}$
we define $T \in \compacts {\compacts B {\cale \oplus B}} \calf$ by
$T(\xi \sum_i x_i E y_i):= \sum_i S(\xi x_i E) y_i$, where $\xi \sum_i x_i E y_i = \xi$.

Define

$T(\xi E z):= S(\xi E) z$



\end{proof}

\begin{remark}
Funktoriellen eigenschaften von $\phi = \phi(\pi)$
\end{remark}



We have seen that we can successively eliminate extended double split exact sequences with a then final homomorphism $P$ eliminating them all.
So we may remark that we can then eliminate words simultanously:

\begin{lemma}
If $z_1 , \ldots , z_n
\in GK^G(A,B)$ then there is a ring homomoprhism $P$, right-invertible by some morphism $Q$ in $GK^G$, so $P Q =1$ in $GK^G$, such that $z_1 P,z_2 P, \ldots, z_n P$ are in $L_0 GK^G$.
\end{lemma}

\begin{definition}

Let $\F$ be the free category with $\pm$ generated by the ring homomorphisms and the $\Delta$s and $e^{-1}$s.

Then we can select a $P_z \in \F$ such that
$z P_z$ is level-0 in the quotient $GK^G=\F/\equiv$, and there is a $Q_z \in \F$ such that $P_z Q_z = 1$ in $GK^G$.

\end{definition}

\begin{definition}
{\rm  
We have seen above that we may assign to each 
morphism $z \in \F$ 
a morphism $w \in \F$ by a finite and well-defined algorithm such that 
\begin{itemize}

\item
$[w] \in L_0 GK^G$

\item
$[z] = 0$ if and only if $[w] = 0$ in $GK^G$

\end{itemize}

Let $\Phi: \F \rightarrow L_0 \F$ denote this assigment, thta is, $\Phi(z):= w$.

Recall that more precisely $w = z P$ for a $GK^G$-rightinvertible ring morphism $P$ depending on $z$. 

Define an equivalence relation on $L_0 \F$ by saying $\Phi(z_1) \equiv \Phi(z_2)$ in if 
$[z_1] = [z_2]$ in $GK^G$. 
 
We get then an assignment $\Psi: GK^G \rightarrow L_0 \F/\equiv$ given by $\Psi([z])= [\Phi(z)]$.

$$\xymatrix{  GK^G   \ar[rd]^\Psi   \\
  L_0 GK^G   \ar[r]^q   \ar[u] &  L_0 \F /\equiv
}
$$
}
\end{definition}

ist $\Psi$ überhaupt injectiv? 

---

jedoch geht das gar nicht, etwa $z_1,z_2 \in L_0$, es gibt zwei $P_1= 1, P_2 \in L_n$

kann nicht $z_1 \equiv z_2 P_2$ setzen

nur $L_0 GK^G$ das richtige7
\begin{lemma}
If $X$ is a finite subset of $GK^G$
then there is an injective map 
$X \rightarrow L_0 GK^G$
\end{lemma}

Kann immer eine injective abb $GK^G \rightarrow L_0GK^G$ wählen 

könnten gleichb sein, zb,  $z, z p_1 \rightarrow z p_1 p_2, z p_1 p_2$ gleich

also lemma falsch

\section{}

vllt unterschied zwischen 
$A \otimes B$ und $A \otimes_\Z B$

\section{}

\begin{definition}
{
Define $f:A \rightarrow A \otimes \K
 \cong \compacts A  { \ell^2(G,A)}$
by
$f(a) \xi = v a \langle v,\xi\rangle$
where $v$ is invariant vector in $\ell^2(g,A)$,
and $\xi \in \ell^2(G,A)$.

$f(a)(b \otimes \xi)= a b \otimes w \langle w,\xi\rangle$
where $a,b \in A, w,\xi \in \ell^2(G)$

$= a \langle 1, b\rangle \otimes w \langle w,\xi\rangle
=
= a \otimes w \langle 1 \otimes w,b \otimes \xi\rangle$
}
\end{definition}

\begin{definition}
{
bei ringen:

$\phi:A \oplus \ldots A \rightarrow A: \phi(a_1, \ldots,a_n) = a_1$

$f(a)(a_1,\ldots,a_n)= (a,0,...,0) \phi(a_1,\ldots,a_n)$
$\approx (b_i,0,...,0) a \phi(a_1, \ldots,a_n)$

mit approx einheit $(b_i)_i \subseteq A$

}
\end{definition}

\begin{definition}
{
jedoch nicht invariant vektor oben

so:
$\phi(a_1, \ldots,a_n) = (1/n)\sum_i a_i =:v$

$v a \phi()$

$v a \phi(v b \phi(x)) = v a \phi(v) b \phi(x)
\approx v a b \phi(x)$

mit $v:\approx (b_i,b_i, .... ,b_i) /n \approx (1,...,1) /n$
  
dann $\phi(v)=1$

---

kann hier einheit $1 \in \F$ nehmen,
da kroern erinbettung $\F \rightarrow M_n(\F)$, 
und dann erst $1_A \otimes f$ 

}
\end{definition}

\section{}

wie ist $f$ kornereinbettung wenn $G$ compacte gruppe

vllt einfach $\F^\infty = \ell^2(G)=\ell^2(\N)$ mit unendlich dim onb

\section{}

\begin{definition}
{
jedoch kornereinbettungsdefinition:
"eindimensionale kompakte operatoren"

$a \mapsto v a \phi(.)$

bzw. limiten davon

sodass bijectiv auf $A$

---

kann darauf verzichten:

nehme $\C = \F$ kommutativer körper (zb $\Q,\Z_p$)

dann kann in $\ell^2(\F)$ basis ergänzen

wird es aber $\F^n$ mit der direkten summe aktion von $G$,
$\rightarrow$ brauche nicht, denn aktion nicht diagonal

}
\end{definition}

\section{}

was ist minimal projektion in $M_n \otimes \F$ ?

wie homotop dort
 
min proj =ein dim proj so wie oben?

\section{}

\begin{lemma}
There  is an isomoprhism 
$$\pi: \calk_B(\cale 
) \otimes D
\rightarrow 
\calk_{B \otimes D}(\cale \otimes D 
) 
$$
$$
\pi(\theta_{\xi,\phi} \otimes d) = \theta_{\xi \otimes d, \phi \otimes m_d}$$

where 
$m_d \in \Theta_D(D)$ is the multiplication operator $m_d(d_2)= d d_2$.
and $\xi \in \cale , \phi \in \Theta_B(\cale)$. 

\end{lemma}

If $G$ is a compact group then
we may choose $\underline E G$ to be a one-point space and so the above Baum-Connes map $\nu$ is exactly the Green-Julg map.

bc für diskrete vector spaces über fields:

die poroper spaces $X$ diskret (?)

kann aber auch alle zulassen

\begin{remark}
{\rm 
If $A$ is a $C^*$-algebra 
then then the usual baum connes map $\nu_{C^*}$ is only
the composition of the above map $\nu$ with the 
induced map $f_*$ for the ring homomorphism  
$f:A \rtimes G \rightarrow A \rtimes_{C^*} G
:= \overline{A \rtimes G}^{C^*}$.
That is, $\nu_{C^*} = \nu f_*$. 

Everything in $KK^G$, that is, $GK^G=KK^G$. 

stimmt nicht, da $GK^G$ groupen größer wenn man auch nicht abgeschlossene $C1*$-algebrwn zulässt 
}
\end{remark}

\begin{lemma}
If the $KK$-theory $C^*$-axioms are incoorporated in $GK^G$-theory then there is a canonical functor
$\zeta: KK^G \rightarrow GK^G$.

kommutiert mit deszent, tensor

---

jedoch auch ev schlecht für $\nabla$-rechnung, denn man muss ja auch die zusätzlichen $C^*$-kornereinbettungen wegbringen 

---

die $C^*$-e s könnte man aber auch separat wegbringen mit $e \rightarrow e e^{-1}$

-->
n, denn man muss homomoprhisemen mit kornereinbettung skippen, und hier mischung chlecht, siehe lemma \ref{lemma43}

---

beachte aber auch tensor product und descent functor

---

zudem hat man zwei $\call_A$ und zwei komakte räume $\calk_A$ 
wenn man alles von $C^*$ mitnehmen will 

-> könnte aber passen, da man ja glieche module mit unterschidlichen $\Theta$ unterscheidet

\end{lemma}
 
jedoch könnte es mit dem functor gehen, wenn
strong bc stimmt 

\begin{lemma}
Let the strong Novikov conjecture be true,
that is, there is a map
$\rho$ such that $\nu_{C^*} \rho = \idd$ in $KK^G$.

Then for every $C^*$-algebra $A$,
$\nu^{G,A}$ is injective.

stimmt wieder nicht, da $GK^G$ größer ist
als $KK^G$ 

jedoch level 0  

jedoch ring homomoprhism und $*$-homomoprhism

\end{lemma} 

\begin{proof}

$$\xymatrix{ \lim_{X}
L_0 GK^G(C_0(X),A) \ar[r]^{\nu^{G,A}}  &   L_0 GK(\C,A \rtimes G)
\ar[r]  & L_0 GK(\C,A \rtimes_{C^*} G)
}
$$

jedoch $L_0 GK^G \neq L_0 KK^G$

deswegen stimmts nicht

habe aber immerhin $\zeta(\rho)$ 

$\zeta(\nu_{C^*}) \neq \nu$
\end{proof}

frage ob $\rho$ durch multiplication mit einem morphism realisiert

jedoch $GK$ und $GK^G$ 

zudem nutzt man sicher etwa eaxaktheit von $KK$ im $KK$-beweis 

\begin{example}

$$\xymatrix{ B \ar[r]^e \ar[d]^\pi &  \calk_B(\cale \oplus \H_B)  \ar[d]^{\phi}    \ar[r]
& L^p(\R) \cap L^q(\R) \square A  \ar[d]^{\psi} 
& K^p(\C)  \ar[l]^{s_\pm}  \ar[d] \\
D \ar[r]^{f} &    \ar[r]
&
& A  \ar[l]^{t_\pm} 
}$$

jedoch hat man nicht double abbildung $s_\pm$

zudem kann man immer free konstruktion

\end{example}

\section{}

bei uvf , lemma 7.3 ganz anfang (format der aktion)
brauche approx einheit

bei lemma 6.3.(iii) brauche positive $b^2$ spannt alle elemente,
durch gebrauch in lemma 7.5

brauche auch bei lemma 8.3 $b^2$ spannt alg

\section{}

$A$ $A,A$-bimodule

$\Theta_A(A), A_\Theta(A)$ 

$a \phi(b) = a x b = \psi(a) b$

\section{}

\begin{definition}

A ring $A$ is called ... if
$A = {\rm lin}_\Z \{a b|a,b \in A\}$

bzw closure davon

\end{definition}

wenn $A$ approx einheit hat, dann erfüllt:

$a = \lim_i a a_i $

\section{}

was mit rotation in ringen? 

wenn man nur polynom homotopy hat

zumindest äquivalenz von ecken einbettung in $M_2(A)$ als aiom in $GK$ 

---

eher mehr, denn in lemma braucht man 
rotation in $f(B) C \oplus C$

oder doch $M_2 (f(B))$   
(?)

\section{}

was mit widersrpuchsfreiheit von axoim eckeneinbettungen 

etwa wenn $A=0$ dann ist kornereinbettung auch nicht invertierbar 

kann man in kategorie für einen morphismus $w \neq 0$ immer $w$ invertierbar fordern ?

bei $A=0$ ist $1_A =0$

\section{}

abgeschlossenheit von $M \square A$. sum ideal und algebra

\section{}

nimmt man körper als ringe, so haben sie einheit 

\section {}

allg probelm bei top ringen -> stetigkeit von ring homomorphismsen

und in der regel bild nicht abgeschlossen, sodass man umkehrabbildung braucht, etwa bei ring isomorphismen, die ebenfalls stetig, 
damit folgenkonvergenz im domain ->  folgenkonvergenz im range und vice versa 

zb lemma \ref{lemma121}, wenn gruppe $G$ unendlich, welches crossed product, verschiedene abschlüsse etc.

\section{}

man könnte topologische ringe, wie top vektorraumalgebren, zulassen, mit stetigen abb,
aber keinen top abschluss bei compacten nehmen

\section{}

bei green julg könnte man top vektoraumalgebren über $\R$ und $\C$ zulassen, da man nur einfache matrixeinbettungen als  kornereinbettung hat 

also homotpy 

\section{}

brauche vmtl kardinality restriction von classe der ringe damit Hom-classen Hom-sets sind 

wie ufert homomorphisme $f:A \rightarrow A$ aus wenn über beliebig viele verschiedene $B$ geht : 

$A \rightarrow^f B \rightarrow^{f^{-1}} A$ in $GK(A,A)$  

-> kann aber fusionieren : $f f^{-1} = 1$

\section{}

axoimatisch $\H_B$: 

brauche dass $\H_B$ als direckte summe mit internal tensorprodukt vertauscht im lemma von descent functor

\section{}

bei GJ scheint es zu genügen dass $A$ vr über $\F$ und $\F$ selbst zulässiger ring, 
und nicht dass alle ringe vr sein müssen

\section{}

\begin{lemma}
gegeben 
$$\xymatrix{ \calk_A(\cale \oplus \calf \oplus \H_A)}$$

cutdown $\calf$ 

\end{lemma}

cutdown ist lemma \ref{lemma10}

man kann auch umgkehrte abbildungen nehmen, also projektion $P$ als cutdown und dann umgekhrte richtung 

man kann \ref{lemma10} auch umgekhrt lessen, also überflüssigen teil löschen 

$\phi f^{-1} = \phi' f^{-1} =  \phi_1 \phi_2 f^{-1} = \phi_1 F^{-1}$ mit lemm \ref{lemma281} 

wo $\phi'$ das rotierte $\phi$ ist

\section{}

muss $\call_A(A)$ so definieren:

alle $\phi \in \Hom A(A)$ sodass $m_a \circ \phi = m_b$ für multiplikationsoperator $m:A \rightarrow \call_A(A)$ 

also, so, sodass $\calk_A(A)$ ideal wird in $\call_A(A)$, wenn $\calk_(A)$ genau die multiplikationsoperatoeren

---

allgemain:  $\call_A(\cale) \subseteq \Hom A(\cale)$ so groß, sodass $\calk_A(\cale)$ ideal wird

	\section{}

definiere so: $\call_A(\cale)$ subalgebra in $\Hom A(\cale)$ sodass $\calk_A(\cale)$ ideal

denn, mit obigen $\Theta_A$ funktioniert $\call_A(A)$ nicht: $\phi=1_A$ , $T$ multiplier von $A$, dann $\phi \circ T = T$ nicht in $\Theta_A(A)$

---

man könnte theoret aber auch auf $\phi=1$ in $\Theta_A(A)$ verzichten, denn $a x= \sum a_i b_i x$ ($a,x,b \in A$) nach axiomat vorausetzung 

aber obiges allgemeiner, besser

\section{}

auch auf kommutativen braucht es i.a. keinen split geben: 

etwa $X=$ kreisscheibe, fkt auf $X$ kann nicht einfach auf $\C$ ausdehnen , ausseerhalb von $X$ konstant machen, da auf rand von $X$ viele verschidene werte

\section{}

kurze exacte folge

$GK(A,X) \rightarrow GK(B,X) \rightarrow ...$

bringe auf level 0

wenn level 0 exact -> obiges exact

da $(i^* z) P = i^* (z P)$

\section{}

chern character

künneth 

k-homology berechnungn (chern)

rstriction von gj und bc auf level-n

twisted equivariant theroy
(G/N )

\section{}

surjectivity gj (homotopy) 
für algebren über körper die $\R$ enthalten

\section{}

$\phi( T x)= \psi(x)$ adjungiertheit

-> $\phi \circ T = \psi$

\section{}

\begin{lemma}

Let $A,B.C$ be an exact sequence such thaat

Let for every $X$ the contravariant functor $A \mapsto L_0 GK^G(A,X)$ be exact. 

Then the contravariant functor $A \mapsto GK^G(A,X)$ is exact for every $X$.

\end{lemma}

\section{}

Functor $L_0 GK^G \rightarrow GK^G$ 

$x y P S$

\section{}

es gibt auch ring homomorphism $P:B \rightarrow D$ um $P_*$ von $GK(A_i,B)$ $i=1,..,n$ simultan nach $L_0 GK(A_i,D)$

\section{}

homotpy von level 1 zyklen auch für ringe 

zudem

$s_+ \nabla_{s_-} = t_+ \nabla_{t_-}$

wen $s_\pm \equiv t_\pm$ homotop

\section{}

$$\calk_B( \cale \otimes \cale ) \otimes \calk_\C(\C) 
\cong \calk_{B \otimes B \otimes \C} (\cale \otimes \cale \otimes \C) 
\cong \calk_{(B \otimes \C) \otimes_\C ( B \otimes \C)} 
((\cale \otimes \C) \otimes_\C 
(\cale \otimes \C) )$$

Homotopy:

complexification

$B_\C:= B \otimes \C$

$f:A \rightarrow B_\C[0,1]$ 
homotopy 

wo auswerung in endpunkten 
$f(a)(0),f(a)(1)$ in $B \otimes \Z$

combiniert mit dem ring homomorphism

$\tau: B \otimes \Z \rightarrow B: \tau(b \otimes n)= b n$

topology auf $B \otimes \C$ discrete $\otimes$ $\C$

\section{}

\section{}

bei kronereinbettung muss $\cale$ wahrscheinlich immer fulles module sein 

\section{}

für alle $\xi \neq  0 \in \cale$ gibt es ein $\psi \in \Theta_A(\cale)$ und ein $a \in \cale $ sodass
$\xi \psi(a) \neq 0$ 

für alle $\xi \neq 0$ gibt es ein $\psi$ sodass 
$\psi(\xi) \neq 0$ ,,
gibt es $\eta$ sodass $\eta \psi(\xi) \neq 0$

->

für jedes $\xi \neq 0$ gibt es ein $a$ sodass $\xi a \neq 0$

für jedes $a \in A$ gibt es $\psi$ und $\xi$ sodass $a=\psi(\xi)$  

für jedes $a \neq 0$ gibt es ein $\xi$ sodass $\xi a \neq 0$

\section{}

\begin{definition}

Two rings $A$ and $B$ are called Morita
equivalent if there are a functional
$B$-module $\cale$ and a functional $A$-module 
$\calf$,
ring homomorphisms $m: A \rightarrow \calk_B(\cale)$, 
$n :B \rightarrow \calk_A(\calf)$ 
such that 
$\cale \otimes_n \calf \cong A$ 
as $A,A$-bomdule, right functional .
And
$\calf \otimes_m \cale \cong B$ 
as $B,B$-bomdule, right functional .

And these is a isomorphism 
$\calf \rightarrow \calk_A(\cale,A)$
as $A,A$-bimodule , right functional 

\end{definition}


$$\xymatrix{ A \ar[r]^m  &  \calk_B(\cale)   \ar[r]^\phi
& \calk_A(\cale \otimes_B \calf)
\ar[r]^\cong
& \calk_A(A)  
\ar[r]^\cong
& A  
}$$

We have an isomorphism
$$r: 
\cale \otimes_n \calk_A(\calf,A) 
= \cale \otimes_n \Theta_A(\calf) 
\rightarrow
\cale \otimes_n \calf $$
of $A$-modules (in the domain $A$-module strucuture  is still given by $M_A$) 
defined by 
$r(\xi \otimes \theta_{\eta, \phi}) = \xi \otimes \eta \phi(1)$, where 
$1$ is here just used as a dummy and not really needed. 

$$\cale \otimes_n \calk_A(\calf,A) 
 \cong \calk(A,\cale) \otimes_n \calk_A(\calf,A) 
$$

$$\calk_A(\calk_A(\calf,A))
$$

$$
\calk_{A} \Big (
\Big ( 
\cale \otimes_N \calk_A(\calf \oplus \H_A)  \Big )  M_A 	
\Big )
\cong
\calk_{A} ( 
\cale ) \otimes_N 1
$$

$$r:\cale \otimes_B \calf \rightarrow 
\cale \otimes_B \calk_A(A,\calf)$$
$$r(\xi \otimes \eta)= \xi \otimes f$$
wo $f(a) = \lim_n \eta a$. 

$r^{-1}(\xi \otimes f)=  \xi \otimes \lim_n f(p_n)$
	$\qquad$ (brauche hier keinen limes) 

-

$r^{-1}(\xi \otimes \eta \phi_d) = \xi \otimes \eta d$
 
wo $\eta_d(a) = d a $

-

Damit
$$\cale \otimes_B \calk_A(A,\calf \oplus \H_A)
\cong
\cale \otimes_B \calk_A(A,\calf)
\cong
\cale \otimes_B \calf
\cong
A
$$

\section{}

\begin{proof}
  $\psi \circ T = \phi(x)$ for $\psi(\eta)= x \eta_1 \oplus 0^{n-1}$ and $\eta \in \F^n$. 

Since $\ker$

Complete 
$\{m\}$ to a vector space basis 
$\{m,b_2 ,\ldots ,b_n\}$ of $\ell^2(G)$. 
Note that $m$ is a $V$-invariant vector. 
Hence, 
$$(\ell^2(G),V) \cong (\F m \oplus \F b_2 \oplus \ldots \F b_n\} ,W)$$ 
where the $G$-action is defined in such a way that this isomorphism is $G$-equivariant, and $\F m$ is the distingushed coordinate. 
 
\end{proof}

\begin{definition}
Define
$$f: \F \rightarrow M_n(\F) : f(x)= x \sum_i 1/n (e_ii)$$

$$\pi: \calk_\F(\F^n) \rightarrow M_n(\F_n)$$

$$\pi^{-1}(x_{ij} e_{ij}) ((y_1,\ldots,y_n)) = \sum_{k=1}^n x_{ik} y_k$$  

wähle basis in $\F^n$ sodass $\F^n \cong {\rm span} \{f(1) , b_2,b_3, \ldots ,b_n \}$
 als vr
\end{definition}

Da $f(1)$ invariant uter der akion $\Ad(V)$, ist $f$ equivariante korerner einbettung


Dann
$$A \otimes_\F M_n(\F) \cong \calk_A(A) \otimes_\F \calk_\F(\F^n) \cong \calk_{A \otimes \F}( A \otimes_\F \F^n)$$
$$\cong \calk_A(A^n)$$
mit lemma \ref{lemma1323}


\section{}

\begin{proof}
  $\psi \circ T = \phi(x)$ for $\psi(\eta)= x \eta_1 \oplus 0^{n-1}$ and $\eta \in \F^n$. 

Since $\ker$

Complete 
$\{m\}$ to a vector space basis 
$\{m,b_2 ,\ldots ,b_n\}$ of $\ell^2(G)$. 
Note that $m$ is a $V$-invariant vector. 
Hence, 
$$(\ell^2(G),V) \cong (\F m \oplus \F b_2 \oplus \ldots \F b_n\} ,W)$$ 
where the $G$-action is defined in such a way that this isomorphism is $G$-equivariant, and $\F m$ is the distingushed coordinate. 
 
\end{proof}

\begin{definition}
Define
$$f: \F \rightarrow M_n(\F) : f(x)= x \sum_i 1/n (e_ii)$$

$$\pi: \calk_\F(\F^n) \rightarrow M_n(\F_n)$$

$$\pi^{-1}(x_{ij} e_{ij}) ((y_1,\ldots,y_n)) = \sum_{k=1}^n x_{ik} y_k$$  

wähle basis in $\F^n$ sodass $\F^n \cong {\rm span} \{f(1) , b_2,b_3, \ldots ,b_n \}$
 als vr
\end{definition}

Da $f(1)$ invariant uter der akion $\Ad(V)$, ist $f$ equivariante korerner einbettung


Dann
$$A \otimes_\F M_n(\F) \cong \calk_A(A) \otimes_\F \calk_\F(\F^n) \cong \calk_{A \otimes \F}( A \otimes_\F \F^n)$$
$$\cong \calk_A(A^n)$$
mit lemma \ref{lemma1323}

\section{}

$$\sigma (\int_{g \in G} (a \otimes k) W_g dg
= (a \otimes k) (1 \otimes U_{g^{-1}}) W_g dg$$

\section{}

\begin{proof}
Clearly, $\C^n=\ell^2(G)$. 
Since $\lambda_\C(p)$ is a minimal projection in $\calk (\C^n)$
\end{proof}

\begin{lemma}
Let $p \in \F \rtimes G$ be a 
projection 
such that $\lambda_\F(p) = \theta_{b,\phi}$ 
is an elementary projection. 

$\phi(\xi)=< \xi,b>$, $b_1=1/n$ 

Then $S \circ T (p) =  p$.  

\end{lemma}

\begin{proof}
Since $\lambda_\F$ is $G$-invariant, $\lambda_\F(p)$ is $G$-invariant with respect to 
$\Ad(V^{-1})$. 
As $\phi$ is of the form $\phi(\xi)= \sum_{g \in G} \xi_g c_g$

Thus $g \circ \theta_{g ,\phi}(g^{-1}(\xi))
= g(b) g (\phi)(\xi)= \theta_{g(b),g(\phi)}$
$=\theta_{b,\phi}$. 
So $g(b)= b$ and $ $

Since $p$ is 
So we have 
$$p (\xi)= \sum_{g \in G} a_g U_g (\xi) = \phi(\xi) b = \sum_{g \in G} \xi_g b_g b$$

Consider a minimal projection in $A \otimes \K$, so of the form 
$$p (\xi)= \int_G a(g) U_g d g = \phi(\xi) b$$

$$p (\xi)= \sum_{g \in G} a_g U_g = \langle \xi, b \rangle b$$

We choose on $A$  $\Theta(A)$
to be the functionals by right multiplication.
$\langle \xi,a\rangle := \xi a$.

We choose on $\ell^2(G,A)$ $\Theta(\ell^2(G,A))$
to be th inner products
$\langle \xi,\eta\rangle := \sum_g \xi_g \eta_g$

$$p (\delta_h)= \sum_{g \in G} a_g \delta_{gh} = b_h b$$
thus
$$a_g = b_1 b_g \qquad \forall g \in G$$

$$a_g = b_h b_{gh} = b_1 b_g \qquad \forall g,h \in G$$

((

$\rightarrow \quad $
$a_1 = b_h b_{h} = b_1 b_1 $

$a_g = b_{g^{-1}} b_1 = b_1 b_g$

))

$$M (L \rtimes 1) (\lambda_A \rtimes 1) \sigma (1_\C)
= n^{-1}\sum_{g,h} \pi(a_g) U_g U_h  \rtimes h
= n^{-1}\sum_{g,h \in G} \pi(a_{g h^{-1}}) U_g  \rtimes h$$
$$=
n^{-1}\sum_{g,h \in G} \pi(b_1^{-2}) \pi(a_h) \pi(a_{g}) U_g  \rtimes h
= 
n^{-1} \sum_{h \in G} \pi(b_1^{-2})  \pi(a_h) p  \rtimes h
$$

$b_1^2= 1/n$, da
$b_1^2=b_k^2$ für alle $k,$ und
$<b,b>=\sum_i b_i^2 =1$ 

geht zumindest für $A=\C$,

dann kommutiert $\pi(a)$ mit $p$
und  $a \mapsto \pi(a) p$
ist eine kornereinbettung
 
----
->

$a_g$ ist ja konstant eigenltich, famit nicht
$a_{g h} = a_g a_h$ etc.

zudem, oben ghört noch $1/n$ dazu  ?

-->

braucht nicht konstant, etwa $<\xi,b> b$ invarinate projektion für $b=(1,-1)$ unf $G=\Z_2$

--
erläuterung oben

$a_{g h^{-1}} = b_1 g_{g h^{-1}}
= b_h b_g = a_h a_g * b_1^{-2}$ 

\end{proof}

\section{}

möglichkeit um doubel split zu konstruienen

nehme 1-fach split und bringe auf doube split

schalte irgendenen homomoprhismjus vor,
etwa szbalgebra $N$ von $M$,,

dann $x:N -> M \square M,$ 

$x\square 1, x gt \square 1$

zb  herdeitary unteralgebra 

---

zb weiters:

wähle 2 unteralgebren $X,A \subseteq \call_B(\H_B)$

betrachte $M:= alg(X \cup A)$ algebra 
generiert bei $X$ und $A$ 

mache split mit $M$, mitte ist $M$  

gehe von $X$ nache $M$

\section{}

\begin{lemma}
Suppose that $\F=A=\C$ with trivial $G$-actions. 

(i)
Let $p \in \C \rtimes G \cong M_{n_1} \oplus \ldots \oplus M_{n_N}$ be a 
minimal projection. 
%
%
Then $S \circ T (p) =  p$.  

(ii) 
Conseuqently, $S \circ T= \id$. 

\end{lemma}

\begin{proof}

Clearly, $\C^n=\ell^2(G)$. 
Let us $\C \rtimes G$ identify with its image under $\lambda_\C$ and do everthing in $\calk(\C^n) \cong M_n$. 
Of course, $p$ is also minimal there. 

Since $p$ is a minimal projection in $\calk (\C^n)$, it has one dimensional range, and so there is a vector $b \in \C^n$ 
with $\|b\|=1$ 
such that $p(\xi)= \langle \xi,b \rangle b$. 
As $p \in \C \rtimes G$, we may express it as 
$p = \sum_{g \in G} a_g U_g$ and so
 $$p(\xi) = \sum_{g \in G} a_g U_g (\xi) 
= \langle \xi,b\rangle b 
$$
Evalution in $\xi = \delta_h$ yields 
$a_g = b_h b_{gh}$. 
It follows
\begin{equation}		\label{idab} 
a_g = b_1 b_g  \quad 
a_{gh^{-1}} = b_h b_g = b_1^{-2} a_g a_h 
\quad b_1 b_{g h^{-1}} = b_h b_g 
\quad b_1^2 = b_g^2
\end{equation}  
 Because of the last identity and $\|b\|^2 = 1$,
one gets $b_1^2 = n^{-1}$. 

Set $L:\C \rightarrow \C \rtimes G$ as $L(x)= xp$. 
Then, with $\sigma$ of lemma \ref{lemma1111} 
and the second identity of (\ref{idab}), 
we get (for $F$ see below)  

$$\Big ( M (L \rtimes 1) (\lambda_A \rtimes 1) \sigma \Big ) (1_\C)
= n^{-1}\sum_{g,h} a_g U_g U_h 
  \rtimes h
= n^{-1}\sum_{g,h \in G} a_{g h^{-1}} U_g  \rtimes h$$
$$=
n^{-1}\sum_{g,h \in G} b_1^{-2} a_h a_{g} U_g  \rtimes h
=  \sum_{h \in G}   a_h p  \rtimes h 
= (F \rtimes 1) \big (\sum_{h \in G} a_h \rtimes h \big )
$$

As $\lambda_\C$ is a $G$-equivariant homomorphism, $p\equiv \lambda_\C(p)$ 
is $\Ad(V^{-1})$-invariant. 
Thus $F : \C \rightarrow \C \rtimes G$ 
defined by
$F(x)= x p$ is a $G$-equivariant corner embedding, 
and so 
homotopic to $f$. 

Thus $F \rtimes 1$ is homotopic to $f \rtimes 1$ (descent functor evaluations). 
 
$(F \rtimes 1) \sigma^{-1}$ is a corner embedding und tus hmotopic to 

zudem müsste man $V_g^{-1}$ statt $U_g$ in 
$\sigma$ nehmen

\end{proof}

\section{}

\begin{proof}

We restrict now 
the last line of the above diagram to the image spaces of the down-going $\lambda$-arrows, so they are now all 
bijective. 

Obviously, again everything in the diagram commutes. 

To perform the flip also in the image space $X$ 
of $\lambda_{A \otimes \K}$, we recall 
that it is isomorphic to the fixed point algebra 
$(A \otimes \K \otimes \K)^{\alpha \otimes \rho  
\otimes \rho}$ 

Apply lemma \ref{lemma118} 
to $A:= (B,\alpha \otimes \rho,rho)$ and
$u:= 1 \otimes H$. 

Consider $X \rightarrow X: \lambda_{A \otimes \K}^{-1} \zeta 
\Ad(u) \zeta^{-1} \lambda_{A \otimes \K}$. 


Because of that, the unitary path $H$ of the flip
$F: G \rtimes \K \otimes \K \rightarrow G \rtimes \K \otimes \K$
to the identity $\idd$ can also be 	done
in the later fixed point algebra 
by averaging the unitary path $H$ at every point of time to a
unitary path $I$ with the same end points.

Set $\K:= \calk_\C \big ( (\ell^2(G),\lambda) \big)$
equipped 
with the $G$-action $ \Ad(\lambda)$.
The $G$-action on $A \otimes \calk$ is $\alpha \otimes \Ad(\lambda)$.

$A \rtimes_\alpha G$ is equipped with the trivial $G$-action.

The second line of the diagram we restrict then to the image spaces of the $\lambda$s down arrows.

They are actually the fixed point algebras under the $G$-action $1 \otimes \Ad(\rho)$,
where $1$ is the identity action on the spaces 
before the last $\cdots \otimes \K$, respectively.

So, $\lambda_A :A \rtimes G \rightarrow (A \otimes \K)^{1_A \otimes \rho}$ is a ring isomorphism.

wenn man oben $L \otimes 1$ einschränkt so definiert man
$L \otimes 1$ als $\lambda_\C^{-1} (L \rtimes 1) \lambda_{A \rtimes G}$,
wo hier die $\lambda$s eingeschränkt, bijectiv

\end{proof}

\section{}

Obviously $\pi(T) = S$. 

This shows

then 

$T(\sum_i \xi_i  x_i m_{b_i} y_i) p m_t q = \sum_i S(\xi_i x_i m_{b_i}) y_i p m_t q
= 
 S( \sum_i \xi_i x_i m_{b_i} y_i p m_t ) q
$

(hier schreibe zuerst $b_i = c_i d_i$, bzw wähle range of $y_i$ in $B$)

dh $T$ multiplizeirt mit $p m_t q$ wohldef 

Brauche in ring $\calk_B(\cale \oplus \H_B$ : 

$r t = 0$ für alle $t$ dann $r = 0$ 

-

damit ev. $\phi(\xi)=0$ für alle $\phi$ -> $\xi=0$ für funktionale

\section{}


$M \mapsto \Hom A (M,X)$   

$X= D_1(A)$

$a x := D_1(f_a) (x)$ wo $a \in A,x \in X$

$f_a:A \rightarrow A: f_a(b)= ba$  $A$-homomoprhism

$f_a \in \calk_A(A)$

$\Hom A(X)$ ring und $\calk_A(X)$ ideal

das bild von ideal unter $D_1$ wieder ideal 

---

es gibt in $\Hom \C (H)$ mehrere ideale,
 etwa trace class operators etc. 

equivalenzbimoudle bringt ieale auf idelae

---

functor $D_1$ ist ring isomorphsmus zwiscehn $\Hom A(A)$
und $\Hom B(X)$ 

wenn man zeigen kann compacten in $\Hom A(A)$ 
das einzige ideal ist, dann muss $D_1$ in doe compacten von $\Hom B(X)$ landan

allerdings hat man auch noch natürlicje transformation dazwischen 

---

es gibt schon mehr ideale, zb
ideale $I$ von $A$ in $\call_A(A)$ 

zudem zb: $\calm(C_0(\R))= \calm(A) = C_b(A)$ 

idalie $I$ in $C_b(A)$ sind nicht $A$, teilweise elemente in $A$ und außerhalb $A$

jedoch $C_0(A)$ maximales ideal in $C_b(A)$ 

(zb $I=C_b(\R \backslash \{x\})$, $I=C_b(\R^+)$ )

---

vllt obiges arguemtn abändern auf: 

$\calk$ ist maximales ideal in $\Hom A$

(bzw. $\call_A$) 
 
 ---

functor links-mod-A auf links-mod-B:

$M \mapsto M \otimes_A \cale$ wo

\section{Rings and Modules}

\begin{definition}

$\Theta_A(\cale)$ must be invariant under 
$G$-action on $\cale$:

$\alpha_g$

\end{definition}

Damit mit 2 aktioenn:

$S_g (\xi \phi(T_{g^{-1}}(\eta))
= S_g(\xi) \alpha_g(\phi())=
S_g(\xi) g(\phi)(\eta)$  komalt

\begin{definition}

$G$-action on $(A,\alpha)$-module $(\cale,\Theta_A(\cale))$ muss
erfüllen:  $\calk_A(\cale)$ is invariant.

\end{definition}

\begin{definition}
{\rm 

Let $(A,\alpha)$ be a ring. 
A corner embedding is an equivariant ring homomorphism
$e: (A,\alpha) \rightarrow \calk_{(A,\alpha)} \big ((\cale \oplus \H_A , S \oplus \alpha \oplus R) \big)$ 
of the form 
$$e(a) \big(\xi \oplus a_1 \oplus a_2 \oplus a_3 \oplus \ldots \big )= 0 \oplus a a_1 \oplus 0 \oplus 0 \oplus \ldots $$ 
($\xi \in \cale, a, a_i \in A$),
where $(\cale,S)$ is a full $(A,\alpha)$-module 
and $(\H_A,\alpha \oplus R)$ is the infinite
direct sum 
with distinuished first coordinate $(A,\alpha)$.

}
\end{definition}

note that indeed, $e(a)= \theta_{(0 \oplus a \oplus 0 \ldots) , (0 \oplus \phi \oplus 0 \oplus \ldots)}$, where $\phi \in \Theta_A(A)$ is the identity map.

BEM:

$B^2 \subseteq B$ ist ideal in $B$
und vllt braucht man $B=B^2$ weil man überall fulle module braucht, und am rechten teil steht meist ring

\begin{definition}

Verlange jedes $B$-module cofull (full) in dem sinne:

für alle $\zeta \in \cale$ gibt es $\phi_i \in \Theta_B(\cale)$ und $\xi_i,\eta_i \in \cale$
sodass

$\zeta = \sum_{i=1}^n \xi_i \phi_i(\eta_i)$

strongly co-full, wenn alle $\eta_i=\eta$ 

\end{definition}

Damit ist in $\calk$ alles summe von produkten:

$\theta_{\zeta,\psi}
= \sum_i \theta_{\xi_i,\phi_i} \theta_{\eta_i,\psi}
= \sum_i \xi_i \phi_i(\eta_i) \psi =   \theta_{ \sum_i \xi_i \phi_i(\eta_i),\psi}$

WEITER:

in $M \square A$ ist alles summe von produkten, 
weil es in $A$ gilt und damit in $(s \square 1(A)$, und in den kompakten, und damit in deren summe 

WEITER:

definiere $\call_A(\cale)$ als multiplier zu kompakten, un dann das ideal $\sum B^2$ genommen  

->

automatisch, da einheit $1 \in \call$ 

WEITER:

ideal $\sum B^2$ ist $G$-invariant

\subsection{}			\label{bed1}

Verlange

für alle $\xi \neq  0 \in \cale$ gibt es ein $\psi \in \Theta_A(\cale)$ und ein $a \in \cale $ sodass
$\xi \psi(a) \neq 0$ 

Damit die null-bedingung an K:

sei $\sum_i \theta_{\xi_i,\phi_i} \neq 0$ 

dann

$ 
\sum_i \theta_{\xi_i,\phi_i} \theta_{\eta,\psi} (a)
= \sum_i \xi_i \phi_i(\eta) \psi (a) \neq 0$

\subsection{}		\label{subsection12}

$\cale \phi(\cale)$ fulles module

bedingung so:

$\xi \phi(\eta)= \xi \phi(\xi_i \phi(\eta_i))
= \xi \phi(\xi_i) \phi(\eta_i)
= \theta_{\xi, \phi} \theta_{\xi_i,\phi}(\eta_i)$

brauche strongly cofull 

->

so, wenn obiges ugleich 0, dann it ja zumindse
ein summand ungleich 0, alo
sein $\theta_{\xi, \phi} \theta_{\xi_i,\phi} \neq 0$  

\subsection{}

$M_n(A) \otimes B 
\rightarrow M_n(A \otimes B)$

ohne $\otimes^\C$

\subsection{}

wenn split exact sequence von der form
$\call_B \square$ mit $M_2$-aktion 
$\Ad(S \oplus T)$, dann ideal bedingung 
($G$-invarinace) automatisch erfüllt, weil
ideal compacte operatoren,

\subsection{}

$B^2 = B^3 = B^n$ offensichtlich , wenn $B=B^2$

\section{}

The easy proof of the following lemma 
will be done in lemma \ref{lemma10}. 

\begin{lemma}		\label{lemma281}

If $e:A \rightarrow \calk_A(\cale \oplus \H_A)$ corner embedding then 
$e^{-1} = \phi f^{-1}$ where 
$\phi:  \calk_A(\cale \oplus \H_A) 
\rightarrow \calk_A(\cale \oplus \calf \oplus \H_A)$
and $f$ the corner embedding, see the left most square of the diagram of lemma \ref{lemma10}.

\end{lemma}

\section{}

FRAGLICH W	OHLDEFINIERT

\begin{lemma}

Sei $C$ unital.

Sei $f:B \rightarrow C$ homomorphismus.

Es gibt homomorphismus abbildung
$$\phi:\calk_B(\cale \oplus \H_B)
\rightarrow \calk_C(\cale \otimes_f C \oplus 
\H_C)$$

$\phi( \theta_{\xi \oplus u,\eta \oplus v} )
= \theta_{\xi \otimes_f 1 \oplus f(u),
\eta \otimes_f 1 \oplus f(v)}$

wo, allgemein

$\theta_{\xi \oplus u,\eta \oplus v} (a \oplus b)
= (\xi \oplus u) (g(a) + v b)$

wo

$g:\cale \rightarrow B$

$B$-module abbildung 

\end{lemma}

\begin{proof}
$\theta_{\xi \oplus u,g \oplus v}
\theta_{\xi^2 \oplus u^2,g^2 \oplus v^2} (a \oplus b)
= \theta_{\xi \oplus u,g \oplus v}(\xi^2 \oplus
 u^2) (g^2(a) + v^2 b)$

$=\xi \oplus u (g(\xi^2) + u^2 v) (g^2(a) + v^2 b)$

$= \xi_{\xi \oplus u (g(\xi^2) + u^2 v), g^2 \oplus v^2}$

$ \rightarrow \phi:$

$= \xi_{\xi (g(\xi^2) + u^2 v) \otimes 1 \oplus f(u (g(\xi^2) + u^2 v)), g^2 \otimes 1 \oplus f(v^2)}$

$= \theta_{\xi  \otimes 1 \oplus f(u)  f(g(\xi^2) + u^2 v)), g^2 \otimes 1 \oplus f(v^2)} (a \otimes x \oplus b)$

$=\xi  \otimes 1 \oplus f(u)  f(g(\xi^2) + u^2 v)) ( f(g^2(a)) x + f(v^2) b)$

$\theta_{\xi \otimes 1 \oplus f(u),g \otimes 1\oplus f(v)}
\theta_{\xi^2 \otimes 1 \oplus f(u^2),g^2 \otimes 1 \oplus f(v^2)} (a \otimes x \oplus b)$

$= \theta_{\xi \otimes 1 \oplus f(u),g \otimes 1 \oplus f(v)}(\xi^2 \otimes 1 \oplus
 f(u^2)) (f(g^2(a))x + f(v^2) b)$

$= \xi \otimes 1 \oplus f(u)
(f(g(\xi^2))  + f(v) f(u^2)) (f(g^2(a))x + f(v^2) b)$

\end{proof}

allgemeiner:

$\cale \subseteq \calf$ 

man könne functionale $\phi$ auf $\cale$ 
auf functionale $\tilde \phi$ auf $\calf$  ausdehnen 

dann abbilding $\calk_C(\cale) \rightarrow \calk_C(\calf)$ 

$\theta_{\xi,\phi} \mapsto \theta_{\xi,\tilde \phi}$  

multiplikativ 

jedoch wohldef?

\section{}

We may simplify the second line of the above diagram of the last lemma further 
by removing the $\H_B$-terms without altering 
its associeted $GK^G$-term:

\begin{corollary}
If in the last lemma the $s_\pm$ are of the form 
$s_\pm = u_\pm \oplus 0_{\H_B} \square 1$ then one can remove the $\oplus \H_B$-terms in the second line of the diagram of the last lemma. 
That is $s_+ \nabla_{s_-} \pi = v_+ \nabla_{v_-}$ for 
$$\xymatrix{ 
D \ar[r]^{f} & \calk_D  ( \cale\otimes_\pi D \oplus \H_D )   \ar[r]
&
\call_D ( \cale  \otimes_\pi D \oplus \H_D  ) \square A
& A  \ar[l]^{v_\pm} 
}$$
\end{corollary}

\begin{proof}
There are two possibilities of proof, both of which yielding the same result.

At first we define 
$$t_\pm(a)= u_\pm(a) \otimes 1 \oplus 0 \square a$$ 
to get the extended double split exact sequence of the corollory. 
Then we apply it to lemma \ref{lemma10} for 
$\calf := \H_B \otimes_\pi D$,
and by observing that the second line of the diagram lemma \ref{lemma10} is the second line of the diagram of the last lemma we ar done.

Either we may use lemma \ref{lemma10} 
to $\calf := \H_B \otimes_\pi D$ to remove it,
or we may define $\psi$ in the last lemma, but with $\calf$ removed, in the same way on $s_-(A)$, and use lemma ... do define $\psi$
on $\calk(\cale \oplus \H_B) \square 0$.

\end{proof}

\section{}

\begin{lemma}		\label{lemma14}

If $\cale$ is a cofull right functional $B$-module and $\pi:B \rightarrow D$ is a ring homomorphism
mapped further to $\call_D(D)$ by left multiplication 
then
$$\calk_{B}(\cale) \otimes_\pi D
\subseteq \calk_D ( \cale \otimes_\pi D)$$

\end{lemma}

\begin{proof}

Set $\phi \otimes \psi (\xi \otimes \eta):=
\psi( \pi( \phi(\xi))(\eta))$ 
füor $ \phi \in \Theta_B(\cale), \psi \in \Theta_D(\calf)$

TRhen

$$\theta_{\xi b,\phi} \otimes 1
= \theta_{\xi \otimes \pi(b), \phi \otimes 1}$$

probe:

$$\xi b\phi(u) \otimes v
= \xi \otimes \pi(b) (\phi \otimes 1)(u \otimes v)$$
$$= \xi \otimes \pi(b) \cdot (\pi(\phi(u)) v)$$
$$= \xi  b \phi(u) \otimes  v$$

\end{proof}

\section{}

\begin{lemma}

complexificatin

complexification of external tensor product:



$$\calk_B( \cale \otimes \cale ) \otimes \calk_\C(\C) 
\cong \calk_{B \otimes B \otimes \C} (\cale \otimes \cale \otimes \C) 
\cong \calk_{(B \otimes \C) \otimes_\C ( B \otimes \C)} 
((\cale \otimes \C) \otimes_\C 
(\cale \otimes \C) )$$



Let $\F$ ang $\G$ be fields, or noth there (omitted), then 

Associativity
$$(A \otimes^\F B) \otimes^\G C =
A \otimes^\F (B \otimes^\G C)$$

If $B$ is a $\F$-vectorspace then
$$\F \otimes^\F B =  B$$
 
Geht
$$A \otimes^\F (B \otimes^\G C) =
( A  \otimes^\G C)\otimes^\F B $$
j, weil 
$$A \otimes^\F B = B \otimes^\F A$$
($A,B$ $\F$ und $\G$ vr, $C$ $\G$ vr) 

weiters
(schreibe $\G=\G \otimes^\G \G$)  


\end{lemma}

\begin{lemma}

If $\phi \circ W \in \Theta_A(\cale)$ for all
$\phi \in \Theta_A(\cale)$ then 
$W$ adjointable. 

--

$\calk_\F(\cale)=\call_\F(\cale)$ 
für finite dim vr $\cale$ 

\end{lemma}

\section{}

We may reformulate further by exchanging the orders of "making a direct sum" and "do matrix embedding" 

\begin{eqnarray*}
s_+^{t_+ e_+} E_+^{(1)} \oplus {s_-^{t_- e_-}} 
E_+^{(2)}  
\oplus u_-^{P} E_+^{(3)} - \big ( {{s_-^{t_+ e_+}} \oplus 
s_+^{t_- e_-} \oplus u_+^{P}} \big ) E_-  
	&=& 0 
\quad   \cdot P' \mbox{ and then \ref{lemma94}}  
\\
\oplus \rightarrow + ?
\end{eqnarray*}

\section{}

\begin{lemma}		\label{lemma1323}
There is an injectgion
$$\pi: \call_A(\cale 
) \otimes \call_B(\calf)
\rightarrow 
\call_{A \otimes B}(\cale \otimes \calf)  
: \pi(S \otimes T)  =  
S \otimes T
$$
which restricts to an isomorphism
$\calk_A(\cale 
) \otimes \calk_B(\calf)
\rightarrow 
\calk_{A \otimes B}(\cale \otimes \calf)  
$
\end{lemma}

\begin{proof}

Surj:

$X = \lim_i \sum_j S_{ij} \otimes T_{ij} \in \calk_{A \otimes B}$

$\xi \otimes \eta \; \phi(\xi_2) \otimes \psi(\eta_2)$

vmtl schwer zu klären stetigkeit von abb und umkehrabb ohne konkrete topology 

auch frage der stetigen fortsetzung 

-> vllt restriction auf compacte = bijectiv als axiom 

$\sum_i \xi_i \phi_i(x) \otimes \eta_i \psi_i(y)=0$
\end{proof}

\section{}

\begin{proof}


---  jedoch hängt der opertor von $b_j$ der  der approx ein heit ab 

Dies zeigt dass
der angegebene bildraum von $\sigma$ 
im defeinitionsberich von $\sigma$ liegt.

Um die umgekehrte menegninklusion zu bekommen,
machen wir die obige rechung von unten nach oben, indem wir nur 1 $T_n$ wählen, und dieses
$y_n=1$ setzen, 
und $a_n$ (und somit $\beta_x(a_n)$) ganz weglassen. 


As $\beta_g(p_i)$ is an apporximate unit of $B$
(brauche als definition ), $T_g=0$ for all $g$.

---
ginge auch:  gilt für alle $p_i \in B$ $\Rightarrow$ $T_g=0$ $\forall g$

It is obvious to see that the above diagram commutes, where $f$ is the obvious canoncial corner embedding.

brauche also, $\xi b= 0$ für alle $b$ dann
$\xi=0$ , also gewisse fullheit von $\cale$ 
 
brauche auch für $B$ 

->  so, nach bedingung \ref{bed1}, gibt es
$\xi \phi(a) \neq 0$ , wähle $b:= \phi(a)$
\end{proof}

\section{}

\begin{lemma}		\label{lemma41}
{\rm (i)}
Any double split 
exact sequence as in 
the 
first line of 
this diagram
can be completed to this diagram
$$\xymatrix{
B  \ar[r]^i \ar[d]
& M  \ar[r]^f 
\ar[d]^\phi
& A 
\ar@<.5ex>[l]^{s_\pm} 	
\ar[d]
\\
B  \ar[r]^j  & M \square A \ar[r]^g  & A 
\ar@<.5ex>[l]^{t_\pm}  
}
$$
such that
the first line ist the second line in GK,
that is, $s_+ \mu \Delta_{s_-}
= t_+ \mu \Delta_{t_-}$.

{\rm (iii)}
The
$G$-action on $M_2(M \square A)$
is of the form 
$\theta \square \delta$.

\end{lemma}

\begin{proof}
Same proof as in \cite{bgenerators},
except: The reference to \cite[gk]{corollary 4.9.(i)} (invariance of ideal
$M_2(j(B))$) is now by assumption of 
definition ...

\end{proof}


\section{}

eher weg: 

\begin{corollary}			\label{cor74}
Let $(M_2(A_i),\theta_i)$ be two $G$-algebras as in lemma \ref{lemma61} ($i=1,2$). Let $\phi:A_1 \rightarrow A_2$ be a non-equivariant $*$-homomorphism.
Then $\phi \otimes 1_{M_2}$ is $G$-equivariant
if and only if it is $G$-equivariant on the lower left corner space 
 $(A_1,\gamma_1)$.
\end{corollary}

\begin{proof}
By lemma \ref{lemma61}.(i), confer also (v).
\end{proof}

\section{}

to clarify, if $W$ $X_\F \subseteq X$ invariant lässt 

\section{}

\begin{lemma}
The Green-Julg isomorphism $S$ and its inverse isomorphism $T$ are functorial, that is, 
for a morphisms $Y : \F \rightarrow A$ and $Z:A \rightarrow B$ one has 
\begin{eqnarray*}
&& S_B (Y Z) = S_A(Y) (Z \rtimes 1)  \\
&& S(\xymatrix{\F \ar[r]^{Y} & A \ar[r]^{Z} 
& B } ) =
\xymatrix{ \F \ar[r]^M & \F \rtimes G \ar[r]^{Y \rtimes \idd} & A \rtimes_\alpha G} 
\end{eqnarray*}
\end{lemma}

\begin{lemma}

$$\xymatrix{ 
KK^G(\C,A \otimes M_n) \ar[r]   \ar[d] 
 & KK(\C,(A\otimes M_n) \rtimes G)   \ar[d] 
 \\
KK^G(\C,A) \ar[r] 
 & KK(\C,A \rtimes G)	}
$$

funktorialität von GJ, von $S$ und $T$
\end{lemma}

\begin{lemma}
wenn $p \in K(A \rtimes G)$ von der form $P (e^{-1} \rtimes 1)$, $e:A \rightarrow A \otimes M_n$ 

---

allgeminer: wenn $p$ von der form $P (m \rtimes 1_G)$ wo $P$ projektion in 
$KK(\C,B \rtimes G)$ und 
$m:B \rightarrow A$ morphism in $GK$, und $B$
beliebig 

\end{lemma}

$$\xymatrix{ 
KK^G(\C,A \otimes M_n) \ar[r]   \ar[d] 
 & KK(\C,(A\otimes M_n) \rtimes G)   \ar[d] 
\ar[r]		&  
 KK(\C, (A \rtimes G) \otimes M_n)
 \\
KK^G(\C,A) \ar[r] 
 & KK(\C,A \rtimes G)	}
$$

wenn $p$ in rechter oberer ecke, dann es gibt $q$ in linker oberer ecke sodass $S(q)=p$. 
Damit $S(q e^{-1}) = p (e \rtimes 1)^{-1}$

Let $G$ be a discrete group. 

Für $\ell^2(G)$ braucht man wahrscheinlich 
kornereinbettungen mit summen $\H_A$ mit $A$-kopien cardinalität bis zur kardinalittä der gruppe $G$

\begin{lemma}   \label{lemma1111}

There is a ring homomorphism
$$\sigma: (A \otimes \K) \rtimes_{ \alpha \otimes \Ad(U)} G \rightarrow
 (A \otimes \K) \rtimes_{\alpha \otimes \triv} G$$

by
$$\sigma ((a \otimes k) \rtimes g)
= (a \otimes k U_{g}) \rtimes g$$

\end{lemma}

\section{}

\begin{lemma}
$\Delta_s (f-g) = \phi \Delta_t E^{-1} - \psi \Delta_s E^{-1}$ 
\end{lemma}

\begin{proof}

\end{proof}


\begin{lemma}

$\Phi(1 - f s) = \Phi(\Delta_s j)$
\end{lemma}

\section{}

\begin{definition}

Let $D$ be a ring.

There is an additive  functor $\tau_D: GK^G \rightarrow
GK^G$ defined by $\tau_D(A)= A \otimes D
= A \otimes_\F D$
and $\tau_D(f: A \rightarrow B)= f \otimes 1$.

\end{definition}

\section{}

\begin{lemma}		

If $\Theta_D(D)= \{1\} \equiv D^+$, then
$\calk_D(D)\cong D$.
\end{lemma}
 
\begin{proof}

$\phi:D \rightarrow \calk_D(D)$
$\phi(d)x =dx = x \phi_1(x)$

Axiom: $d x=0$ for all $x \in D$ then $d=0$
\end{proof}

\section{}

vmtl brauche hier wieder vektorraumalgebren, 
da kornereinbettung 

--> rückschau inverse halbgruppen

\section{}

These works also for every subring of $\call_A(\H_A)$ insteat of that ring (?).

Das kann man aber sowieso leicht ableiten durch composition einer einbettung-

\section{}

\begin{corollary}[{\cite{bgenerators}[corollary 14.2]}]  
					\label{cor142}
Consider the diagram (\ref{bz2}).
Make its `negative' diagram
where we exchange $t_-$ and $t_+$ and
transform the $M_2$-action under coordinate flip.
Then its associated element in $GK$ is 
the negative, that is, 
$$- t_+ \mu \Delta_{t_-} e^{-1} =
t_- \mu \Delta_{t_+} e^{-1}$$

\end{corollary}

\begin{proof}


Considering a sum as in the last lemma we have
$$(t_+ \oplus t_-) \mu \Delta_{t_- \oplus t_+}
= (t_- \oplus t_+)  \Delta_{t_- \oplus t_+} =0$$
by the 
rotation homotopy
$V_s ((t_- \oplus t_+) \otimes 1_\C) V_{-s}$,
where 
$$V_s= 1 \otimes \Big (\begin{matrix} \cos s & \sin s \\ - \sin s &  \cos s \end{matrix}
\Big ) \oplus 1 \otimes  1 \in 
\call_B(\calk) \otimes M_2(\C) \oplus \tilde A 
\otimes \C \cong
(\call_B( \cale^2) \oplus \tilde A) \otimes \C$$  
for $s \in [0,\pi/2]$
and where we define the $G$-action
on 
$M_2(X[0,1])$ 
for $X=\call_B(\cale \oplus \cale)$ by $(\theta^{V_s} 
)_{s \in [0,\pi/2]}$ 
so that we can apply lemma \ref{lemma123}.
%
%

Compute that the homotopy runs really in $M_2(\call_B(\cale)) \square_{t_ \oplus t_+} 
A = \image {(t_- \oplus t_+)} +  M_2(\calk_B(\cale))$ as $t_-(a)- t_+(a) \in \calk_B(\cale)$ for all $a \in A$.
\end{proof}

\section{}

\begin{lemma}			\label{lemma85}  

Let $A$ be unital. Let the first line, an extended double plsit exact sequence, and the plain matrix embedding $g$ (i.e. without occurence of $\cale$)  
of the 
following diagram be given. 
Thereby, the cardinalities of the index sets of the direct sums $\H_A$ and $\H_B$ may differ.

$$\xymatrix{
B \ar[r] 
\ar[d]   &
\calk_B(\cdots) \ar[r] 
\ar[d] 
& \call_B( \cale \oplus \H_B) \square  A 
\ar[d]^\phi  &
 A \ar[l]^{s_\pm \oplus 0 \square}  \ar[d]^g    \\
B \ar[r] 
& \calk_B(\cdots) \ar[r] 
& \call_B
\Big (  \H_A \otimes_{s_-} \cale
 \oplus \H_A \otimes_{s_+} \cale
\oplus \H_B 
\Big ) \square  \calk_A(\H_A) 
&
 \calk_A(\H_A) \ar[l]^{t_\pm }  
}$$

Then we complete these data to the above diagram such that 
$$g {-1} ((s_+ \oplus 0) \square 1) 
\nabla_{(s_- \oplus 0) \square 1} 
= t_+ \nabla_{t_-}$$  

\end{lemma}

\begin{proof}

We may assume that $s_\pm(\idd)=\idd$, otherwise cut $\cale$ down to $s_-(1) \cale$ 
by lemma \ref{lemma10}. 
We at first ignore any $G$-action. 

We use the 
functional $B$-module isomorphism $p_-: A \otimes_{s_-} \cale \rightarrow \cale$ defined by 
$p_-(a \otimes \xi)= s_-(a)(\xi)$ and $p_-  {-1} (\xi)= 1 \otimes \xi$, and an analogous one for $s_+$. 
 
Let $\pi: H_\cale \oplus H_\cale  \oplus \H_B \rightarrow \H_A \otimes_{s_-} \cale
 \oplus \H_A \otimes_{s_+} \cale  \oplus \H_B$
be the canonical isomorphism induced by $p_\pm$,
and $F$ be the flip operator defined by 
$F(S \oplus T \oplus Y)= T \oplus S \oplus Y$ on the domain of $\pi$, where $H_\cale := \oplus_{i \in I} \cale$. 

Set
$$\phi(T \square a)= \pi \circ (T \oplus 0 \oplus 0 \oplus 0 \oplus \cdots) \circ \pi^{-1} \square g(a)$$

That means, $T$ operates on the right hand side only at the first summand $A \otimes_{s_-} \cale$ of $H_\cale \cong \H_A \otimes_{s_-} \cale$
and on $\H_B$. 

We set
\begin{eqnarray*}
t_-(T) &=& T \otimes 1_\cale \oplus 0_{ \H_A \otimes_{s_+} \cale
}  \oplus 0_{\H_B} 
\square T \\
t_+(T) &=& \pi \circ F \circ \pi {-1} \circ ( 0_{ \H_A \otimes_{s_-} \cale
} \oplus  T \otimes 1_\cale  \oplus 0_{\H_B} ) \circ \pi \circ F^{-1} \circ \pi^{-1} \square T  
\end{eqnarray*}


If the $M_2$-action of the first line of the diagram is $\Ad (S \oplus R , T \oplus R) \square (\alpha \otimes 1)$
then on the second line it is
$$\Ad (V \otimes S \oplus V \otimes S \oplus R,
V \otimes T \oplus V \otimes T \oplus R)
 \square \Ad(V)$$
where $S,T$ act on $\cale$, $R$ on $\H_B$ and $V$ on $\H_A$.

We declare now $\pi$ to be $G$-equivariant by pulling back the now given $G$-action on the range of $\pi$ to its domain. 

Since $(\H_A,V)= (\H_A, \alpha \oplus W)$ by definition \ref{def21}, the 
map $\phi$ is seen to be $G$-equivariant. 

By lemma \ref{lemma272}, applied to the first 
line of the above diagram, we have 
$S_g S_{g -1} 
-  S_g T_{g -1} \in \calk_A(\cale)$. 

Hence, for $T \in \calk_A(\H_A)$ we get, 
on $X:=\H_A \otimes_{s_-} \cale$,  
$$
(T \otimes 1) 
( V_g V_{g  -1} \otimes S_g S_{g -1} 
-  V_g V_{g  -1} \otimes S_g T_{g -1}  )
\subseteq 
(T \otimes 1) (1_{\H_A}  \otimes \calk_A(\cale))
 \subseteq \calk_A(X)$$

Conseuqently, by lemma \ref{lemma272} again, 
the $M_2$-action of the second line of the above diagram is valid. 
 Finalize the proof by checking 
lemma   \ref{lemma21}. 

\end{proof}

\section{}

-> und dann als multiplikationsoperatorne 

cofull -> sind nicht multiplikationsoperatoren

\section{}

\begin{proof}
A simple computation shows that $T_g z S_{g^{-1}}$ is in $\call_A(\cale)$ and in $\calk_A(\cale)$ for $z$ compact. 

This is actually just the $G$-invariance condition observed in Def \ref{defcompact} for 
$\calk_A((\cale,S),(\cale,T))$. 
\end{proof}

\section{}

We may assume that $s_\pm(\idd)=\idd$, otherwise cut $\cale$ down to $s_-(1) \cale$ 
by lemma \ref{lemma10}.  ---> muss nicht yei

\section{}

\begin{proof}
Suposse that $\nu(z)=0$. Then 
$z P$ is a level-0 morphism for some $GK^G$-invertible 
ring homomoprhism $P$ with right-inverse $Q$.
Hence, $\nu(zP) = \nu(z) (P\rtimes 1)=0$ by 
the functoriality of the Baum-Connes map. By  injectivity of the Baum-Connes map on $L_0 GK^G$ we get $zP=0$, and so  
$z PQ=z=0$. 
 %
%
%
\end{proof}

\section{}

\begin{lemma}
BC ist rechts funktoiral:

$\lim GK^G(C_0(Y) ,A) \rightarrow \lim GK^G(C_0(Y) ,B): x \mapsto x z$

$z \in GK^G(A,B)$

$\mu_G(xz)= \mu_G(x) j^G(z)$

\end{lemma}

\begin{lemma}
BC map injectiv $\Leftrightarrow$
injectiv auf level-0

\end{lemma}

\section{}

\begin{proof}

One does this as in lemma... 
If one does not use cofullness of $B$, then we just 
We just mention the following modification:

By lemma \ref{lemma61}.(iv), $M_2(M)$
embedds equivariantly into $\call_M(M \oplus M,{\rm Ad}(\alpha \oplus \gamma))$, confer the formula in lemma
\ref{lemma51}. 
By assumptions on a split exact sequence,
$\gamma$ leaves $B$ invariant, and so restricts 
to a $G$-action on $B$. 
Otherwise, by cofullness: $\gamma(b_1 b_2)= \gamma(b_1) \alpha(b_2) \in B$ for $b_1,b_2 \in B$.

\end{proof}

\section{}

\begin{proof}
Write $\theta$ as in lemma \ref{lemma61}.
The corner action $\alpha$ leaves $j(B)$ invariant. 
Hence, by the formulas of lemma \ref{lemma61}.(i) we see that $\gamma, \beta,\delta$ leave 
a product $ab \in j(B)$ ($a,b \in j(B)$) invariant. By quadratik 
we are done.   
\end{proof}

\section{}

Nun ist $f g \kappa$ homotop zu
$m \psi \kappa$ durch rotation, wie zu zeigen war.

We do not need to bother wether the last module is cofull. Comming from the third line, where it is cofull, so it is cofull $\cale \oplus A$, so $\cale$ is itself cofull.

\section{}

We have a 
$B$-module functional isomorphism
$(X,l)$ defined by 
\begin{eqnarray*}
&&  X:A \otimes_\pi B \rightarrow B_0:= 
\sum 
\pi(A) B  \subseteq B
: X(a\otimes b) = 
\pi(a) b 				\\
&&  l: \Theta_B(A \otimes_\pi B) 
\rightarrow \Theta_B(B_0) : l(m_y \otimes m_x) = m_{x \pi(y)}
\end{eqnarray*} 
%
where $m_x$ denotes left multiplication operator ($x \in B,y \in A$) 
and $\Theta_B(B_0)$ is defined to be the image of $l$, into a (now functional) $B$-submodule $B_0$ of $B$, see also lemma \ref{lemma17}.

\section{}

sthet schon in der einleitung :::

If $\ring G$ 
is the collection of all 
separable $G$-equivariant $C^*$-algebras, 
the modules $\cale$ in definition 
just the countably generated Hilbert $A$-modules 
(or, equivalently, just th countable infinite 
direct sums $\cale=(\oplus_{i \in \N} A,S)$) 
and the compact operators are defined 
to be the norm-closure of the algebraically defined compact operators as above, 
then $GK^G$-theory is just $KK^G$-theory. 

\fi

\end{document}